\documentclass[12pt,a4paper]{article}
\usepackage{t1enc}
\usepackage[latin1]{inputenc}
\usepackage[english]{babel}
\usepackage{amssymb}
\usepackage{amsmath}
\usepackage{graphics}

\usepackage{amsthm, amssymb}
\usepackage{amsfonts}
\usepackage{epsfig,multicol}   

\newtheorem{theorem}[subsubsection]{Theorem}
\newtheorem{proposition}[subsubsection]{Proposition}
\newtheorem{lemma}[subsubsection]{Lemma}
\newtheorem{corollary}[subsubsection]{Corollary}
\newtheorem{definition}[subsubsection]{Definition}
\newtheorem{example}[subsubsection]{Example}
\newtheorem{remark}[subsubsection]{Remark}
  
\newtheorem{diagram}[subsubsection]{Fig.}

\newtheorem{formula}[subsubsection]{}
\newtheorem{question}[subsubsection]{Question}
 
\newfont{\gothic} { ygoth scaled \magstep{1.5}}
\newcommand{\notetoself}[1]{\marginpar{\tiny #1}}
\newcommand{\5}{\vskip 5pt}

\def\<{\langle}
\def\>{\rangle}  
\renewcommand\P{{$\cal P \hspace {4pt}$}}
\newcommand\A{\mbox{\gothic A}\hspace {6pt}}
\newcommand\T{\mbox{\gothic T}\hspace {6pt}}

\newcommand\F{{\cal F}}
\newcommand{\comment}[1]{}

\begin{document}

\def\hpic #1 #2 {\mbox{$\begin{array}[c]{l} \epsfig{file=#1,height=#2}
\end{array}$}}
 
\def\vpic #1 #2 {\mbox{$\begin{array}[c]{l} \epsfig{file=FIGURES/#1,width=#2}
\end{array}$}}
 \def\vjpic #1 #2 {\mbox{ \epsfig{file=#1,height=#2}}}
 \def\uplus {\vjpic{shadedcup} {0.12in} }
 
\title{Quadratic Tangles in Planar Algebras}
\author{Vaughan F.R. Jones
\thanks{Supported in part by NSF Grant DMS 0856316, the Marsden fund UOA520,
 and the Swiss National Science Foundation.}
} \maketitle 

\begin{abstract}In planar algebras, we show how to project certain simple ``quadratic''  tangles onto
the linear space spanned by ``linear'' and ``constant'' tangles. We obtain some corollaries about
the principal graphs and annular structure of subfactors.
\end{abstract}

\section{Introduction}

A planar algebra \P  consists of vector spaces $P_n$ together with
multilinear operations between them indexed by planar tangles-large discs
with internal (``input'') discs all connected up by non-intersecting 
curves called strings. Thus a planar algebra may be thought of as made
up from generators $R_i \in P_n$ to which linear combinations of
planar tangles may be appplied to obtain all elements of \P.

It was shown in \cite{P20} that a certain kind of planar algebra, called 
a subfactor planar algebra, is equivalent to the standard invariant
of an extremal finite index subfactor. We define this notion carefully
in section \ref{pa}, after which the term ``planar algebra'' will
mean subfactor planar algebra.
 
The simplest tangles are ones without input discs which can be called
\emph{constant} tangles. They are the analogue of the identity in
an ordinary associative algebra  and supply
a quotient of the Temperley-Lieb (TL) algebra in any unital planar algebra.
Tangles of the next level of complexity are the annular tangles
first appearing in \cite{Jaff}.  In the terminology of this paper they should be called
\emph{linear} tangles but the term linear is somewhat loaded so we will
tend to call them annular. They form a category and planar
algebras can be decomposed as a module
over the corresponding algebroid. This was done in \cite{J21}. The outcome
is a sequence of numbers which we will call the "annular multiplicity sequence"  $(a_n)$
which are the difference between the dimension of $P_n$ and the dimension 
of the image of the annular algebroid applied to $P_k$ for $k<n$. These numbers
are easily obtained from the dimensions of the $P_n$ by a change of variable
in the generating function-see \cite{J21}. For a subfactor algebra we define the
"supertransitivity" to be largest $n$ for which $a_k=0$ for $1\leq k\leq n$. The reason
for the terminology is explained in section \ref{superexcess} (see also section 4 of \cite{GJ}).  

In the present paper we begin the much more difficult task of studying
tangles with two input discs which we call ``quadratic tangles''. They
are not closed under any natural operation and one problem is to
uncover their mathematical structure. 
It is easy enough to list all quadratic tangles.However there is no guarantee that the
elements of \P obtained from
different tangles will be linearly independent. Indeed even for 
the TL tangles this may not be so and it is precisely this 
linear dependence that gives rise to the discrete spectrum of
the index for subfactors and the existence of $SU(2)$ TQFT's (\cite{J3},\cite{TL}).
Linear dependence of labelled tangles is most easily approached if
the tangle below gives a positive definite sesquilinear form on $P_n$ for 
each $n$. This positivity is one of the axioms of a subfactor planar algebra.
(The precise meaning of these pictures will be explained in the next
section, in the meantime one may just think of the $a$'s and $b$'s as tensors
with indices on the distinguished boundary points and joining boundary points by
a string signifies contraction of the corresponding indices.)

\begin{diagram}\label{ip}{The inner product tangle on $P_n$ (for $n=3$)\\}
$\displaystyle \langle R,S\rangle = $ \vpic {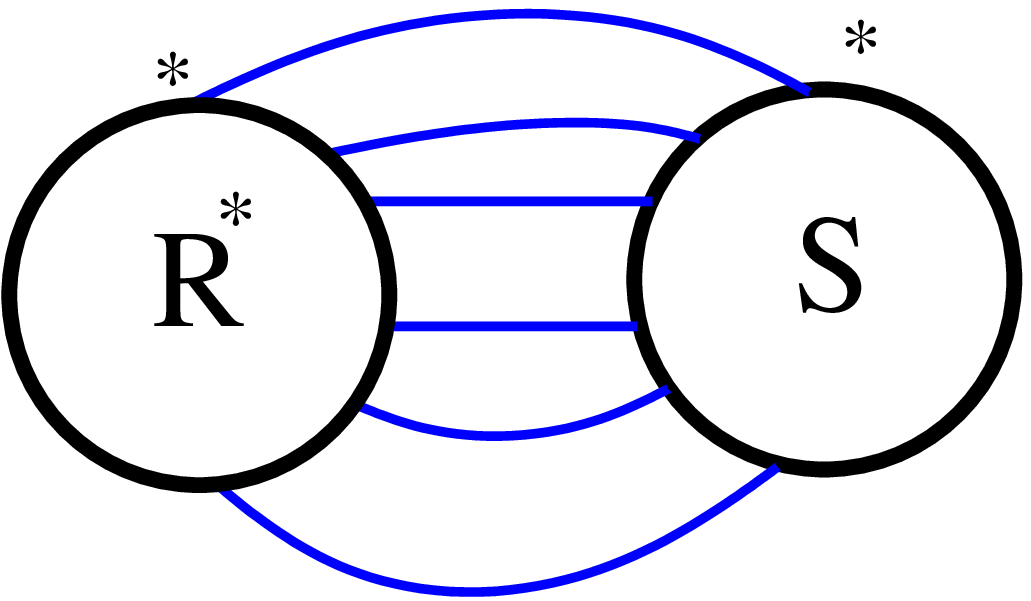} {2in}
\end{diagram}
\5
The dimension of the vector space spanned by a set of labelled tangles
is then given by the rank of the matrix of inner products. 
In this way we were able to use the powerful results of \cite{GL}
to obtain the dimensions of the relevant modules over the
annular category in \cite{J21},\cite{JR} .

Another frequently encountered tangle is the multiplication tangle:

\begin{diagram}\label{multiplication}
\vpic{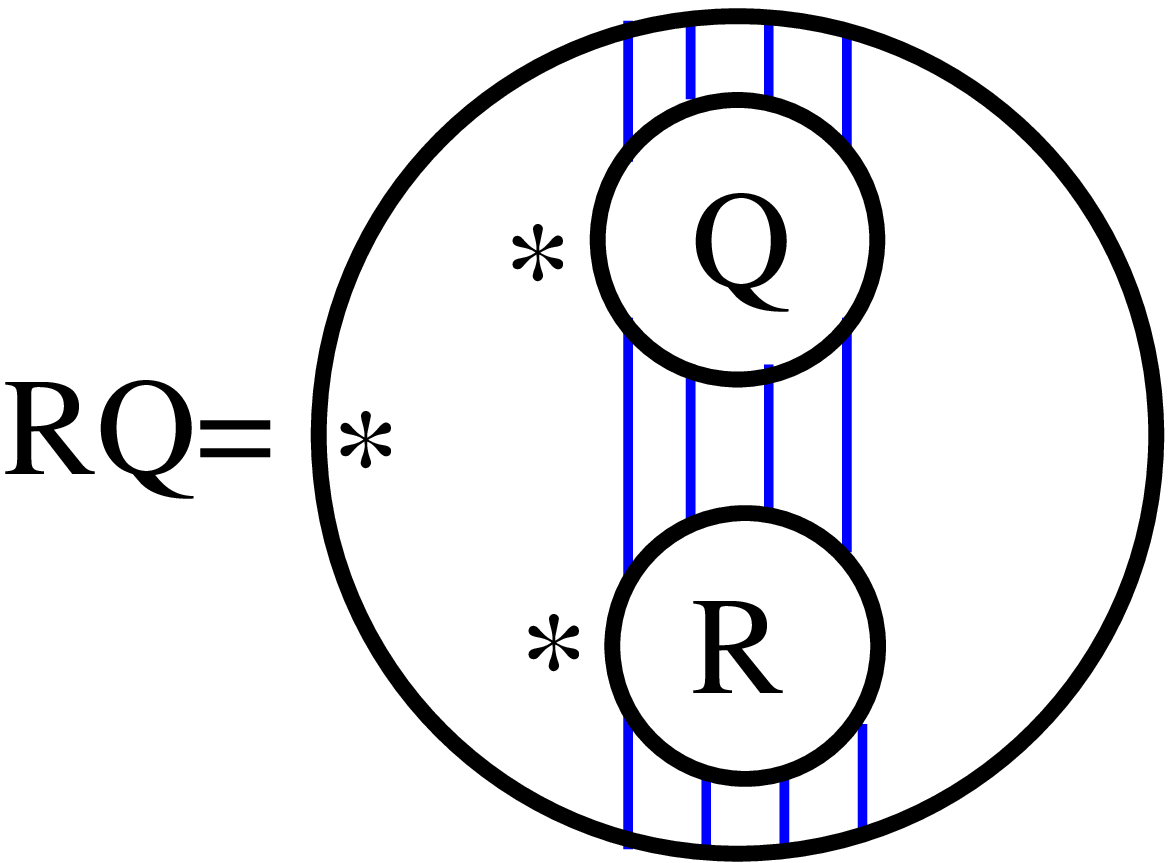} {2in}
\end{diagram}

This tangle is particularly well understood. It turns each of the $P_n$ ($n=4$ in the
diagram) into
an associative algebra which was the main ingredient in the first (obscure) 
appearance of subfactor PA's in \cite{J3}. 

In this paper we take a next step in the study of quadratic tangles by investigating
the tangle $R\circ Q$ (and its rotations):\\
\begin{diagram} \label{rcircq}
\vpic{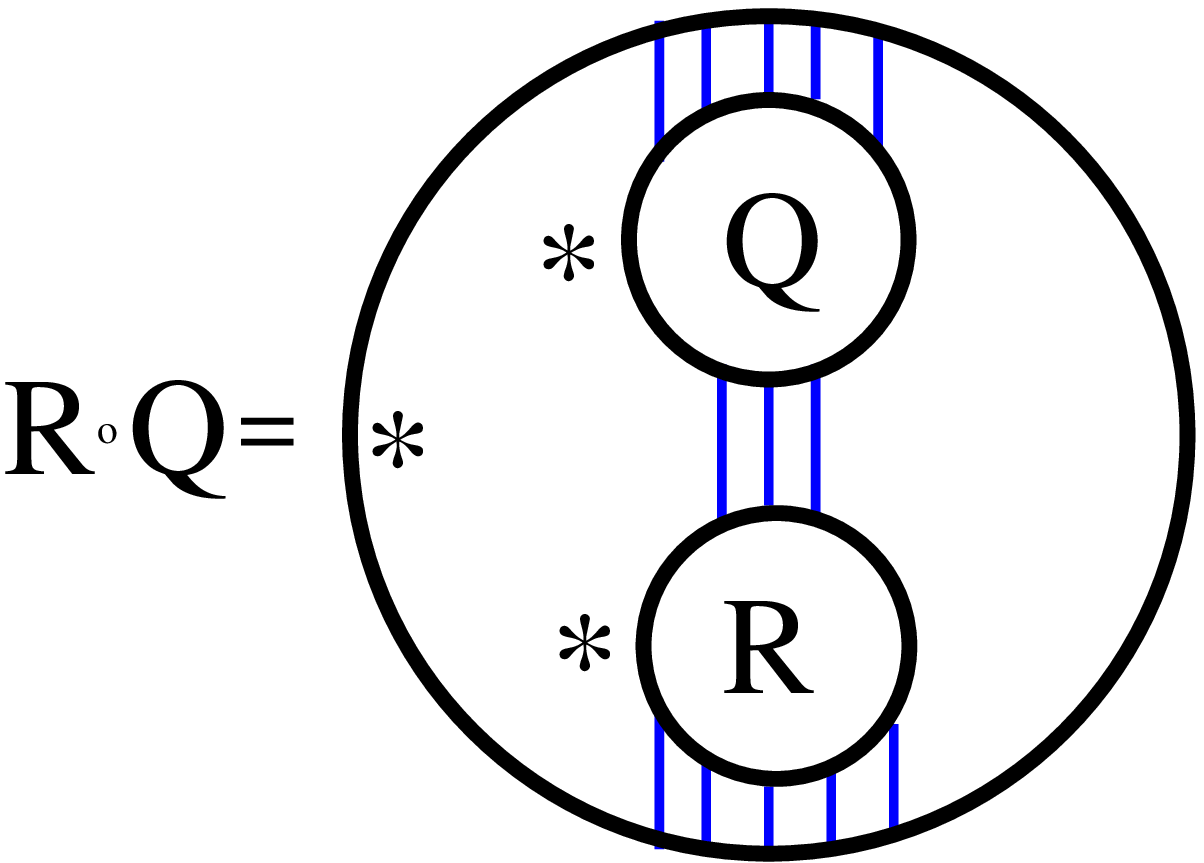} {2in}
\end{diagram}
This tangle maps $P_n\otimes P_n$ to $P_{n+1}$ (again $n=4$ in the diagram). Its image may then be compared with
the annular consequences of elements in $P_k$ for $k\leq n$. If \P is $(n-1)$-supertransitive
the only such annular consequences are the Temperley-Lieb tangles and the consequences
of the $R$'s and $Q$'s themselves. The main hard work of this paper is to calculate completely
explicitly the 
orthogonal projection of $R\circ Q$ and all of its rotations onto the space spanned by
these annular consequences when the subfactor planar algebra is $(n-1)$-supertransitive.
We will see that the only parameters for this calculation are the structure constants for
the algebra $P_n$ (in fact just the traces of cubic monomials) and the action of rotation.

By itself this projection is not very exciting and certainly does not justify the work but there
are ways to exploit this knowledge.
 For instance we will see that for certain powers of the rotation $\rho$ it is possible to evaluate
 $$\langle R\circ Q, \rho^k (T\circ S)\rangle$$ using planar manipulations. The inner products between
 the projections onto annular consequences are given by the annular theory which leads to 
non-trivial identities and inequalities.

The first version of this paper was available in preprint form as early as 2003. It was one
of the contributing factors to a project to classify all subfactors of index $\leq 5$ which
involves several people including Bigelow, Morrison, Peters and Snyder (see \cite{Pet},\cite{BMPS}). This 
project in turn influenced more recent versions of this paper and changed its emphasis 
somewhat. The author would like to gratefully acknowledge the input of all the people involved
in this classification project. In particular theorem \ref{evenst} was inspired by the
classification project but we include it here because we wanted to show that
the quadratic tangle technique applies more generally than in \cite{J31}.

\comment{
In section \ref{poincare} we apply our ideas under certain
hypotheses on the Poincar\'e series of \P.
The \emph{critical depth} of \P is the smallest $n$ for which
$dim P_{n,\pm}$ is greater than $\displaystyle{{1\over{n+1}} {2n \choose n}}$.
If  $\displaystyle{dim P_{n,\pm}={1\over{n+1}} {2n \choose n}+1}$
we may choose an essentially unique $R\in P_{n,+}$ and look at quadratic
tangles labelled by $R$.  Such tangles in
$P_{n+1,\pm}$ give lower bounds on the dimension of $P_{n+1,\pm}$
or alternatively knowledge of $dim P_{n+1,\pm}$ gives constraints
on the tetrahedral structure constants.

end comment
}
In section \ref{subfactors} we treat concrete examples,  obtaining  obstructions for 
graphs to be principal graphs of subfactors, and some new information
about the Haagerup and Haagerup-Asaeda subfactors. 

\5
Throughout this paper we will use diagrams with a small number of
strings to define tangles where the number of strings is
arbritrary. This is a huge savings in notation and we shall 
strive to use enough strings so that the general
situation is clear.

\section{Planar Algebras}\label{pas}

The definition of a planar algebra has evolved a bit since the
original one in \cite{J18} so we give a detailed definition
which is, we hope, the ultimate one, at least for shaded planar algebras.

\subsection{Planar Tangles.}
        
\begin{definition}\label{tangle}
Planar $k$-tangles. \\

A {\rm planar $k$-tangle } will consist of a smooth closed  disc
$D_0$ in $\mathbb C$ together with a finite
(possibly empty) set $\cal D$ of disjoint smooth discs in the interior of $D$. Each
disc $D \in \cal D$, and $D_0$, will have an even number
$2k_D\geq 0$ of marked points on its boundary
with $k=k_{D_0}$. Inside $D_0$, but outside the
interiors of the $D\in \cal D$, there is also a finite
set of disjoint smoothly embedded curves called
{\rm strings} which are either closed curves or
whose boundaries are marked points of $D_0$ and the $D\in \cal D$.
Each marked point is the boundary point of some
string, which meets the boundary of the
corresponding disc transversally. The connected components of
the complement of the strings in 
$\displaystyle{\stackrel{\circ}{D_0}\backslash \bigcup_{D\in \cal D}D}$ 
are called {\rm regions}. The connected components of the boundary of a
disc minus its marked points will be called the {\rm intervals} of that disc.
The regions of the tangle will be shaded
so that  regions whose boundaries meet
are shaded differently. 
The shading will be considered to extend to the intervals which
are part of the boundary of a region.

Finally, to each disc in a tangle there is a distinguished interval
on its boundary(which may be shaded or not). 

\end{definition}

\begin{remark}
{\rm Observe that smooth diffeomorphisms of $\mathbb C$
 act
on tangles in the obvious way-if $\phi$ is such a diffeomorphism, which
may be orientation reversing, and $I$ is the distinguished boundary interval 
of a disc $D$ in the tangle $T$, then $\phi(I)$ is the distinguished boundary 
interval of $\phi(D)$ in $\phi(T)$. }
\end {remark}

\begin{definition} The set of all planar $k$-tangles for $k>0$ will
be called ${\cal T}_k$. If the distinguished interval of $D_0$ for 
$T \in \cal T$ is unshaded, $T$ will be called {\rm positive} and
if it is shaded, $T$ will be called {\rm negative}. Thus ${\cal T}_k$ is
the disjoint union of sets of positive and negative tangles:
$\displaystyle {{\cal T}_k = {\cal T}_{k,+} \sqcup {\cal T}_{k,-}}$.  
\end{definition}

We will often have to draw pictures of tangles. To indicate the
distinguished interval on the boundary of a disc we will place a
*, near to that disc, in the region whose boundary contains the distinguished
interval. To avoid confusion we will always draw the discs of a tangle as
round circles and avoid round curves for closed strings.
An example of a positive $4$-tangle illustrating all the above ingredients
is given below.\\

\vpic{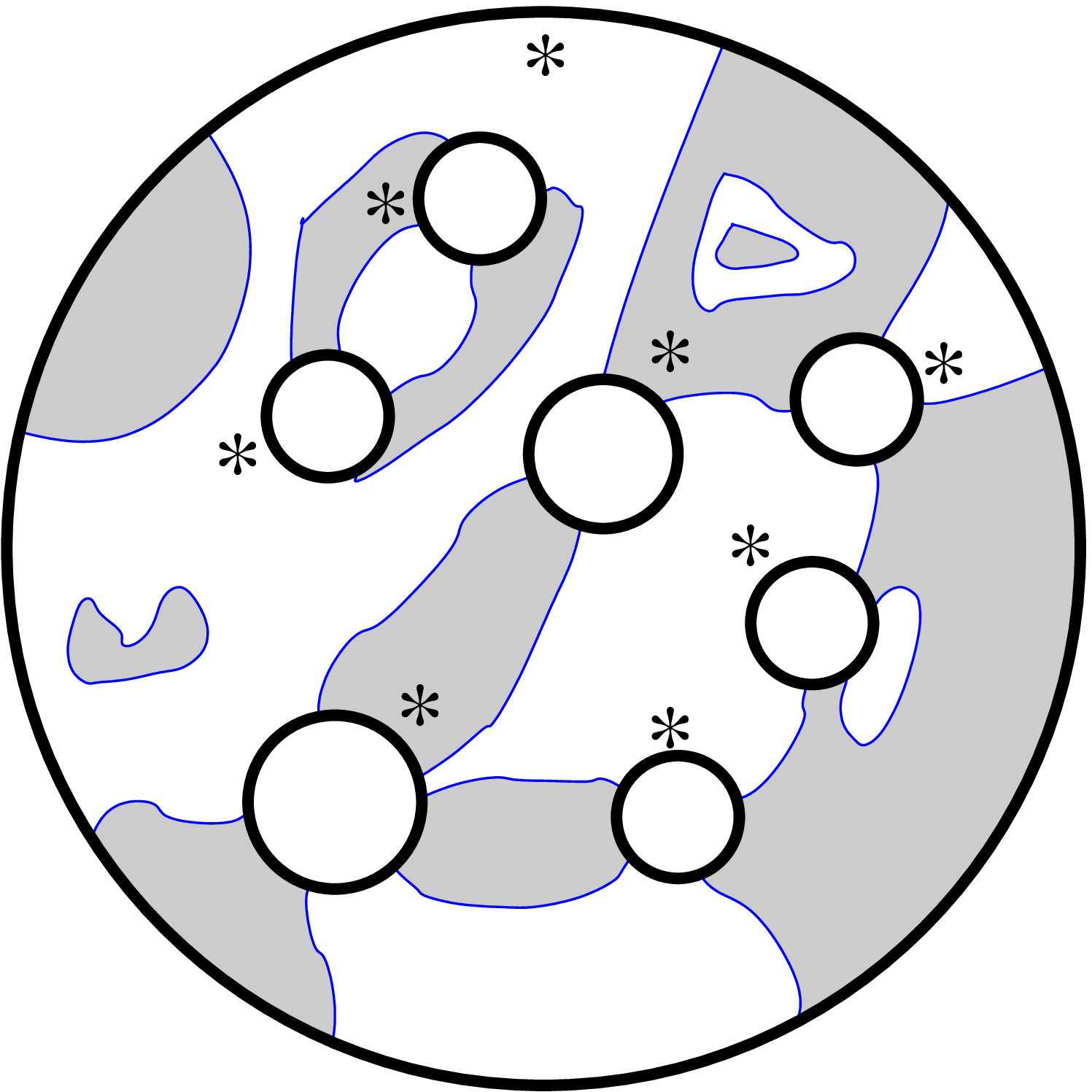} {2.5in}

\5\5

\subsection{Operadic Structure}
Planar tangles admit a partially defined ``gluing" or ``composition"
 operation which
we now define.

Suppose  we are given  planar $k$ and $k'$-tangles $T$ and $S$ respectively,
and a disk $D$ of $T$ which is identical to the $D_0$ of $S$ as a
smoothly embedded curve with points and shaded intervals.
If we obtain a $k$-tangle as the union of the strings and discs of $S$ and $T$,
excepting the disc $D$,
 we call
that  $k$-tangle $T\circ_D S$.   Otherwise the gluing is not defined.

We exhibit
the comoposition of tangles in the picture below.

\vpic{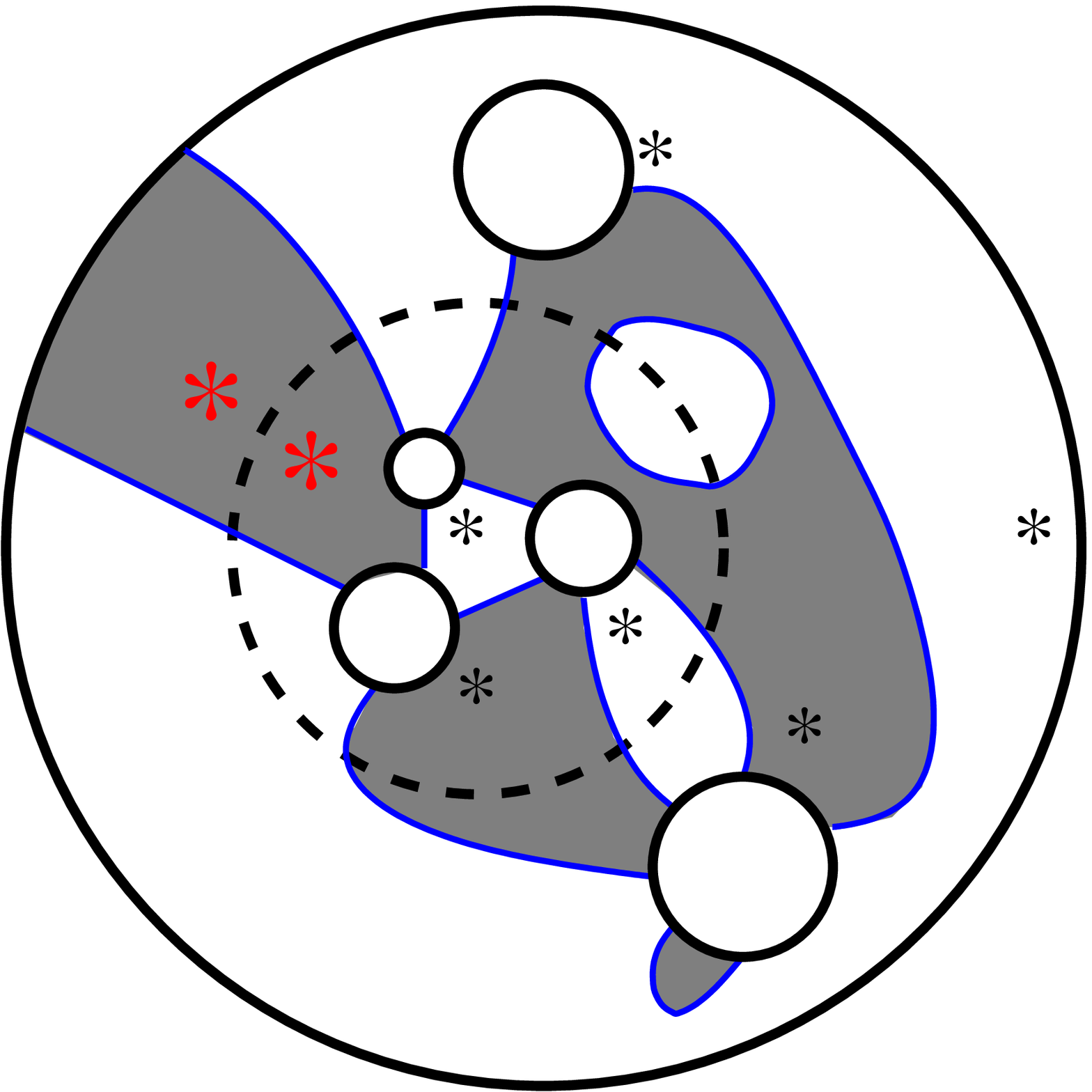} {2.5in}

Here the tangle $S$ is the dotted circle and everything inside it, the tangle $T$ is the
dotted circle and everything outside it and inside the outer circle which is its $D_0$. 
But the dotted circle is not part of a disc of $T\circ_D S$.
The large (red) *'s are in $S$ and $T$ but not in $T\circ_D S$.

\comment{ If we say that each disc in a tangle is ``coloured" by the number
of its marked boundary points and the shading of the
distinguished boundary interval, then it seems natural to consider the
object we have defined-the set of isotopy classes of 
planar tangles- to be a ``coloured
 operad". Such a notion has
properties in common with the notion of a category and that of 
an operad.

end comment
} 
\5\5 
\subsection{Planar Algebras}
Before giving the formal definition of a planar algebra we recall the notion of
the cartesian product of vector spaces over an index set $\cal I$,\quad
$\displaystyle \times_{i \in \cal I} V_i$. This is the set of functions $f$ from
$\cal I$ to the union of the $V_i$ with $f(i) \in V_i$. Vector space operations
are pointwise. Multilinearity is defined in the obvious way, and one convert 
multinearity into linearity in the usual way to obtain
$\displaystyle{ \otimes_{i \in \cal I} V_i}$, 
the tensor product indexed by $\cal I$. 
\vskip 8pt
\begin{definition}{Planar algebra.}\label{pa}\\
 A (shaded) planar algebra \P will be a family of
$\mathbb Z /2\mathbb Z$-graded vector spaces indexed by the set $\{\mathbb N \cup \{0\}\}$,
where $P_{k,\pm}$ will denote the $\pm$ graded space indexed by $k$.
To each planar $k$-tangle $T$ for $k\geq 0$ and ${\cal D}_T$ non empty, there will be a multilinear map 
$$Z_T : \times _{\scriptstyle {i\in {\cal D}_T}} P_i \rightarrow P_{D_0}$$
where $P_D$ is the vector space indexed by half the number of  marked boundary points of $i$
and graded by $+$ if the distinguished interval of $D$ is unshaded and
$-$ if it is shaded. 

The map $Z_T$ is called the "partition function" of $T$ and is subject to the following two requirements:

\vskip 5pt
(\romannumeral 1) (Isotopy invariance) If $\varphi$ is an \underline{orientation preserving}
diffeomorphism of $\mathbb C$ then
$$Z_T = Z_{\varphi (T)}$$ where the sets of internal discs of $T$ and $\varphi(T)$ 
are identified using $\varphi$.\\
\vskip 3pt
(\romannumeral 2)(Naturality)If $T\circ _D S$ exists and ${\cal D}_S$
is non-empty,  $$Z_{T  \circ _D S} = Z_T \circ_D Z_S$$
Where $D$ is an internal disc in $T$, and to define the right hand side
 of the equation, first observe that ${\cal D}_{T \circ_D S}$ is
the same as $({\cal D}_T - \{D\}) \cup {\cal D}_S$. Thus given a 
function $f$ on ${\cal D}_{T \circ_D S}$ to the appropriate vector spaces, we
may define a function $\tilde f$ on ${\cal D}_T$ by \\
\[\tilde f (E)= \left\{ \begin{array}{ll}
                        f(E) & \mbox{if $E \neq D$}\\
                        Z_S (f|_{{\cal D}_S}) & \mbox{if $E=D$}
                         \end{array}
                         \right. \]

Finally the formula $Z_T \circ_D Z_S(f) = Z_T(\tilde f)$ defines the right hand side. 

\end{definition}
 The unital property could no doubt be included as part of the
 above definition by careful consideration of the empty set but we prefer to 
 make it clear by doing it separately.
 
 \begin{definition} A shaded planar algebra will be called \underline{unital} if
 for every planar $(k,\pm)$-tangle S without internal discs there is an element $Z_S\in P_{k,\pm}$,
 depending on $S$ only up to isotopy,
 such that if $T\circ _D S$ exists then
 $$Z_{T  \circ _D S}(f) = Z_T(\tilde f)$$ where
 
 \[\tilde f (E)= \left\{ \begin{array}{ll}
                        f(E) & \mbox{if $E \neq D$}\\
                        Z_S  & \mbox{if $E=D$}
                         \end{array}
                         \right. \]

 \end{definition}
\begin{remark} It is sometimes convenient to work with just one of $P_{n,+}$
and $P_{n,-}$. We agree on the convention that $P_n$ will mean $P_{n,+}$.
\end{remark}

It follows from the axioms that in any unital planar algebra \P the linear span of the $Z_S$ as
$S$ runs through all planar tangles with no internal discs, forms a unital planar subalgebra.

Here is a basic example of a unital shaded planar algebra.

\begin{example}

The Temperley Lieb algebra $TL$.\\
\rm Let $\delta$ be an arbitrary element of the field $K$. We will define a planar algebra $TL(\delta)$
or just $TL$ for short.
We must first define the vector spaces $TL_{k,\pm}$. We let $TL_{k,+}$ be the vector space whose
basis is the set of all isotopy classes of connected $k$-tangles with no internal discs, and for which the
distinguished interval on the boundary is unshaded. (Here "connected" simply means as
a subset of $\mathbb C$, i.e. there are no closed strings.) It is well known
that the dimension of $TL_{k,+}$ is the Catalan number ${1 \over {k+1}} {2k \choose k} $.
Similary define  $TL_{k,-}$, requiring the disitinguished interval to be shaded.

The definition of the maps $Z_T$ is transparent: a multilinear map is defined on 
basis elements so given a $k$-tangle $T$
 it suffices to define a linear combination of tangles, given basis elements associated
to each  internal disc of $T$. But the basis elements are themselves tangles so they
can be glued into their corresponding discs as in the composition of tangles once they
have been isotoped into the correct position. The
resulting tangle will in general not be connected as some closed strings will appear
in the gluing. Just remove the closed strings, each time multiplying the tangle 
by $\delta$, until the tangle is connected.
This multiple $\tilde T$ of a basis element is the
result of applying $Z_T$ to the basis elements associated to the internal discs.

For the unital structure, if the tangle $S$ has no internal discs we put $Z_S= \tilde S$.

\end{example}

\begin{remark}In the example above there is a constant $\delta \in K$ with the 
property that the partition function of a tangle containing a closed contractible
string
is $\delta$ times that of the same tangle with the string removed.  We shall call such a planar algebra
a \underline {planar algebra with parameter $\delta$}.
\end{remark}

\subsection{Labelled Tangles}
If \P is a planar algebra and $T$ a planar tangle, and we are given elements
$x_D\in P_D$ for $D$'s in some subset $\cal S$ of the internal discs ${\cal D}_T$ of $T$ (with
$P_D$ as in \ref{pa}), we form the ``labelled tangle'' $T_{x_D}$ by
writing each $x_D$ in its $D$ and then forming the linear
map 
$$Z_{T_{x_D}}:\bigotimes _{\scriptstyle {D'\in {\cal D}_T\setminus {\cal S}}} P_{D'} \rightarrow P_{D_0}$$
in the obvious way.

Here is an example of a labelled tangle which defines a map from 
$P_{2,+}\otimes P_{3,-}$ to $P_{1,+}$.

\begin{diagram}

\vpic{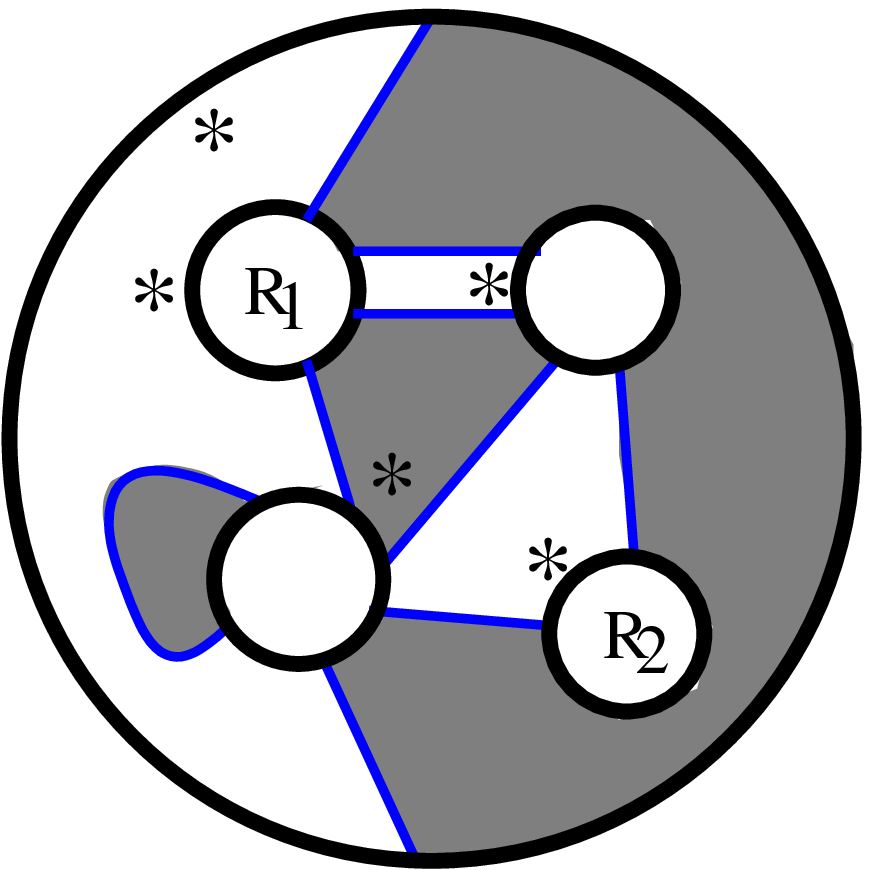} {3 in}
\end{diagram}

We will say that $T$ is fully labelled if $S=
{\cal D}_T$ in the above.

The special case where $dim P_{0,\pm} =1$ is common. In this case we will typically
leave out the external boundary disc from the diagram-see for example \ref{ip}.
We will also leave out the shading when it is defined by other knowledge. For
instance in \ref{ip}, since we know that $a$ and $b$ are in $P_{3,+}$ the shading
is defined by the distinguished intervals indicated with $*$'s.

Spherical invariance is most easily expressed in terms of labelled tangles.
We say that a planar algebra \P is \underline{spherical} if 
$dim P_{0,\pm} =1$ and the partition function of any fully labelled 0-tangle
is invariant under spherical isotopy. Note that spherical isotopies can
pass from tangles with the outside region shaded to ones where it is 
unshaded.

\subsection{Star Structure}

Another natural operation on planar tangles will be crucial in $C^*$-algebra
considerations. We assume the field is $\mathbb R$ or $\mathbb C$.
\comment{
\begin{definition} If $T$ is a planar $k$-tangle we define the {\rm conjugate}
tangle $\bar T$ by applying complex conjugation to $T$ itself. The distinguished
 intervals
on the discs of $\bar T$ are defined to be the
images under complex conjugation of the original ones.

\end{definition}

\begin{proposition} If $T$ and $S$ are isotopic tangles then so 

are $\bar T$ and $\bar S$. 

{\rm Proof. If $\phi$ is an orientation preserving diffeomorphism of $D$ with
$\phi (T)=S$ then the conjugate of $\phi$ by complex conjugation is 
orientation preserving and maps $\bar T$ to $\bar S$. $\square$}
\end{proposition}

Thus the conjugate operation on tangles passes to isotopy classes.
Since the compostion of two orientation reversing diffeomorphisms
preserves orientation, any orientation reversing diffeomorphism could be
used in place of complex conjugation to give the same operation on istotopy
classes of tangles.

\begin{definition}\label{planarstar}
A planar *-algebra will be a planar algebra \P over $\mathbb R$ or $\mathbb C$ with
a conjutate-linear involution $*$ on each $P_{n,\pm}$ such
that for any tangle $T$,

$$Z_{\overline T}(x_1^*,x_2^*,...,x_k^*)=Z_T(x_1,x_2,...,x_k)^*.$$
\end{definition}
 end comment
}
 \begin{definition} We will say that a planar algebra $P$ is a \emph{*-planar
algebra} if there each $P_{n,\pm}$ possesses a conjugate linear involution
* and If $\theta$ is an \emph{orientation reversing} diffeomorphism of $\mathbb C$, then $Z_{\theta(T)}(f\circ\theta^{-1})=Z_T(f^*)^*$. 
\end{definition}

The Temperley-Lieb planar algebra is a planar *-algebra if $\delta \in \mathbb R$
and the involution on each $P_{n,\pm}$ is the conjugate-linear extension
of complex conjugation acting on tangles.

\subsection{Subfactor planar algebras.}

\begin{definition}
A \underline{subfactor planar algebra} \P will be a spherical planar *-algebra
with $dim P_{n,\pm} <\infty$ for all $n$ and such that the inner product
defined by figure \ref{ip} is positive definite for all $n$ and grading $\pm$.
\end{definition}

It was shown in \cite{J18} that a subfactor planar algebra is the same thing
as the standard invariant of a finite index extremal subfactor of
a type II$_1$ factor.

Here are some well known facts concerning subfactor planar algebras.

\begin{description}
\item{(\romannumeral 1)} The parameter $\delta$ is $>0$. The spaces $P_{0,\pm}$ are
identified with $\mathbb C$ by identifying the empty tangle with $1\in \mathbb C$.

\item{(\romannumeral 2)} Algebra structures on $P_{n,\pm}$ are given by
the following tangle (with both choices of shading).
\5
\vpic{mult} {2.2in}

\item{(\romannumeral 3)} There is a pair of pointed bipartite
graphs, called the \underline{principal
graphs} $\Gamma,*$ and $\Gamma',*$ such that $P_{n,+}$ and $P_{n,-}$ have bases 
indexed by the loops of length $2n$ based at $*$ on $\Gamma$ and $\Gamma'$
respectively. The multiplication of these basis elements is easily defined
using the first half of the first loop and the second half of the second,
assuming the second half of the first is equal to the first half of the
second (otherwise the answer is zero).
\item{(\romannumeral 4)} The tangle below (with both choices
of shading) give vector space isomporphisms between $P_{n,+}$ and $P_{n,-}$ 
(for $n>0$) called the
``Fourier transforms''. 
\5
\begin{diagram}\label{fourierdiagram}
\vpic{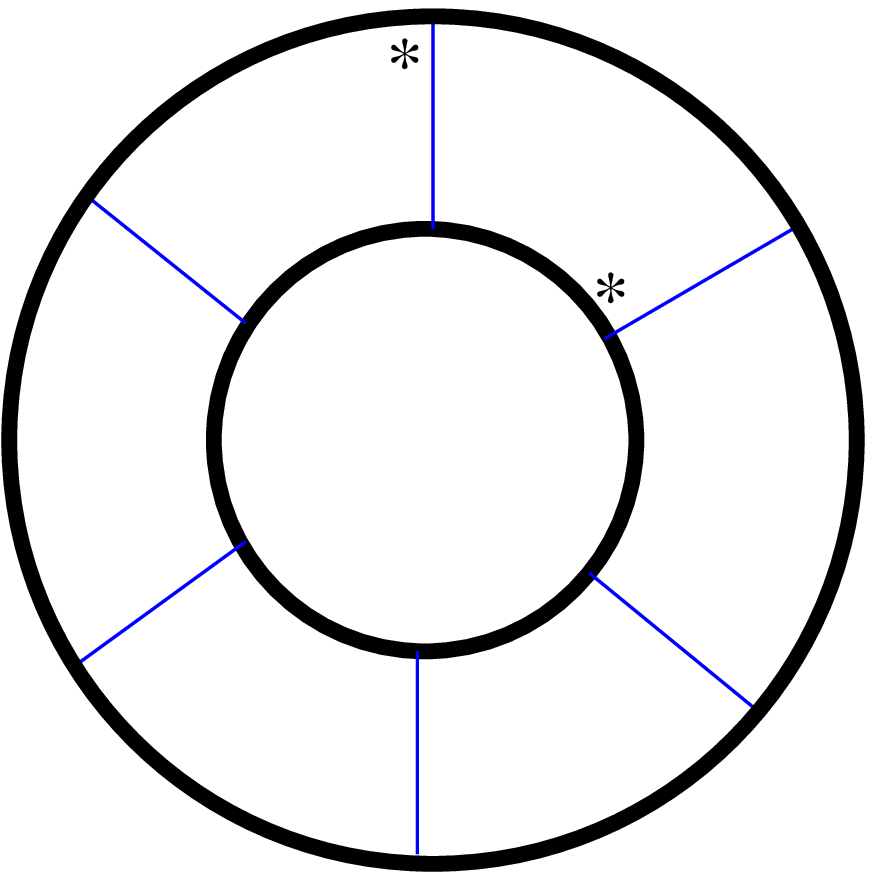} {1.5in}
\end{diagram}
\5
\comment{
\notetoself{changed the direction of Fourier-picture below
probably wrong}
end comment
}
In particular one may pull back the multiplication on $P_{n,-}$
to $P_{n,+}$. If $n$ is odd the two algebra structures are 
naturally *-anti-isomorphic using the $n$th. power of the
Fourier transform tangle $\cal F$. This allows us to identify
the minimal central projections of $P_{n,+}$ with those of 
$P_{n,-}$. If $n$ is even the $n$th. power of $\cal F$ gives anti-isomorphisms
from $P_{n,+}$ and $P_{n,-}$ to themselves defining a possibly non-trivial 
bijection between the central projections. From a more representation-theoretic
point of view these central projections correspond to bimodules and these maps
are the "contragredient" maps.

\comment {
 We record the tangle 
defining this multiplication for
which we use the wedge symbol:
 $$ R\wedge Q={\cal F}^{-1}({\cal F}(R){\cal F}(Q))$$

$R\wedge Q=$  \vpic{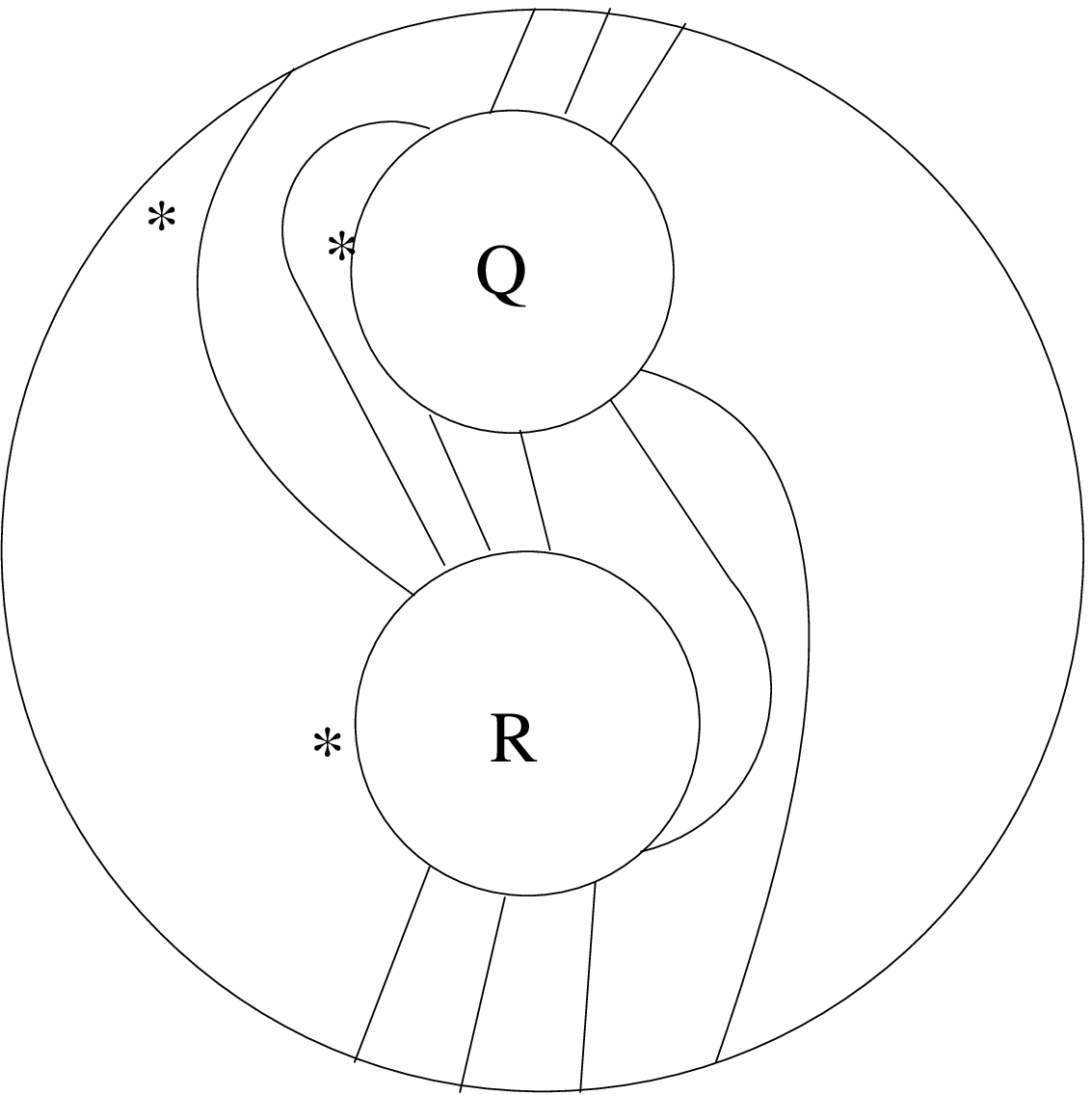} {2.2in}

end of comment   
}

\item{(\romannumeral 5)} The "square" of the Fourier transform 
acts on each $P_{n,\pm}$ and will be of vital importance for this paper.
We will call it the rotation $\rho$:\\
$\rho=$\vpic{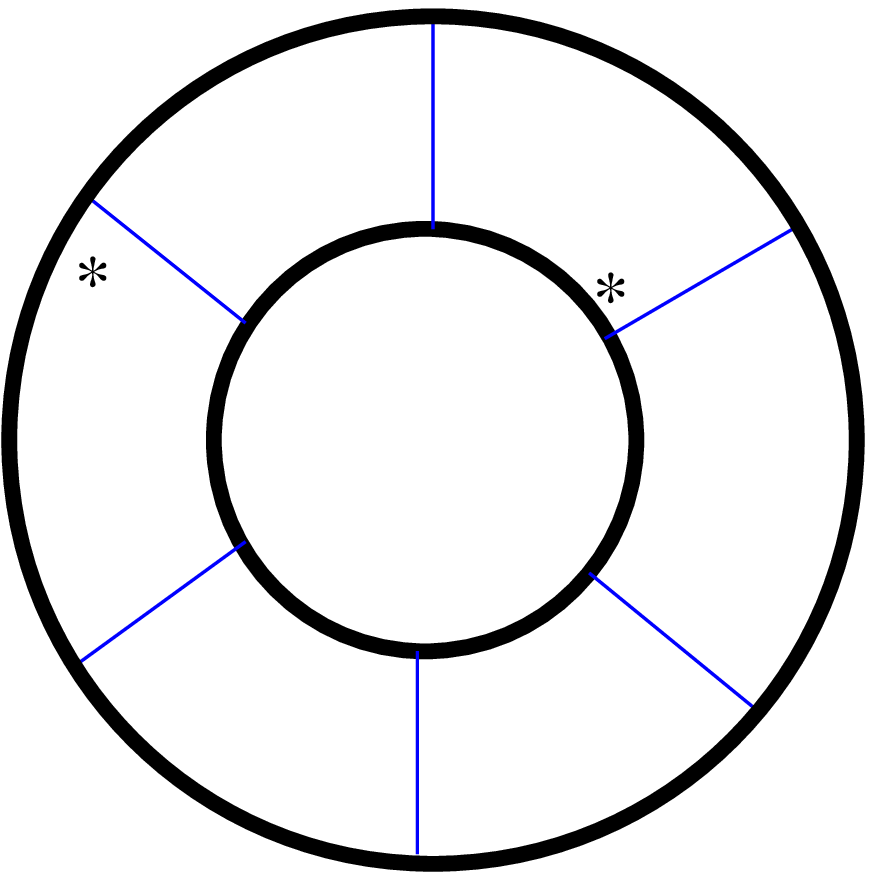} {1.5in}

\item{(\romannumeral 6)} $\Gamma$ is finite iff $\Gamma'$ is
in which case $\delta$ is the norm of the adjacency
matrix of $\Gamma$ and $\Gamma'$ and
the subfactor/planar algebra is said to
have \underline{finite depth}. 
 
\item{(\romannumeral 7)} There is a \underline{trace} $Tr$
on each $P_{n,\pm}$ defined by the following $0-tangle$:
$Tr(R)=$ \vpic{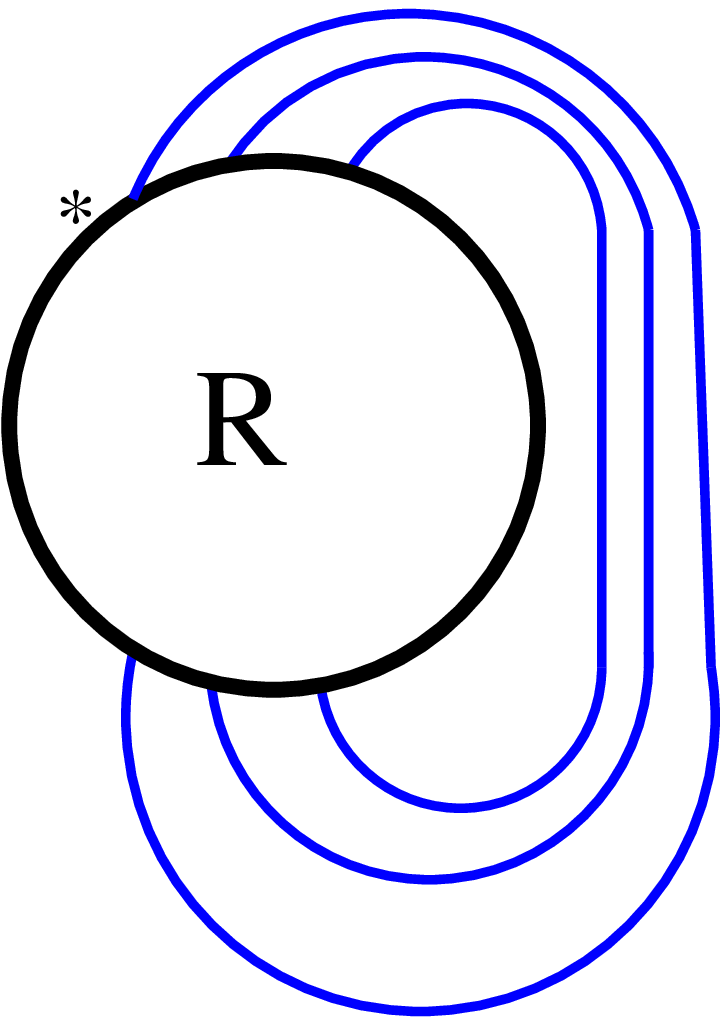} {1.1in}
\5
The innner product $\langle R,S\rangle$ on $P_n$ of figure \ref{ip}
is given alternatively by $Tr(R^*S)$. Our convention is that it is 
\underline{linear} in the second variable and \underline{antilinear}
in the first, as implied by \ref{ip}.

The algebra $P_{n,\pm}$ is semisimple over $\mathbb C$ and
its simple components are matrix algebras indexed by the
vertices of the principal graph at distance
$n$ from *. The trace $Tr$ is thus given by assigning 
a ``weight'' by which the usual matrix trace must be 
multiplied in each simple component. This multiple is
given by the Perron-Frobenius eigenvector (thought of 
as a function on the vertices) of the 
adjaceny matrix of the principal graph, normalised so that
the value at * is $1$. Because of the
bipartite structure care may be needed in computing this
eigenvector. To be sure, one defines $\Lambda$ to be
the (possibly non-square) bipartite adjacency matrix and constructs the
eigenvector for the adjacency matrix 
$\displaystyle \left( \begin{array}{cc}
0 & \Lambda \cr
\Lambda^t & 0 \cr
\end{array} \right)
$ as $\displaystyle \left( \begin{array}{c}
\Lambda v \cr
\delta v \cr
\end{array} \right)$
where $v$ is the Perron Frobenius eigenvector (of
eigenvalue $\displaystyle \delta^2$) of $\displaystyle {\Lambda^t \Lambda}$.

\item{(\romannumeral 8)}
There are ``partial traces'' or ``conditional expectations'' the
simplest of which is ${\cal E}$, the
map from $P_{n,\pm}$ to $P_{n-1,\pm}$ defined by the following
tangle:
\begin{diagram} \label{condexp}
\qquad
${\cal E}= \qquad \qquad $ \vpic{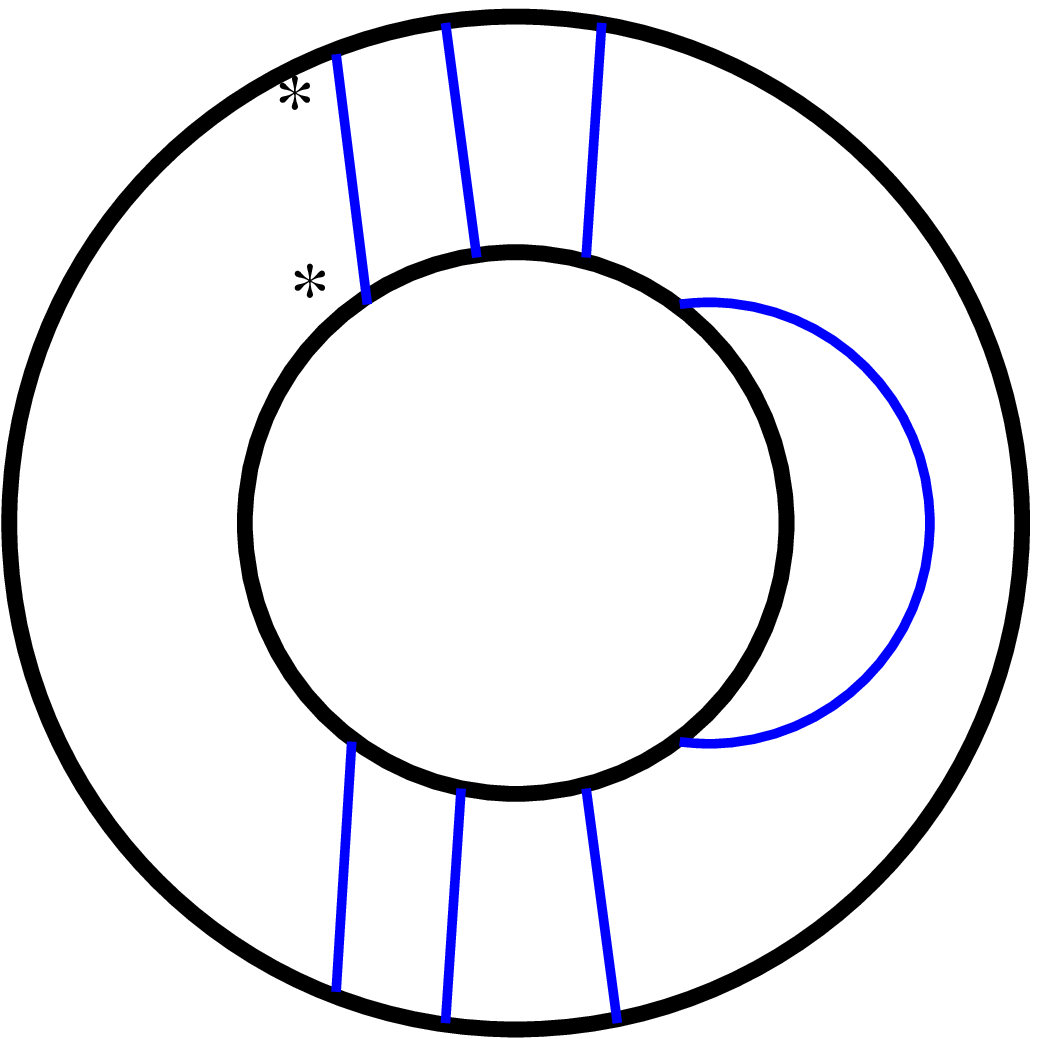} {1.3in}
\end{diagram}
It is obvious that 
$$Tr({\cal E}(x))=Tr(x).$$

\item{(\romannumeral 9)}  The TL diagrams span a subalgebra of $P_{n,+}$ and
$P_{n,-}$ which we will call TL for short.
 We will call $E_i$, for $1\leq i\leq n-1$ the element of $P_{n,\pm}$
defined by the tangle below:
\begin{diagram}\label{ei}
\quad \vpic{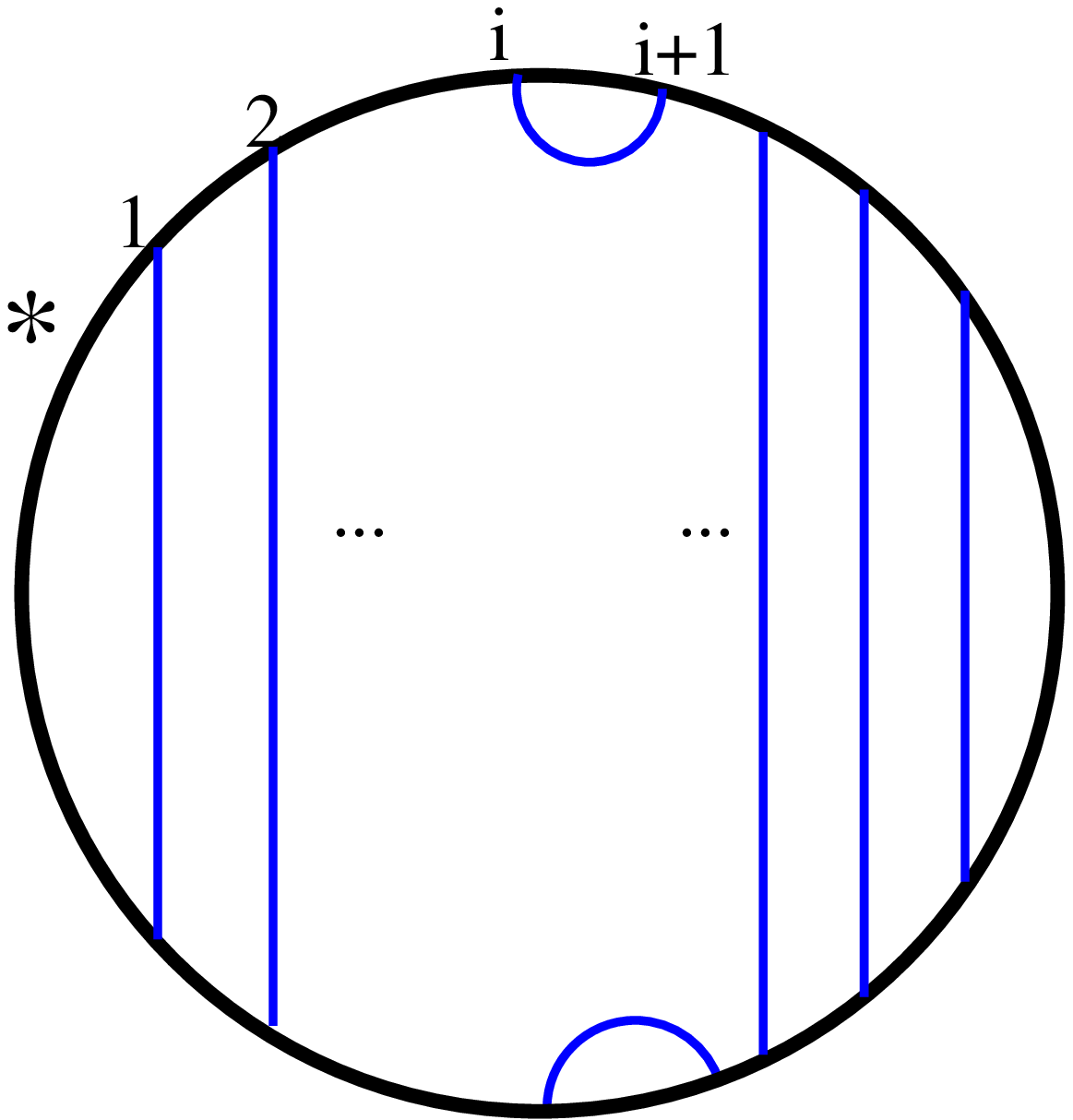} {1.5in}
\end{diagram}

The algebra generated by the $E_i$ is a 2-sided ideal
in  TL. Its identity is the JW projection $p_n$ which is
uniquely defined up to a scalar and the property
\begin{formula}\label{jwproperty}
\begin{align*} E_i p_n=&p_n E_i =0 & \forall i<n \cr
Tr(xE_n)=&Tr(x)& \hbox{   for  } x\in TL_n
\end{align*}
\end{formula}

 One has the formulae
\begin{formula}\label{morejw}

\begin{align*}
Tr(p_n)= & \quad [n+1]\cr
Tr(xE_n)= & \quad Tr(x) \hbox{   for  } x\in TL_n \cr
{\cal E}(p_n)= &\quad \frac{[n+1]}{[n]}p_{n-1} \cr
 p_{n+1}= & \quad p_n-\frac{[n]}{[n+1]}p_nE_np_n \qquad (\cite{Wn1})
\end{align*}

\end{formula}\label{jwtrace}
where $\delta=q+q^{-1}$ and $[r]$ is the quantum integer 
$\displaystyle {q^r-q^{-r} \over {q-q^{-1}}}$. Note $\displaystyle [r+1]=\delta [r] -[r-1]$.

We will also use a formula from \cite{J21}. Namely if $x$ is the TL$_n$ element \vpic{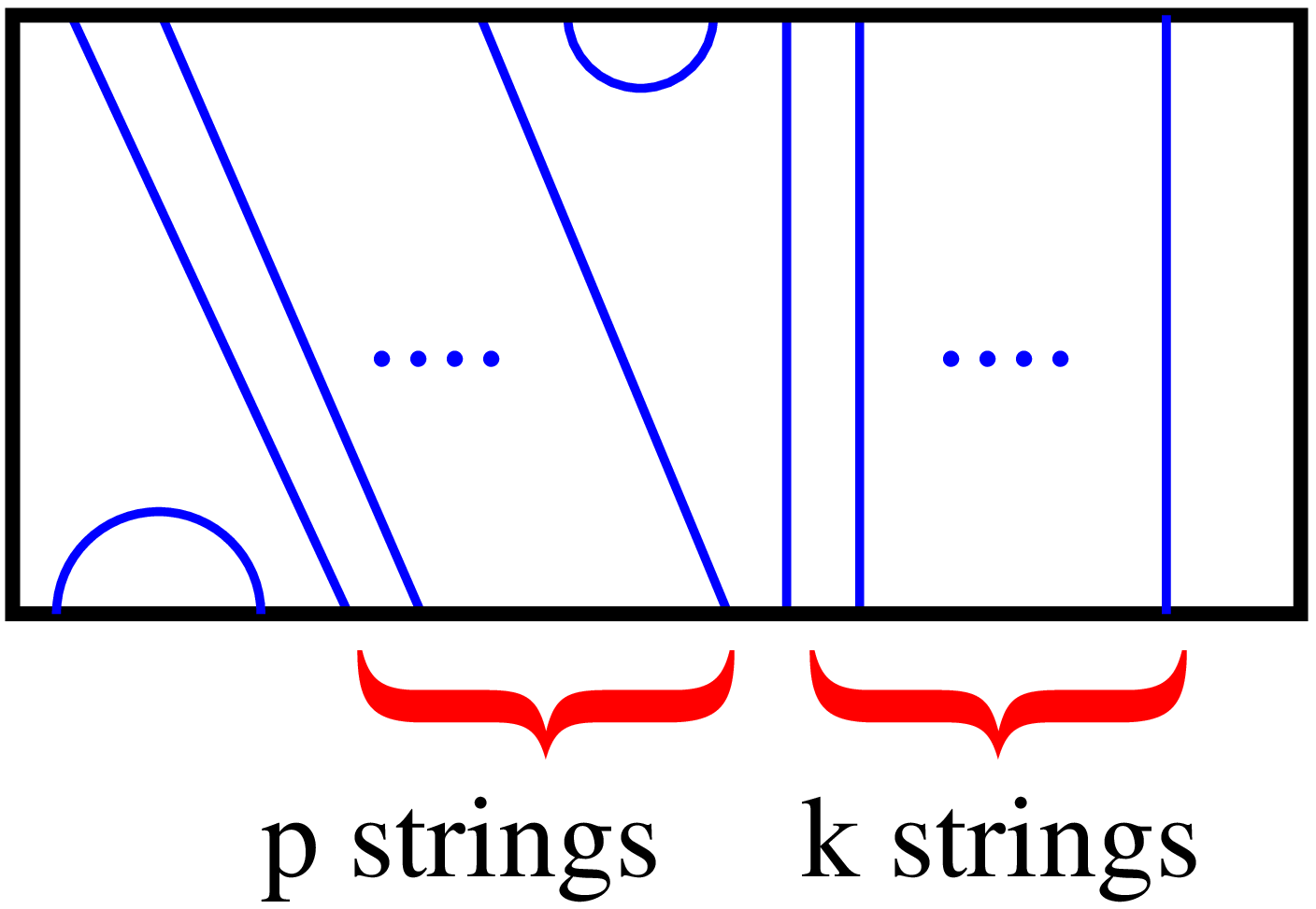} {1.4in}
then the coefficient of $x$ in $p_n$ is 
\begin{formula}\label{prodeis}$$(-1)^{p-1}\frac{[k+1]}{[n]}=(-1)^{p-1}\frac{[n-p-1]}{[n]}$$
\end{formula}
\end{description}

(\romannumeral 10) Notation: the Fourier transform gives a canonical identification of
$P_{n,+}$ with $P_{n,-}$ so one of the two is redundant data. It is convenient to set
$$P_n=P_{n,+}.$$

\emph{FOR THE REST OF THIS PAPER ``PLANAR ALGEBRA'' WILL MEAN ``SUBFACTOR PLANAR ALGEBRA''.}
\comment{\section{Triangular structure constants.}\label{tetras}

\comment{
We will calculate the value of the partition function of any labelled 0-tangle with 
exactly three input discs in terms of the algebra structures and traces on $P_{n,\pm}$. 
By the decomposition under the annular category it suffices to do the calculation
when all of the labels are lowest weight vectors of irreducible modules. 
end comment }
end comment}
\subsection{Lowest weight vectors.}
 In this paper a lowest weight vector $R$ will be an element of $P_n$ so that

\begin{description}
\item{(\romannumeral 1)} \quad $R^*=R.$
\item{(\romannumeral 2)}
\begin{diagram}\label{killed}
\qquad \vpic{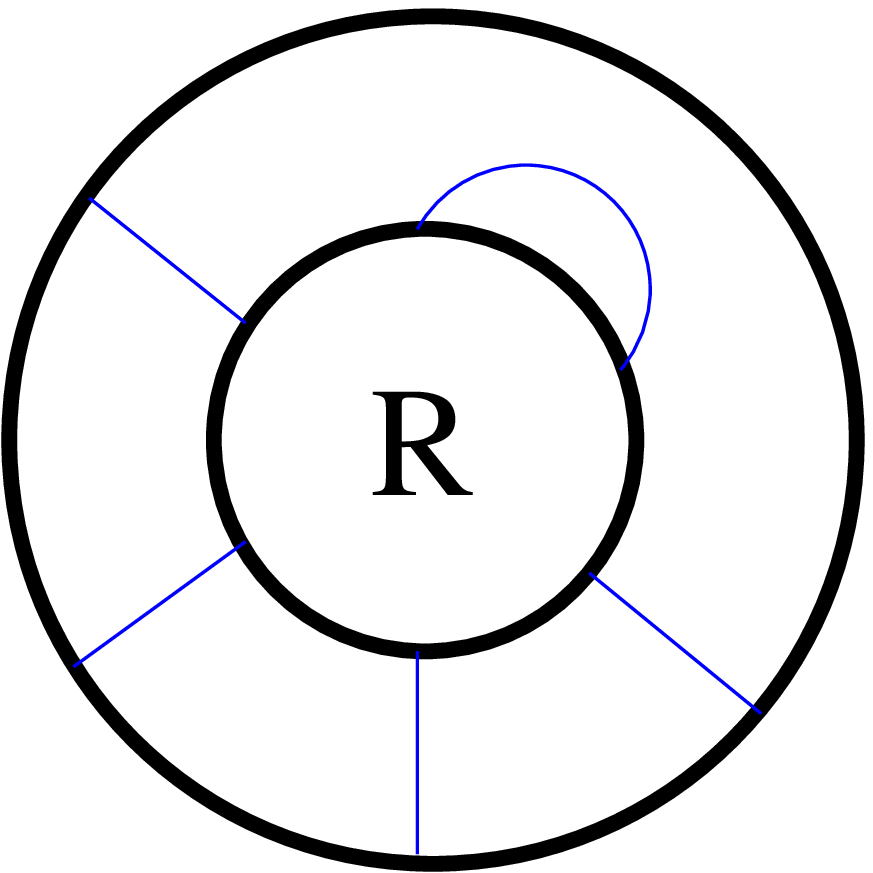} {1.5in} 
$=0$ \rm{\quad for any positions of *.}
\end{diagram}

\item{(\romannumeral 3)}
\begin{diagram}\label{rotated}
$\rho(R)=$
\qquad \vpic{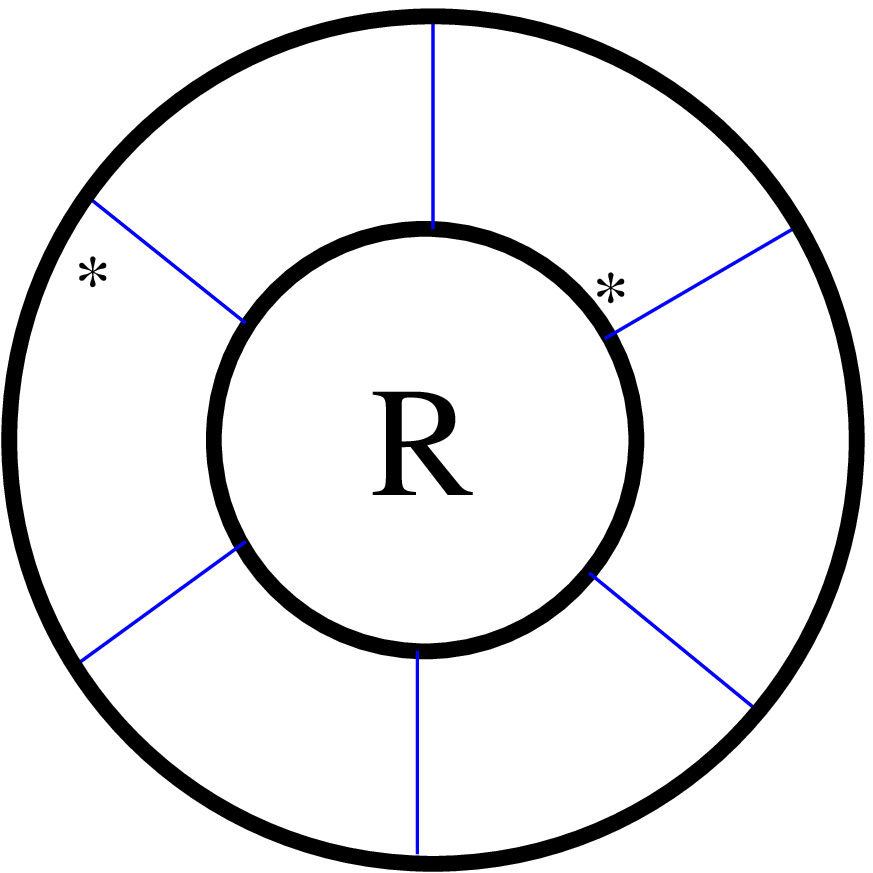} {1.3in} 
\rm{\quad$ =\omega R$\\ for some $n$th. root of unity $\omega$.}
\end{diagram}
The diagram can be confusing to apply so let
us say verbally what it means: ``if you see a tangle with an input disc $D$ 
labelled $R$, then the tangle is $\omega$ times the same tangle but with
the * of $D$ rotated counterclockwise by 2''.
 \end{description}
\5
As shown in \cite{J21}, any planar algebra may be decomposed into an orthogonal
direct sum of irreducible subspaces for the action of annular tangles, each irreducible
summand being generated by a lowest weight vector.

\section{Supertransitivity and Annular multiplicity.}\label{superexcess}
Transitivity of a group action on a set $X$ is measured by
the number of orbits on the Cartesian powers $X,X^2, X^3$,
etc. The smallest number of orbits is attained for the full symmetric
group $S_X$ and we say the group action is $k$-transitive if it has the same 
number of orbits on $X^k$ as does $S_X$. Transitivity can be further
quantified by the number of orbits that an $S_X$ orbit breaks into
on $X^k$ when the action fails to be $k-transitive$. 
There is a planar algebra \P associated to a group action on $X$
for which $dim P_k$ is the number of orbits on $X^k$. The planar
algebra contains a copy of the planar algebra for $S_X$ and the
action is $k$-transitive if $P_k$ is no bigger than the symmetric group
planar algebra. Thus the symmetric group planar algebra is universal in this
situation. 
We will call it the \emph{partition} planar algebra. As a union of finite dimensional
algebras it already appears in \cite{J16} and \cite{Martin}.
It depends on $X$ of course but only 
through \#$(X)$. Other planar algebras may be universal for other
situations. The Fuss-Catalan subalgebra of \cite{BJ} is 
such an example for subfactors which are not maximal. Another
 one, sometimes called the string
algebra, is universal for the Wassermann subfactors for representations
of compact groups. But the truly universal planar algebra in this regard is the TL
planar algebra (a quotient of)which is contained in any planar algebra.
Motivated by this discussion we will define a notion of 
\emph{supertransitivity} measured by how small the algebra is
compared to its TL subalgebra.

\subsection{The partition planar algebra.}\label{partalgebra}
In \cite{J18} we defined a planar algebra \P, called the ``spin model
planar algebra'' associated to a vector space $V$ of
dimension $k$ with a fixed basis numbered $1,2,...,k$. For $n>0$ 
$P_{n,\pm}$ is $\otimes^n V$ and $P_{0,+}=V$, $P_{0,-} =\mathbb C$.
The planar operad acts by representing an element of $\otimes^nV$
as a tensor with $n$ indices (with respect to the given basis).
The indices are associated with the shaded regions and summed
over internal shaded regions in a tangle. There is also a subtle
factor in the partition function coming from the curvature along
the strings. This factor is only necessary to make a closed string
count $\sqrt k$ independently of how it is shaded whereas without
this factor a closed string would count $k$ if the region inside
it is shaded and $1$ otherwise. For more details see \cite{J18}. Nothing
in the planar algebra structure differentiates between the basis
vectors so the symmetric group $S_k$ acts on $P$ by planar *-algebra
automorphisms.

\begin{definition} If $G$ is a group acting on the
set $\{1,2,...,k\}$ we define $P^G$ to be the fixed point sub-planar algebra of
the spin model planar algebra under the action of $G$.
The \underline{partition} planar algebra
${\cal C}=\{C_{n,\pm}|n=0,1,2,...\}$ is $P^{S_k}$.
\end{definition}
\begin{remark} \rm{ If $G$ acts transitively, passing to the fixed point algebra makes 
$dim C_{0,+}=dim C_{0,-}=1$ and spherical invariance of the partition
function is clear, as is positive definiteness of the inner product.
So $P^G$ is a subfactor planar algebra.
The subfactor it comes from is the ``group-subgroup'' subfactor- choose
an outer action of $G$ on a II$_1$ factor $M$ and consider the

subfactor $M^{G}\subseteq M^{H}$ where $H$ is the stabilizer of a point
in $\{1,2,...,k\}$  .}
\end{remark}

\begin{proposition} The action of $G$ is $r$-transitive iff $P^G_r=C_r$.
\end{proposition} 
\begin{proof}By definition the dimension of $P^G_{r,\pm}$ is the number of
orbits for the action of $G$ on $\{1,2,...,k\}^r$.
\end{proof}

\subsection{Supertransitivity}

\begin{definition} The planar algebra $P$ will be called \underline{$r$-supertransitive}
if $P_{r,\pm}=TL_{r,\pm}$.
\end{definition}

\begin{example} The $D_{2n}$ planar algebra with $\delta=2\cos \pi/(2n-2)$ is
$r$-supertransitive for $r<2n-2$ but not for $r=2n-2$.
\end{example}
 
The following question indicates just how little we know about subfactors.
Analogy with the Mathieu groups suggests that the answer could 
be interesting indeed.

\begin{question} For each $r>0$ is there a (subfactor) planar algebra which
is $r$-supertransitive but not equal to $TL$ (for $\delta >2$)? 
\end{question}
The question is even open for $r=4$ if we require that the planar algebra is
not $(r+1)$-supertransitive.
\begin{example} The partition planar algebra $\cal C$ is $3-$supertransitive
but not $4$-supertransitve.
\5
\rm{This is because there are just $5$ orbits of $S_X$ on $X^3$ but
$15$ on $X^4$ (at least for $\#(X)>3$).

}
\end{example}
\5

The record for supertransitivity at the time of writing is seven and is held by the
``extended Haagerup" subfactor constructed in \cite{BMPS}. The Asaeda-Haagerup subfactor
is $5$-supertransitive.
 
 \subsection{Annular multiplicity}
A subfactor planar algebra is always a direct sum of orthogonal irreducible modules
for the action of the annular category-see \cite{J21}. The irreducible representations of
the annular category are completely determined by their "lowest weight"  and
rotational eigenvalue. For a planar algebra, in practical terms this means
that each $P_n$ contains a canonical subspace $AC_n$ which is the image of
the annular category on $P_{n-1}$. Elements of $AC_n$ are called "annular consequences"
of elements in $P_k$, $k<n$.  
\begin{definition}\label{annularexcess}
If \P is a planar algebra the number $$a_n=\dim P_n - \dim AC_n$$ for $n>0$ is called the 
\emph{\underline{$nth$ annular multiplicity}}.\\
If the subfactor is $(n-1)$-supertransitive, $a_n$ is just called the \underline{multiplicity}.

\end{definition}
Obviously $k$-supertransitivity is equivalent to $a_r=0$ for $r\leq k$ but it is quite
possible for $a_r$ to be zero without $r$-supertransitivity. There is an explicit
formula for the generating function of the $a_n$ in terms of the generating function
for the dimensions of the planar algebra itself in \cite{J21}.
We will record the annular multiplicity as a sequence. For instance the multiplicity sequence
for the partition planar algebra begins  00010.... 
Obviously all the 0's up to the supertransitivity have no interest and we will be 
interested in situations where that sequence of zeros can have arbitrary length so
we will abbreviate the sequence by compressing the leading zeros to a *. Thus the annular multiplicity sequence for the partition planar
algebra would begin *10. Indeed our real applications to subfactors will
concern those with multiplicity sequence beginning *10 or *20. If we wish to give the leading zeros
as well we will use the notation $0^ka_{k+1}a_{k+2}$ for a sequence beginning with
$k$ zeros followed by $a_{k+1}a_{k+2}$ followed by anything at all.

We will allow this terminology to refer to either a subfactor or its planar algebra.

\subsection{Chirality.} If a planar algebra is $n-1$-supertransitive with $n$th. annulary multiplicity
equal to $k$, the rotation acts on the othogonal complement of TL in the $n-$box space.
This implies that planar algebras may have a ``handedness". Over the complex numbers
one may diagonalise the rotation on this orthogonal complement to obtain a family
of $n-$th. roots of unity. These numbers will be called the ``chirality" of
the planar algebra.

 \comment{To further quantify the transitivity of a group action one might say that
a group action on $X$ has \emph{$k$th. annular multiplicity} $p$ if there
are $p$ more orbits on $X^{k}$ than there are for the whole symmetric
group. (Thus the action would have $k$th. annular multiplicity  $0$ iff it is $k$-transitive.

\begin{definition}A planar algbera \P has  \emph{$k$-multiplicity} $p$ if
the codimension of $TL_{k,\pm}$ in $P_{k,\pm}$ is $p$.
\end{definition}

Thus $k$-supertransitive means that the $k$-multiplicity is zero. If $e>0$ we will say that
a subfactor simply has "multiplicity $e$" if it is $k$-supertransitive and has $(k+1)$-multiplicity
$e$.

end comment }
\subsection{Examples}
\begin{example}  The $D_{2n}$ planar algebra with $\delta=2\cos \pi/(2n-2)$
has annular multiplicity sequence beginning $0^{n}10$. It has multiplicity one. The principal
graphs are both (graphs with 2n vertices) \\
\vpic{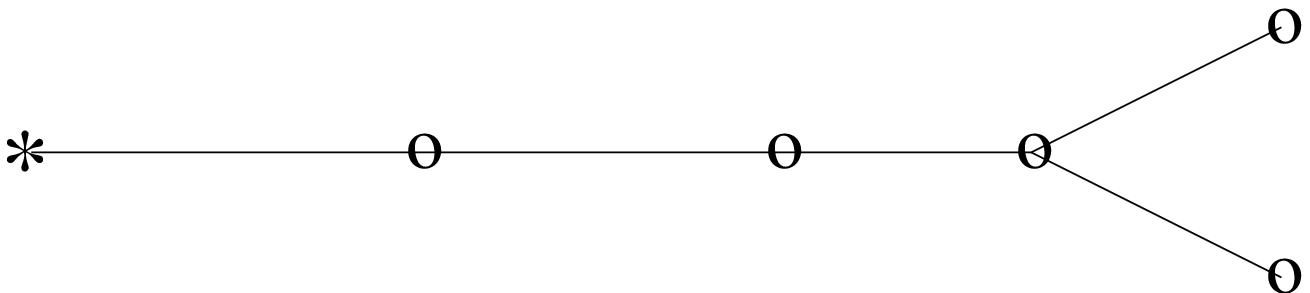} {2in}

The chirality of this example was calculated in \cite{J18}. It is $-1$.
\end{example}

\begin{example} The Haagerup subfactors of index $5+\sqrt 13 \over 2$ 
have multiplicity sequence $0^310$. They have multiplicity one and are 3-supertransitive.
\5
The principal graphs are \\
\vpic{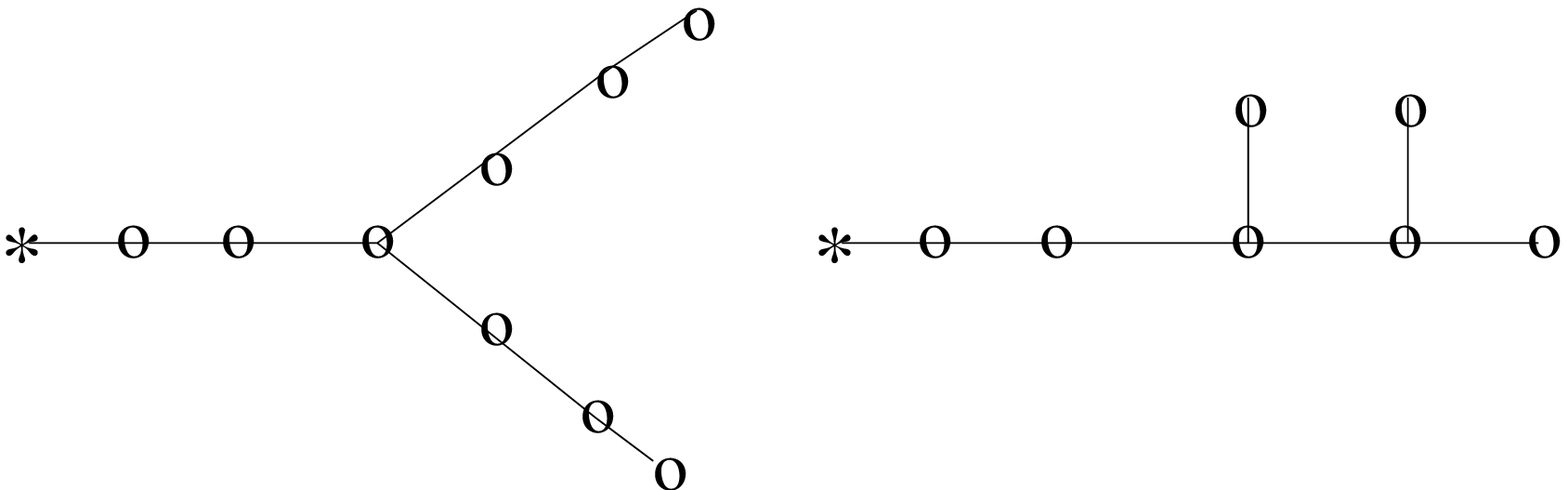} {3in}

\end{example}

Once a principal graph is known it is easy enough to calculate all the annular multiplicities. For
Haagerup the sequence begins $0^31010$. We will see that the chirality is $-1$.

\begin{example} \label{combpa}The partition planar algebra has multiplicity sequence $0^310$.

\end{example}

This  follows from the statement that the principal graphs for the partition planar
algebra and the Haagerup one are the same for distance $\leq 5 $ from *.

\rm{
Let us
 calculate both principal graphs for $\cal C$ (for $k >4$)
to distance $5$ from *, together with the weights of the
trace $Tr$ which we will use later. Our method will be a bit ad hoc but adapted to the
needs of this paper. By counting orbits we know there are two central projections  $e$ and $f$ 
orthogonal to the basic construction in ${\cal C}_4$. This is the same as saying that the
principal graphs are both as below at distance $\leq 4$ from *:

\begin{diagram} \label{4box}
\vpic{4box} {3in}
\end{diagram}

Our first step will be to calculate $\alpha =Tr(e)$ and $\beta=Tr(f)$.
 We know that 
$${\cal C}_{4,+} \cong M_2(\mathbb C)\oplus M_3(\mathbb C)\oplus e\mathbb C
\oplus f\mathbb C$$
and that the trace of a minimal projection in $M_2(\mathbb C)$ is 1,
that of a minimal projection in $M_3(\mathbb C)$ is $\delta^2-1$ 
and that 
\begin{formula}\label{alphaplusbeta}
\qquad \qquad $\alpha + \beta = \delta^4 - 3\delta^2 +1$
\end{formula}
 To obtain another relation on $\alpha$ and $\beta$ one may proceed as follows.
By \cite{J16}  ${\cal C}_{4,+}$ is spanned
by TL and the flip transposition $S$  on $V\otimes V$. With attention
to normalisation one gets $Tr(S)=\delta^2 (=k)$. But in the
2-dimensional representation $S$ is the identity and in the
3-dimensional one it fixes one basis element and exchanges the
other two so it has trace $1$. We may assume $eS=e$ and $fS=-f$
for if both reductions had the same sign $S$ would be in TL.
Thus 
\begin{formula}\label{otheralphabeta}
\qquad \qquad $\alpha - \beta +2+(\delta^2-1) = \delta^2$
\end{formula}
 Combining this with \ref{alphaplusbeta} we obtain
\begin{formula}\label{comb+}
\qquad \qquad $\displaystyle{\alpha={\delta^4-3\delta^2\over 2}\qquad
\beta={\delta^4-3\delta^2+2\over 2}}$
\end{formula}

We may now repeat the argument for ${\cal C}_{4,-}$. It is easy to check that
${\cal F}(S)$ is a multiple of a projection  which is in the centre. It is the 
identity in the 2-dimensional representation and zero in the 3-dimensional
one. The trace of the projection is $\delta^2$ and it must be zero on 
one of the projections orthogonal to TL and 1 on the other. Thus the equation 
corresponding to \ref{otheralphabeta} becomes simply
$$2+\alpha =\delta^2.$$

Combining this with \ref{alphaplusbeta} we obtain
\begin{formula}\label{comb-}
\qquad \qquad $\displaystyle{\alpha={\delta^2-2}\qquad
\beta={\delta^4-4\delta^2+3}}$
\end{formula}

Now we can deduce the algbebra structures of ${\cal C}_{4,\pm}$.
By counting orbits we see that the principal graphs to distance
five must have precisely two extra vertices (for $k >4$). But since
the traces form an eigenvector for the adjacency matrix, nothing
else can be attached to the vertex with trace $\delta^4-2$ in
the graph for ${\cal C}_{4,-}$, and something must be attached to
both vertices for ${\cal C}_{4,+}$. The conclusion is that the
two principal graphs are as below to distance $5$ from *.
\5
\vpic{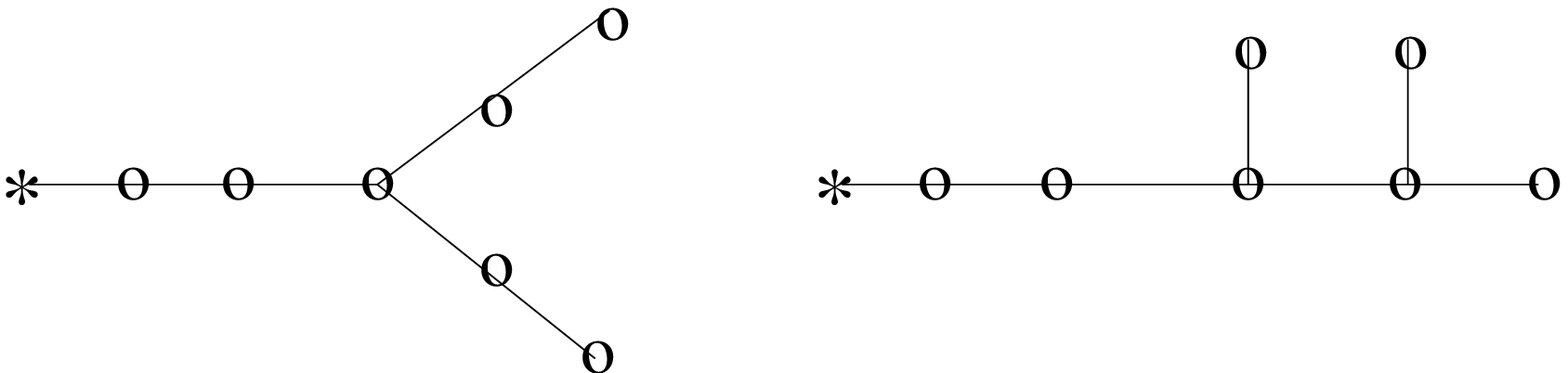} {4in}

}
By counting orbits one may calculate the chirality of the partition 
planar algebra. It is $-1$.

Note the identity between the principal graphs of the Haagerup subfactor and the
partition planar algebra for distance up to five from *.
The partition algebra has been studied 
by several people (\cite{Martin},\cite{HR}) with generic
values of the parameter $k$. One might be tempted to think
that the Haagerup subfactor is some kind of specialisation
of the partition algebra but this is not at all the case.
For the chiralities are $-1$ (Haagerup) and $+1$ (partition), so these two planar algberas
must be considered very distant cousins indeed.

\begin{example} {Haagerup-Asaeda.} The principal graphs and traces are in \cite{AH}. The
annular multiplicity sequence begins $0^51010$. We will see that the chirality is $+1$.
\end{example} 

\begin{example}\label{fusscatalan} {Fuss-Catalan} The planar algebra of \cite{BJ} is 1-supertransitive. For generic values
of the parameters $a$ and $b$ (e.g. $a>2, b>2$ the principal graph begins:\\

\vpic{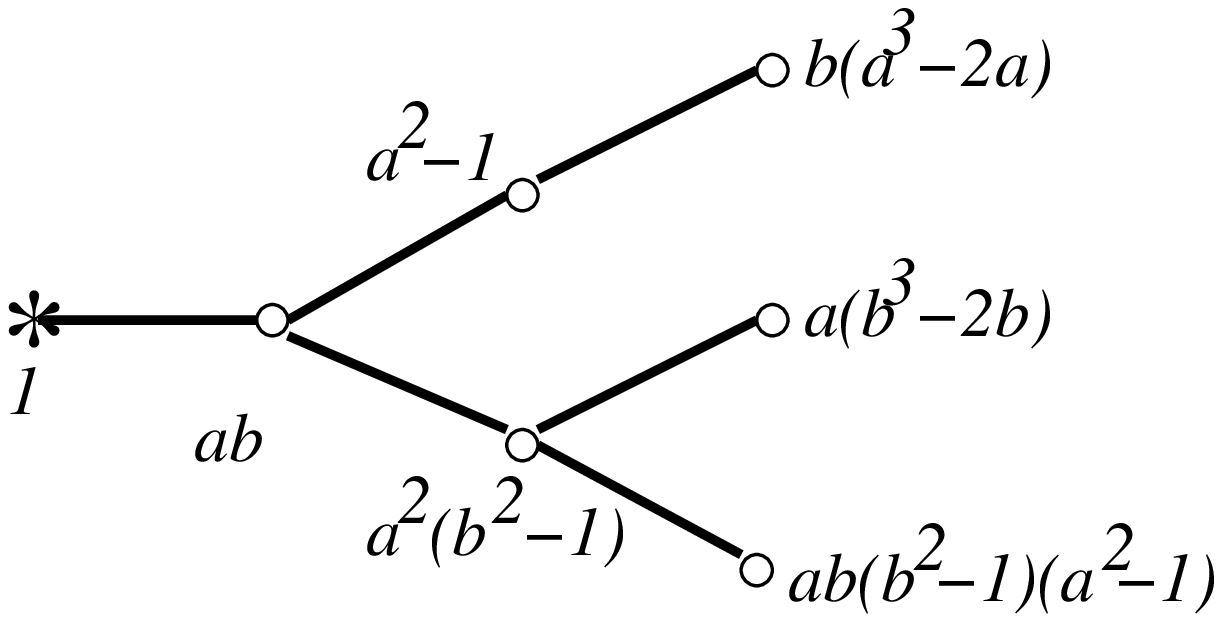} {3in}

\rm{We have recorded the (non-normalised) traces of the minimal projections corresponding
to the vertices of the graph. The dual principal graph is the same except that the roles of
$a$ and $b$ are reversed. 

The multiplicity sequence begins $*11$. Observe however that for the special allowed value
$a=\sqrt 2$ the principal graph and dual principal graph begin:\\

\vpic{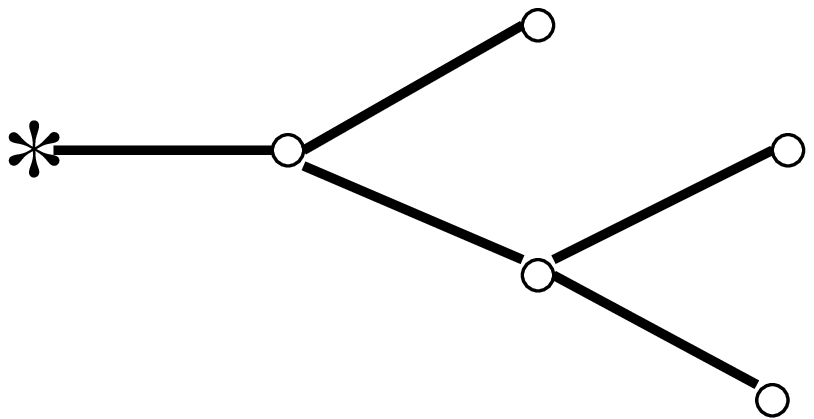} {3in}   \vpic{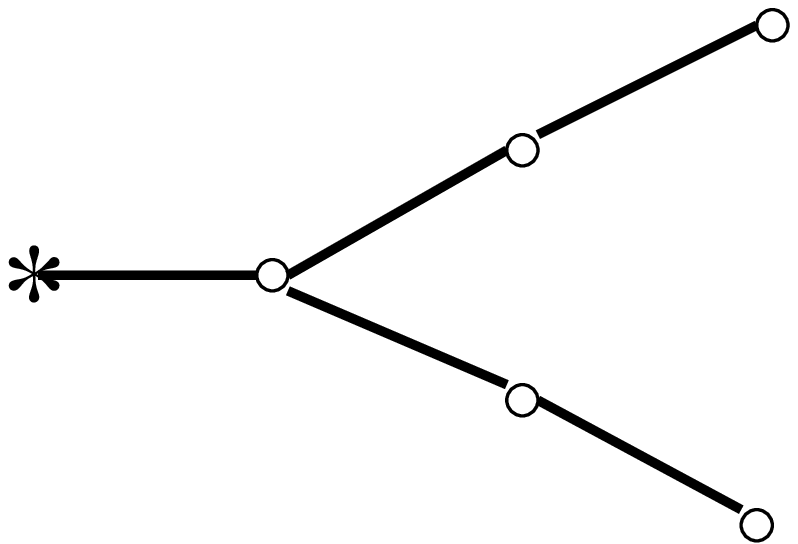} {2.8in}

and the multiplicity sequence then begins $*10$.

The biprojection generating the Fuss-Catalan algebra is rotationally invariant
so the chirality of the algebra is $+1$.
}
\end{example}
\begin{remark}\label{differentdual}
\rm{Observe that in all the examples whose annular multiplicity sequence begins $0^k10$ the principal and
dual principal graphs are different and begin like the Haagerup ones, but with a longer or shorter 
initial segment from $*$. We will explain this phenomenon with the quadratic tangle results.}
\end{remark}

\comment{\subsection{FAD}

If a subfactor planar algebra is $(n-1)$-supertansitive with multiplicity $e$ then there is a
guaranteed minimal dimension for $P_{n+1}$. If we choose an orthonormal basis $\mathfrak B$
for the $e$-dimensional subspace of $P_{n}$ which is the orthogonal complement of $TL$
consisiting of eigenvectors for the rotation, then each $R\in \mathfrak B$ generates an annular
TL-module-its "annular consequences".  Let $\A$ be the  annular consequence
space in $P_{n+1}$. 
\begin{definition} With notation as above, the \emph{first annular deficiency} (FAD) of
\P is $$\dim{P_{n+1}}-\dim{TL_{n+1}}-\dim \A$$
\end{definition}
\begin{remark}
If $\delta>2$ we can be more explicit:\\ $\displaystyle FAD(\mbox{\P})=\dim P_{n+1}-\frac{1}{n+2} {2n+2\choose n+1} - (2n+2)e$.
\end{remark}

For instance the Haagerup and Haagerup-Asaeda subfactors have zero FAD. 
end comment}

\section{Inner product formulae.}\label{ipf}
\subsection{Setup for this section.}\label{setup}
Let \P have annular multiplicity beginning $0^{n-1}e$. 
The cyclic group of order $n$ acts unitarily on the $e$-dimensional orthogonal
complement of the Temperley-Lieb subspace of $P_{n,+}$ via the
rotation tangle. Choose an orthonormal basis  $\mathfrak B =\{R\}$ of
self-adjoint eigenvectors for the rotation with $\omega_R$ being
the eigenvalue of $R\in \mathfrak B$. For each $R\in \mathfrak B$, ${\cal F}(R)^*=\omega^{-1}{\cal F}(R)$
so we \emph{choose} a square root $\sigma_R$ of $\omega$ and define $\check R= \sigma^{-1}{\cal F}(R) \in P_{n,-}$
so that $(\check R)^*=\check R$. Thus\\
\vspace {10pt}
$\displaystyle {\cal F}(R)=$ \vpic{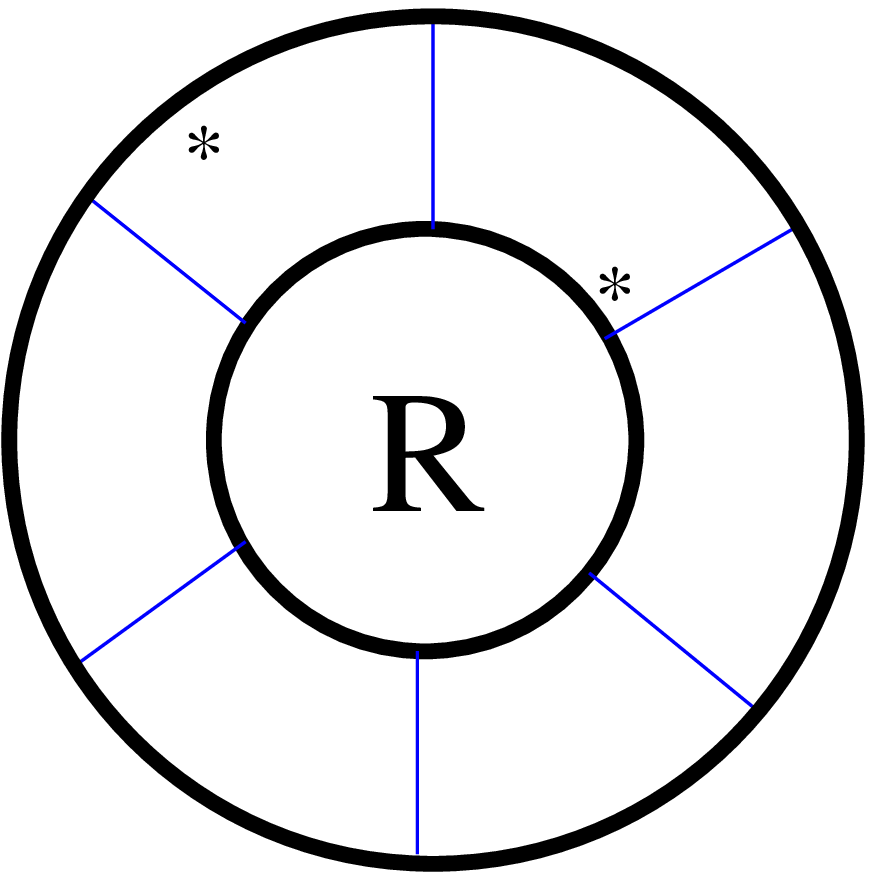} {1.0in}  $\displaystyle= \sigma_R \check R$ \quad and \quad
${\cal F}(\check R)=$ \vpic{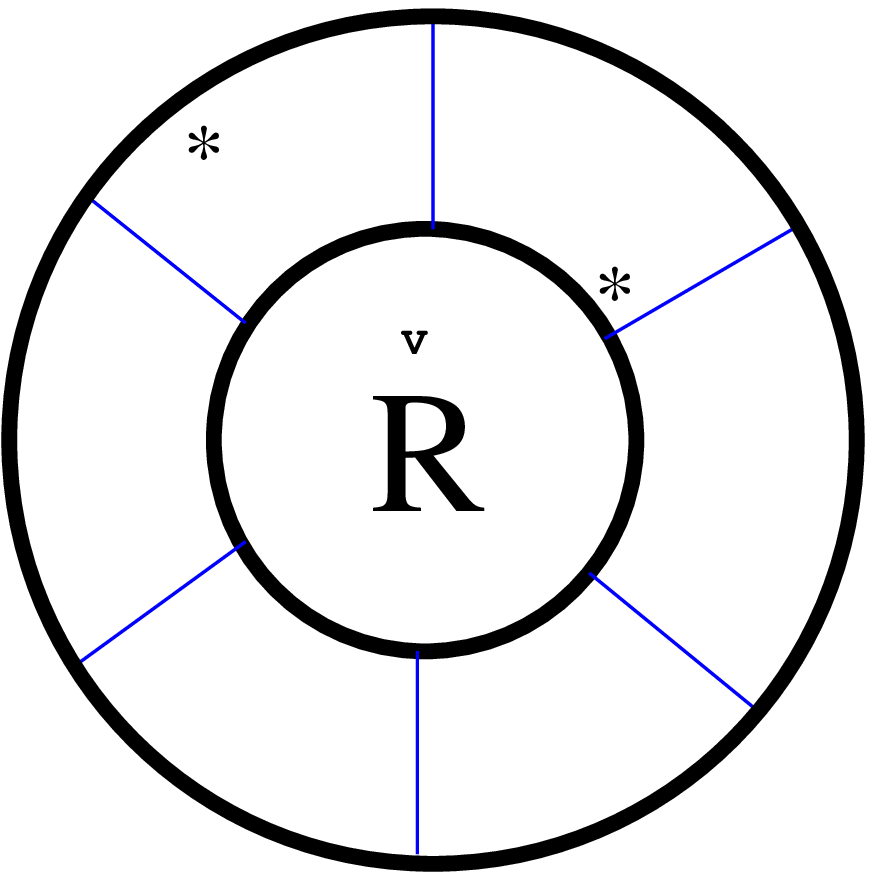} {1.0in}  $\displaystyle= \sigma_R  R$\\
and we record the verbal version:\\
``If a tangle contains a disc $D$ labelled with $R$ (resp. $\check R$) then it is the same as $\sigma_R$ times 
the same tangle with the * of $D$ rotated counterclockwise by one, and $R$ (resp. $\check R$) in $D$ replaced by $\check R$ (resp. $R$).''

\begin{remark}\label{signs} As we have defined them in \ref{setup}, all the $R$'s $\check R$'s and $\sigma$'s could
be changed by a sign. One can do better than this under appropriate circumstances. For instance if
the traces of both $R^3$ and $(\check R)^3$ (which are necessarily real) are non-zero, one can impose
the choice of $R$ and $\check R$ which makes both these traces positive. Then $\sigma_R$ is an unambiguous
$2n$th. root of unity. But this does not always happen,
even in the Haagerup subfactor, so we must live with the sign ambiguity. Of course any constraints we obtain must depend
only on $\omega_R$. 

\end{remark}

\subsection{Inner products of annular consequences}
\begin{definition}
Let $\A$ be the subspace of $P_{n+1,+}$ spanned by the image of
$\mathfrak B$ under the annular Temperley-Lieb tangles. 
\end{definition}
 
We will use two bases for $\A$ - carefully chosen annular consequences of
the $R\in \mathfrak B$ and the dual basis. (That the annular images are linearly independent
is shown in \cite{J21}.) 

\comment{ Some arbitrary choice is necessary but we want to minimise it so we will
choose two annular consequences of each $R$ and let the rotation define the rest.

\begin{definition} For $R\in \mathfrak B$ let $\cup R =$  \vpic{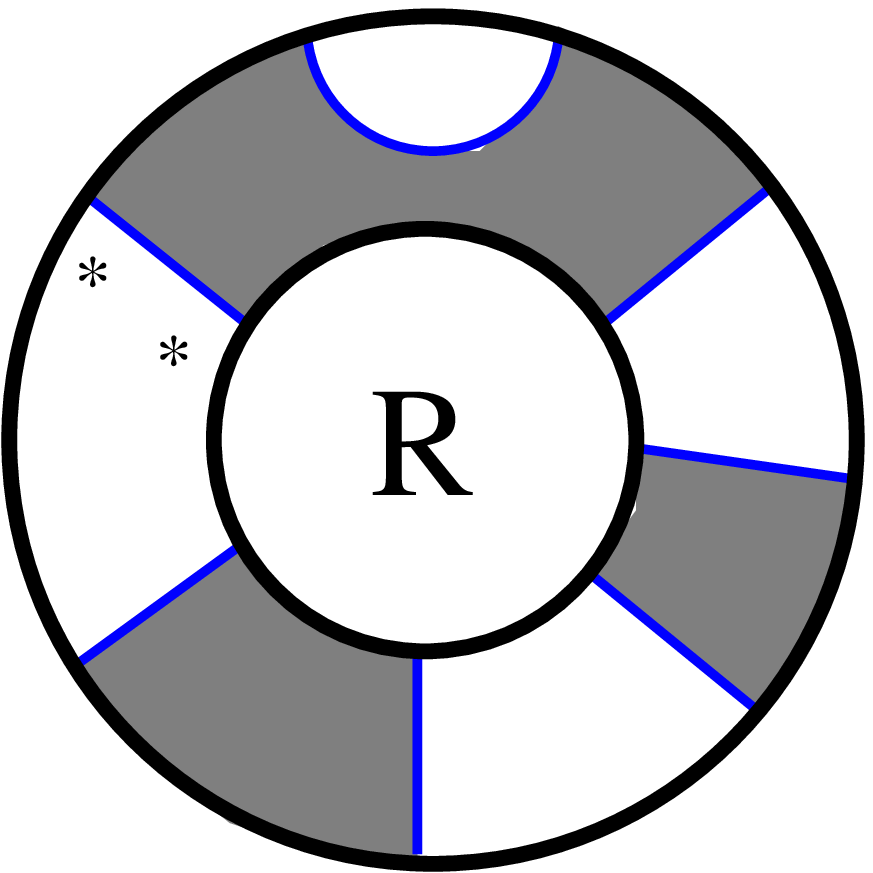} {1in} 
and $\uplus R = $ \vpic{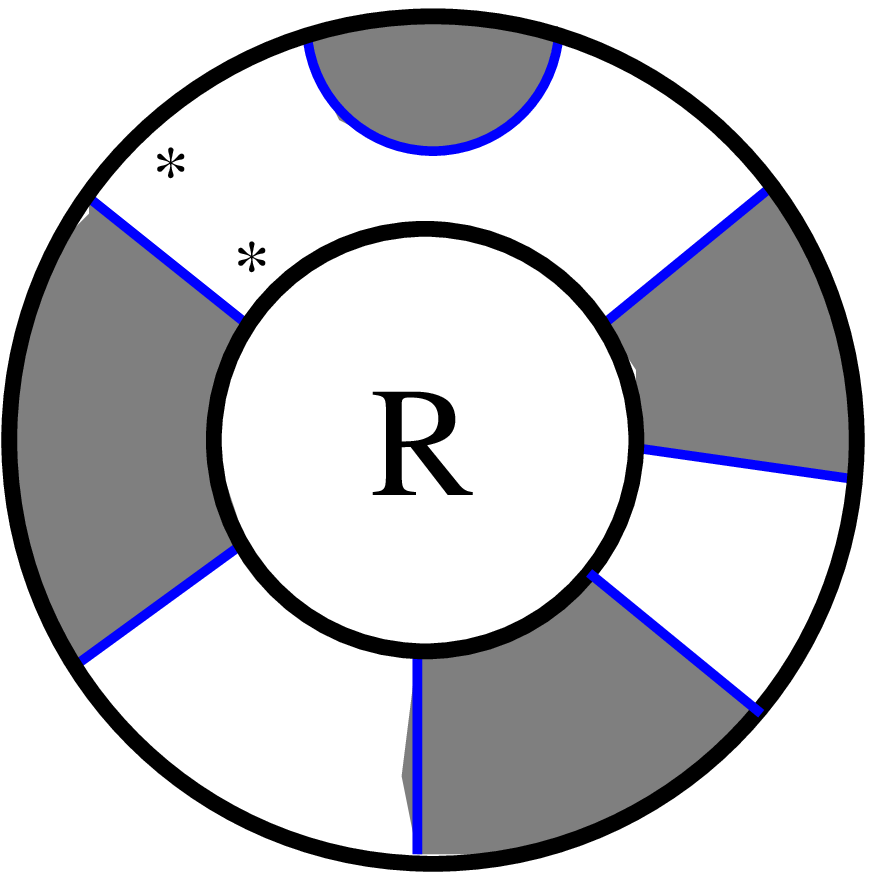} {1in} .

From these define for $i\in \mathbb Z / (n+1)\mathbb Z$,\\
$$\cup _i R= \rho^i (\cup R)$$ and $$\uplus_i R= \rho^i(\uplus R).$$

\end{definition}

end comment}

\begin{definition} In $P_{n+1,+}\oplus P_{n+1,-}$ let $$w=( \vpic{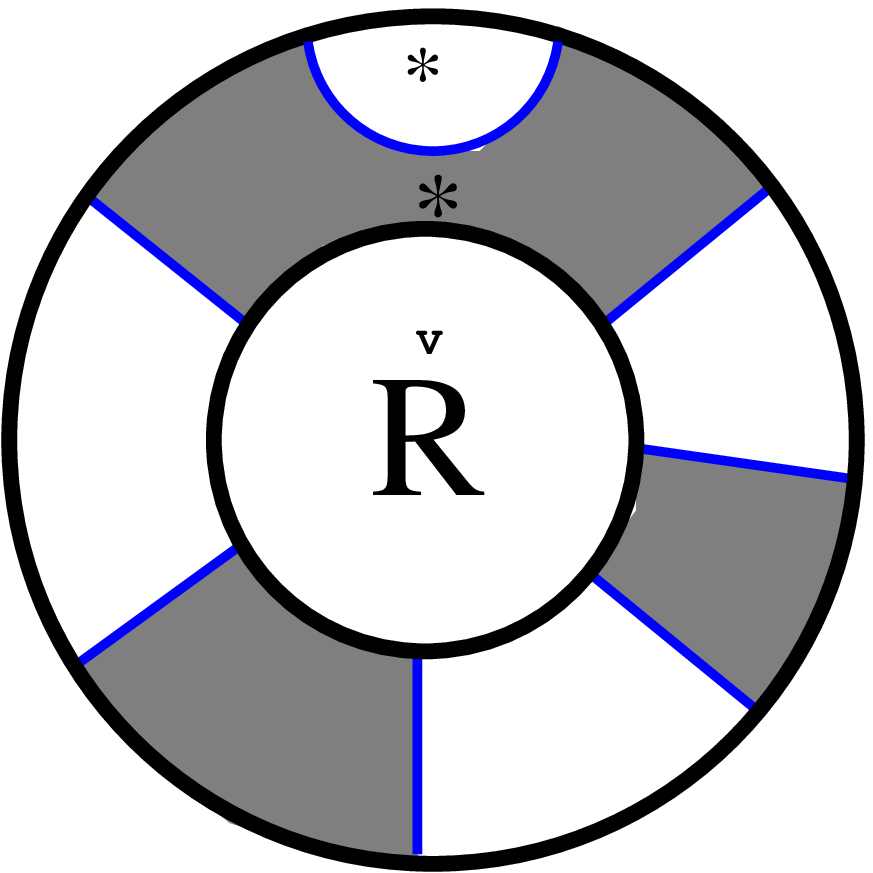} {1in} , \vpic{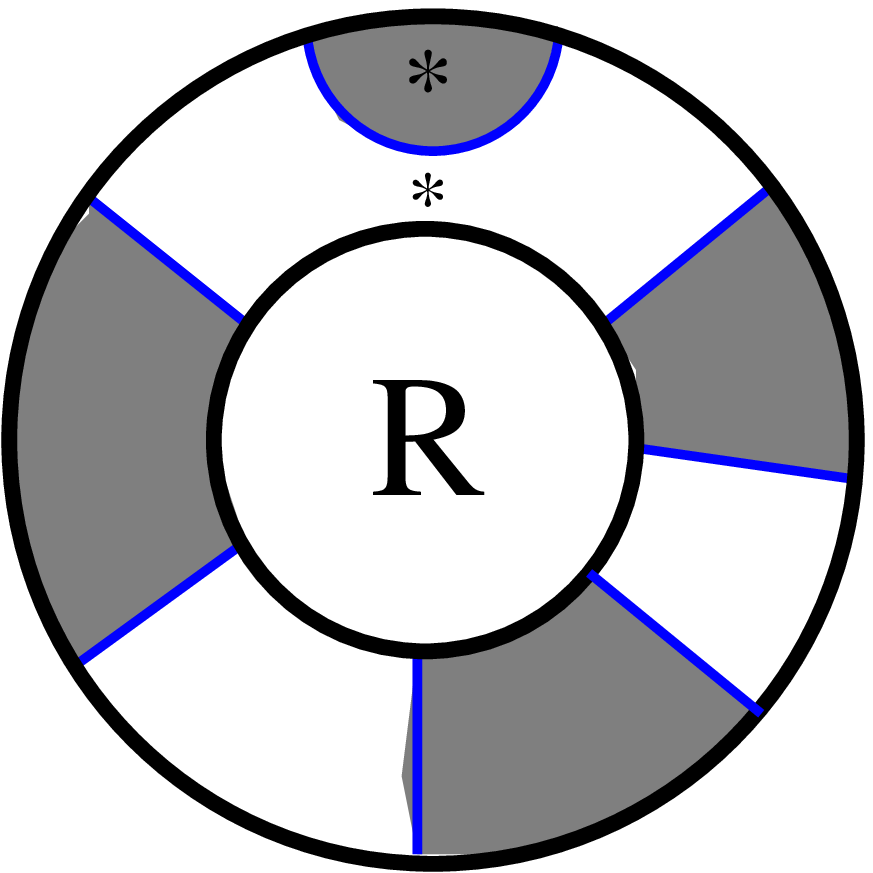} {1in} )$$

Let $\pi_\pm$ be the projections onto the first and second components of\\ $P_{n+1,+}\oplus P_{n+1,-}$ and 
for $i\in \mathbb Z/(2n+2)\mathbb Z$ define the following elements of $P_{n+1,+}$ :
$$\cup_i R=\pi_+ ({\cal F}^i (w)).$$

\end{definition}

\begin{proposition} We have the following inner product formulae:\\ 

$\begin{array}{cc}
(a) &   \qquad \langle \cup_iR,\cup_iR\rangle=\delta  \cr
(b) &  \qquad  \langle \cup_iR,\cup_{i\pm1}R\rangle=\sigma_R^{\pm 1} \cr

\end{array}$

All other inner products among the  $\cup_i R$ are zero .
\end{proposition}
\begin{proof} Since $\cal F$ is unitary and ${\cal F} \pi_+=\pi_-{\cal F}$,
$$\langle {\pi_+ {\cal F}^i(w),\pi_+{\cal F}^j(w)\rangle=\langle \pi_\pm (w),\pi_\pm {\cal F}^{j-i} (w)}\rangle$$
the sign of $\pi$ depending on the parity of $i$. This is clearly zero if $j-i\notin \{-1,0,1\}$ and equal to $\delta$
if $i=j$.  Both $\pi_+(w)$ and $\pi_-(w)$ are self-adjoint and the pictures for  computing the remaining inner products are
the same for both parities of $i$, up to changing shading and interchanging $R$ and $\check R$. We illustrate below
with $\langle \cup_0R,\cup_1R\rangle$. (The shadings are implied by the *'s and $R$, $\check R$):\\
\vpic{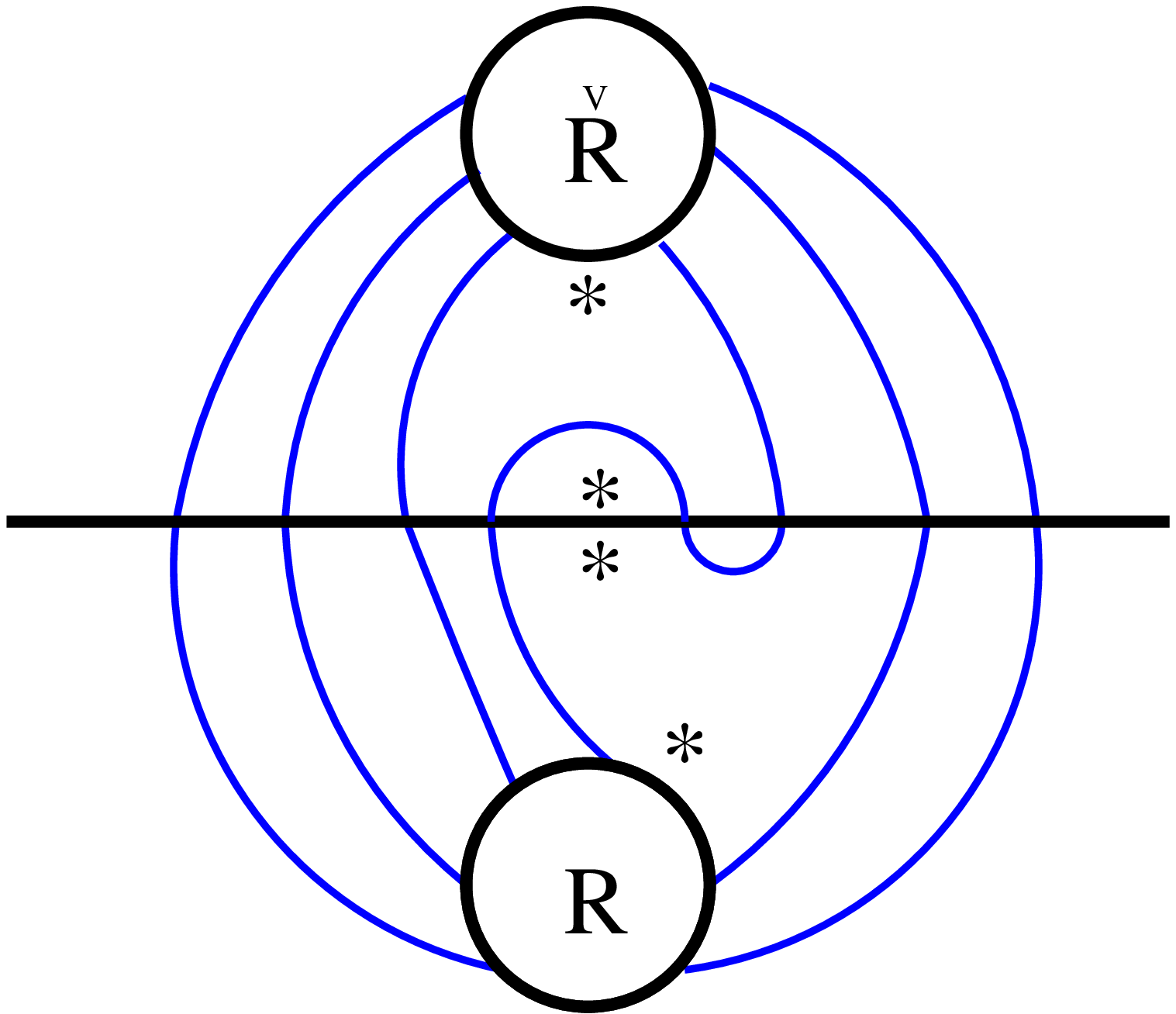} {2in}

\begin{corollary} \label{rootrot}The map $\rho^{1/2} : \mathfrak A \rightarrow \mathfrak A$  defined by 
$$\rho^{1/2}(\cup_iR)=\cup_{i+1}R$$ (for $ i\in \mathbb Z/(2n+2)\mathbb Z$) is unitary. (And of course $(\rho^{1/2})^2=\rho$.)
\end{corollary}

\comment{
(b)\vpic{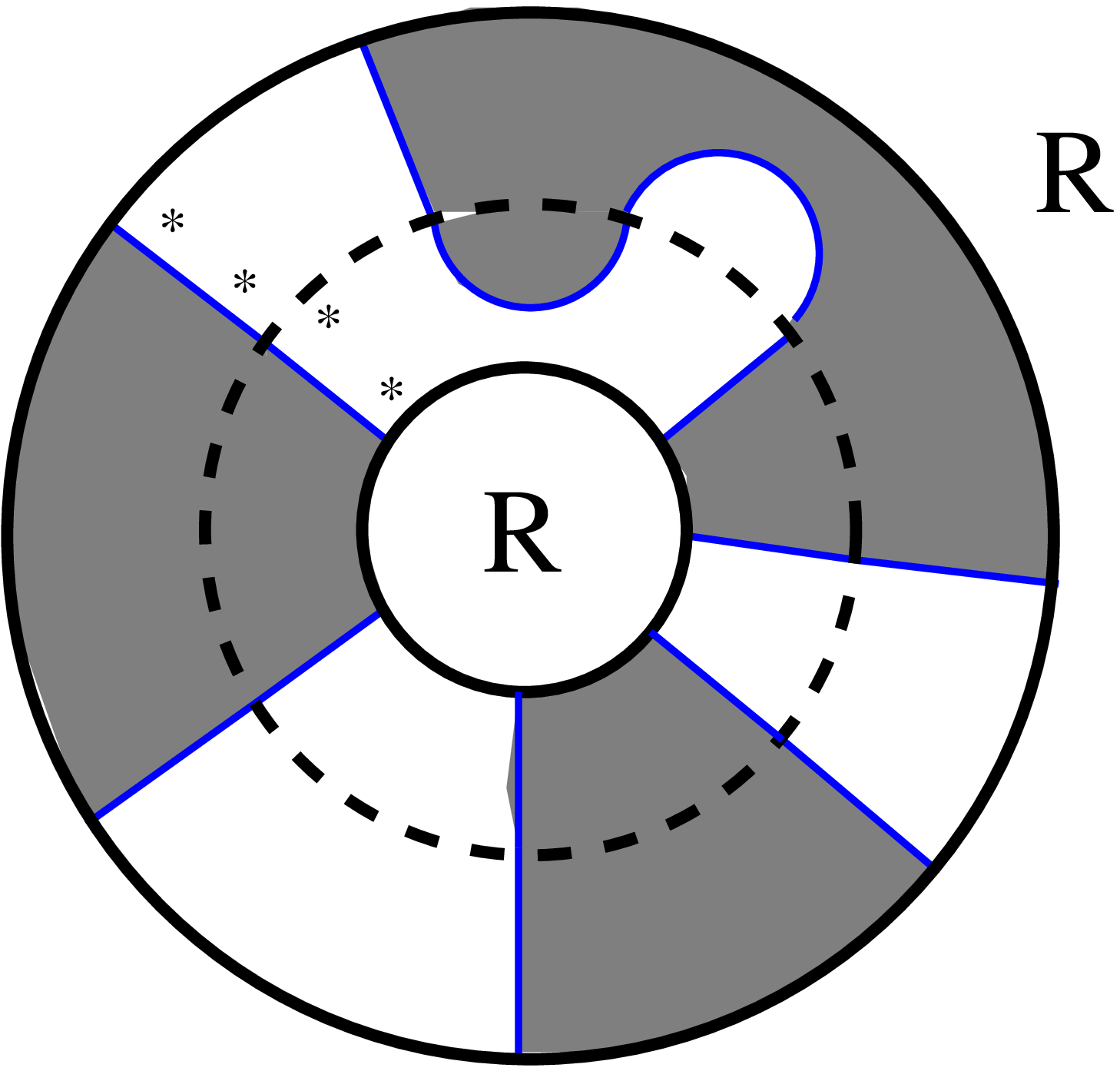} {1.5in} which we recognise as $  \langle \cup_0R,\uplus_0R\rangle$ 
and is equal to   $ \langle R,R\rangle$ by an isotopy, and\\
(c)\vpic{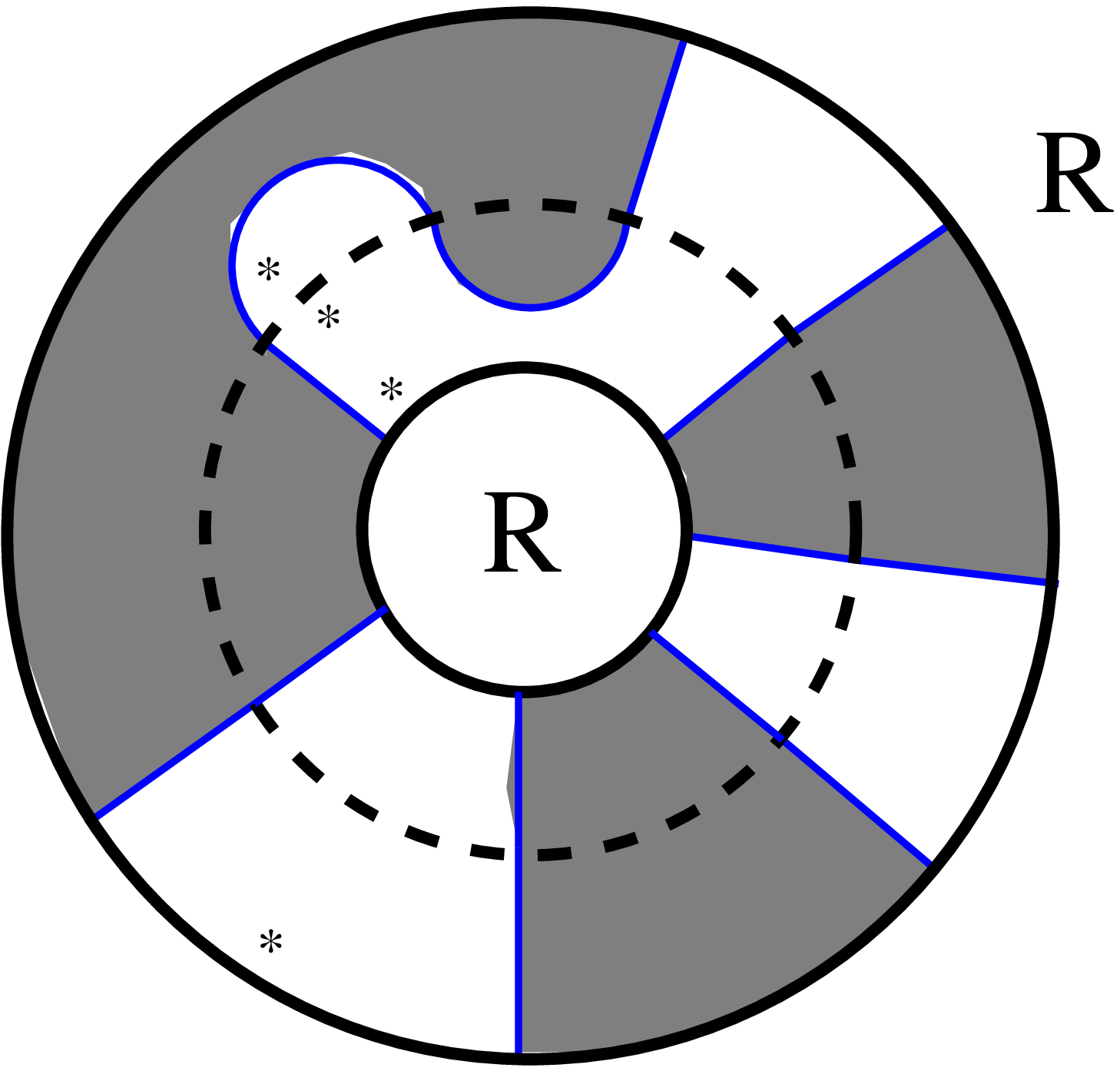} {1.5in} which we recognise as $  \langle \cup_{-1}R,\uplus_0R\rangle$.
After an isotopy and an application of \ref{rotated} we obtain the answer.
end comment}
\end{proof}

To project onto $\A$ we will use the dual basis to $\mathfrak B$. Fortunately
there is an elegant pictorial formula for the dual element $\hat R$ to $R$.
First we define an unnormalised form of it.

\begin{definition} For $R\in \mathfrak B$ as above let 
$$\tilde w = (\vpic{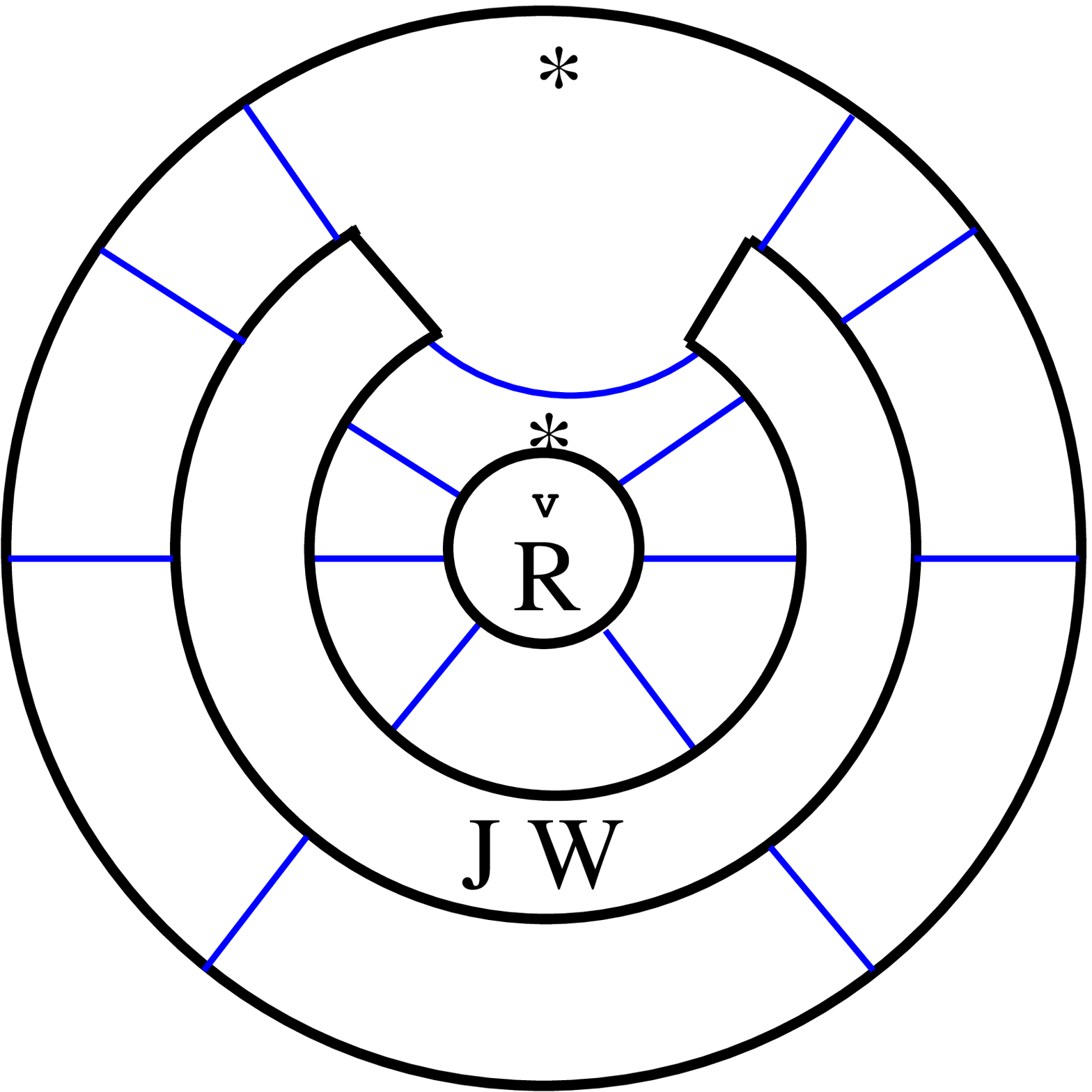} {1.5in} , \vpic{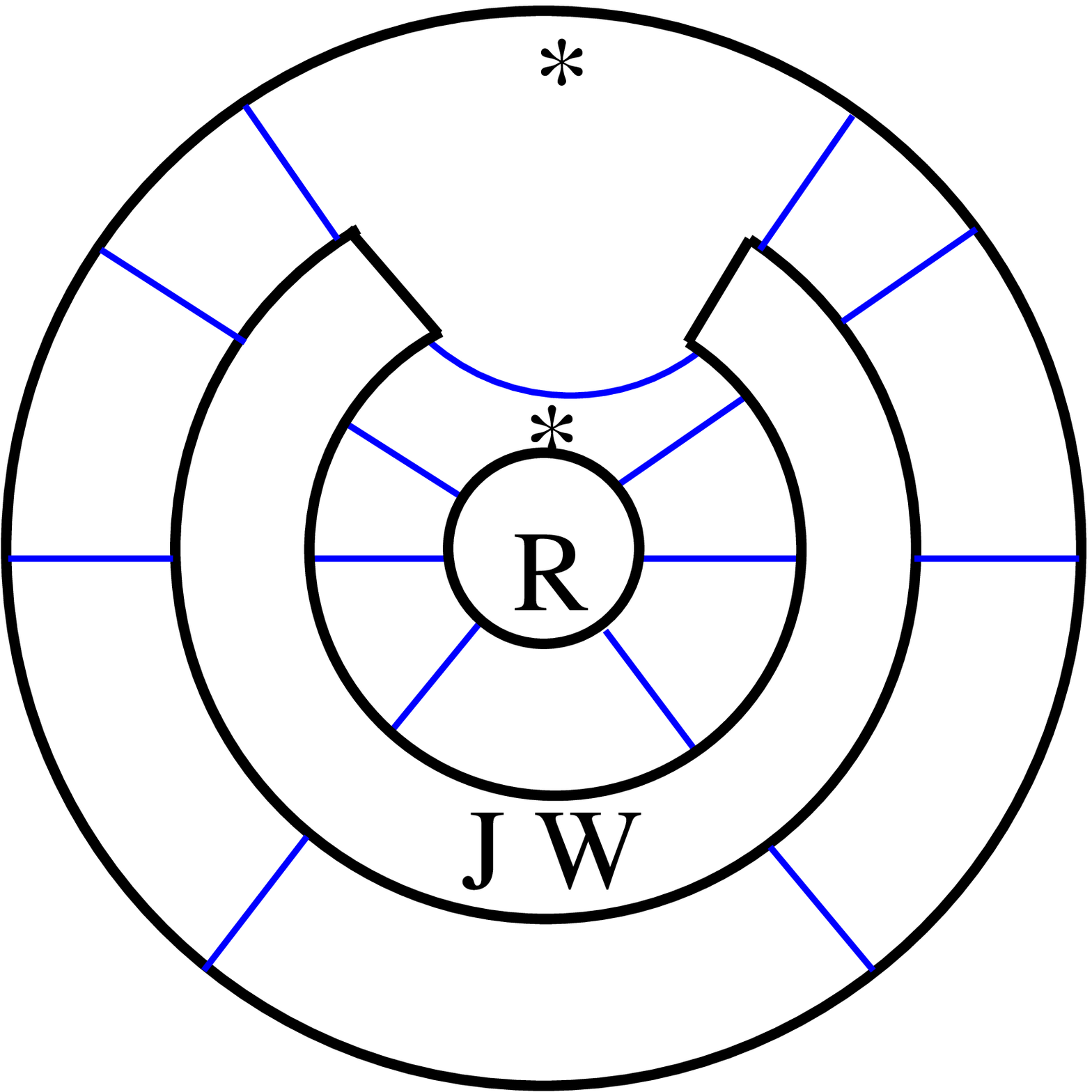} {1.5in} ) \in P_{n+1,+}\oplus P_{n+1,-}$$

\comment{
$$(a)   \qquad \tilde {\uplus R} = \vpic{shadedcupdual} {2in} $$ and
$$(a)   \qquad \tilde {\cup R} = \vpic{cupdual} {2in} $$ 
endcomment }
where in the crescent shaped areas we have inserted the appropriately
shaded Jones-Wenzl
idempotent $p_{2n+2}$ (the shading is implicit from the position of star).
\end{definition}

\begin{definition} For an integer k and a  $\omega\neq 0$ 
define $$W_{k,\omega}(q)=q^k+q^{-k}-\omega-\omega^{-1}.$$
\end{definition}

For $R$, $\check R$ as above and for $i\in \mathbb Z/(2n+2)\mathbb Z $
set $$\hat {\cup_i R}=\frac{[2n+2]}{W_{2n+2,\omega_R}(q)} \pi_+({\cal F}^i(\tilde w))$$

\comment{ \rho^i(\tilde {\cup R})$$
and$$\hat {\uplus_i R}=\frac{[2n+2]}{W_{2n+2,\omega_R}(q)} \rho^i(\tilde {\uplus R})$$
endcomment}

\begin{lemma}  For all $i$ and $j$,
$$\langle \hat {\cup_i R},{\cup_jR}\rangle={\{}\begin{array}{cc}
1& i=j\cr
0& i\neq j
\end{array}$$
\comment {$$\langle \hat {\uplus_i R},{\uplus_jR}\rangle=\{\begin{array}{cc}
1& i=j\cr
0& i\neq j
\end{array}$$
$$\langle \hat{\uplus_i R},\cup_j R\rangle=\langle \hat{\cup_i R},\uplus_j R\rangle=0.$$
endcomment}
\end{lemma}
\begin{proof} It is clear from the properties of the $p_k$'s that the only non-zero inner
products occur as written. So it is only a matter of the normalisation, which by unitarity
of the rotation reduces to showing 
$\displaystyle \langle \tilde{\cup_0  R}, \cup_0 R\rangle=\frac{W_{2n+2,\omega_R}(q)}{[2n+2]}$.
 
 We draw this inner product below:\\
 
 \vpic{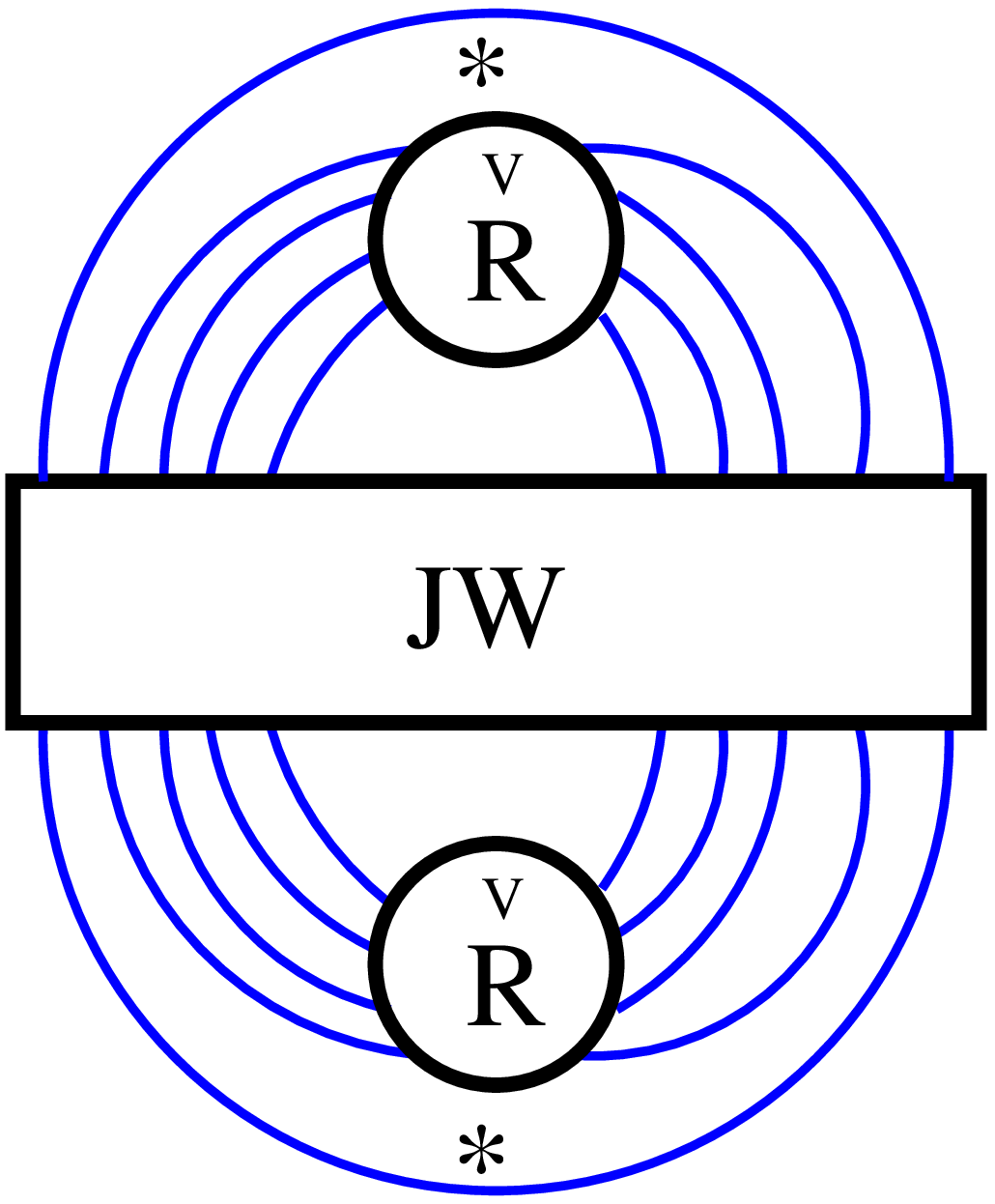} {2in}
 
There are 5 TL elements that give non-zero contributions to the inner product.
The first
is the identity which clearly contributes $\delta$. The next four come in reflected pairs.
The first pair contains \vpic{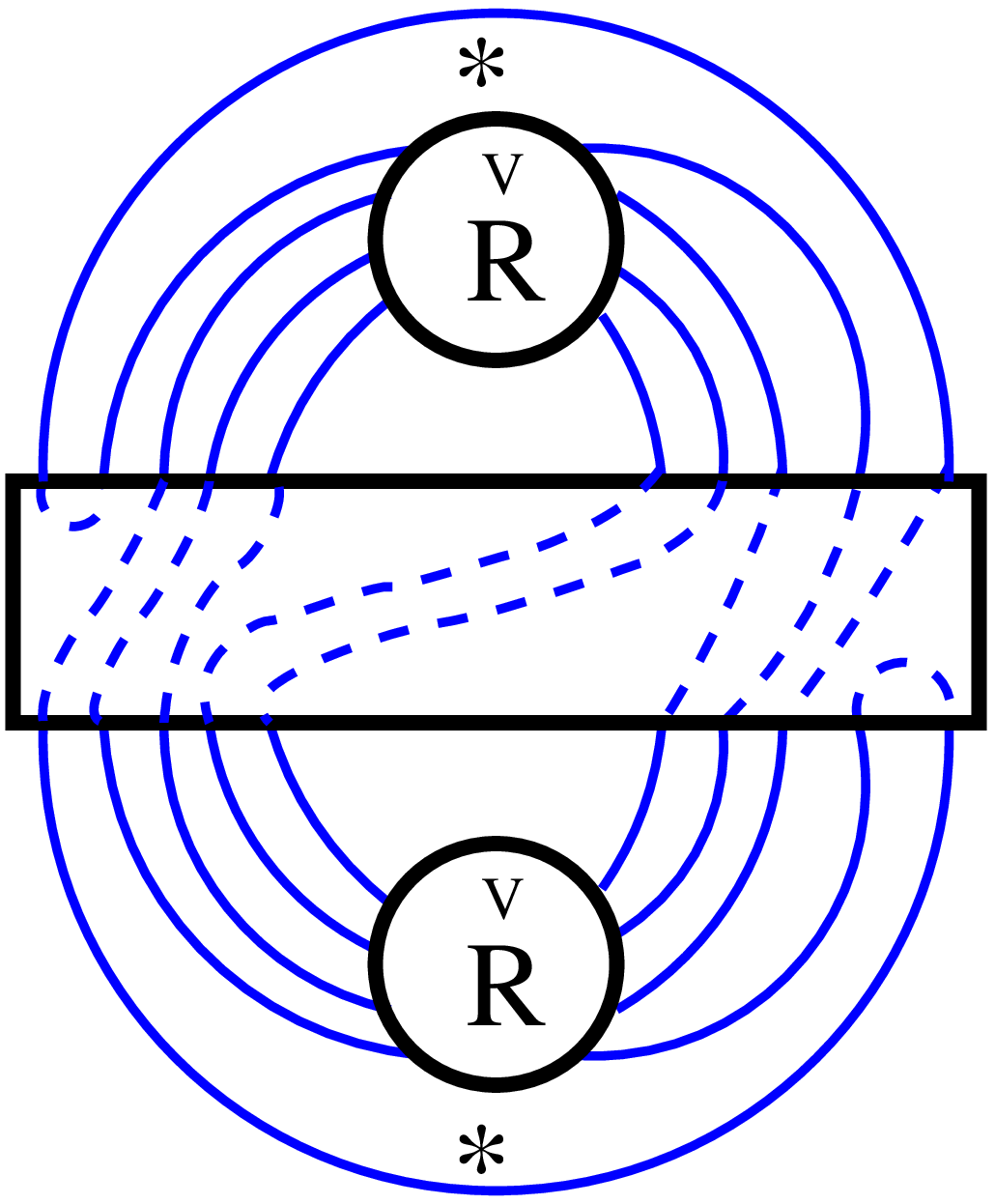} {1.3in} so the pair contributes 
$\displaystyle\frac{-(\omega+\omega^{-1})}{[2n+2]}$ by 
\ref{prodeis}. The other pair contains \vpic{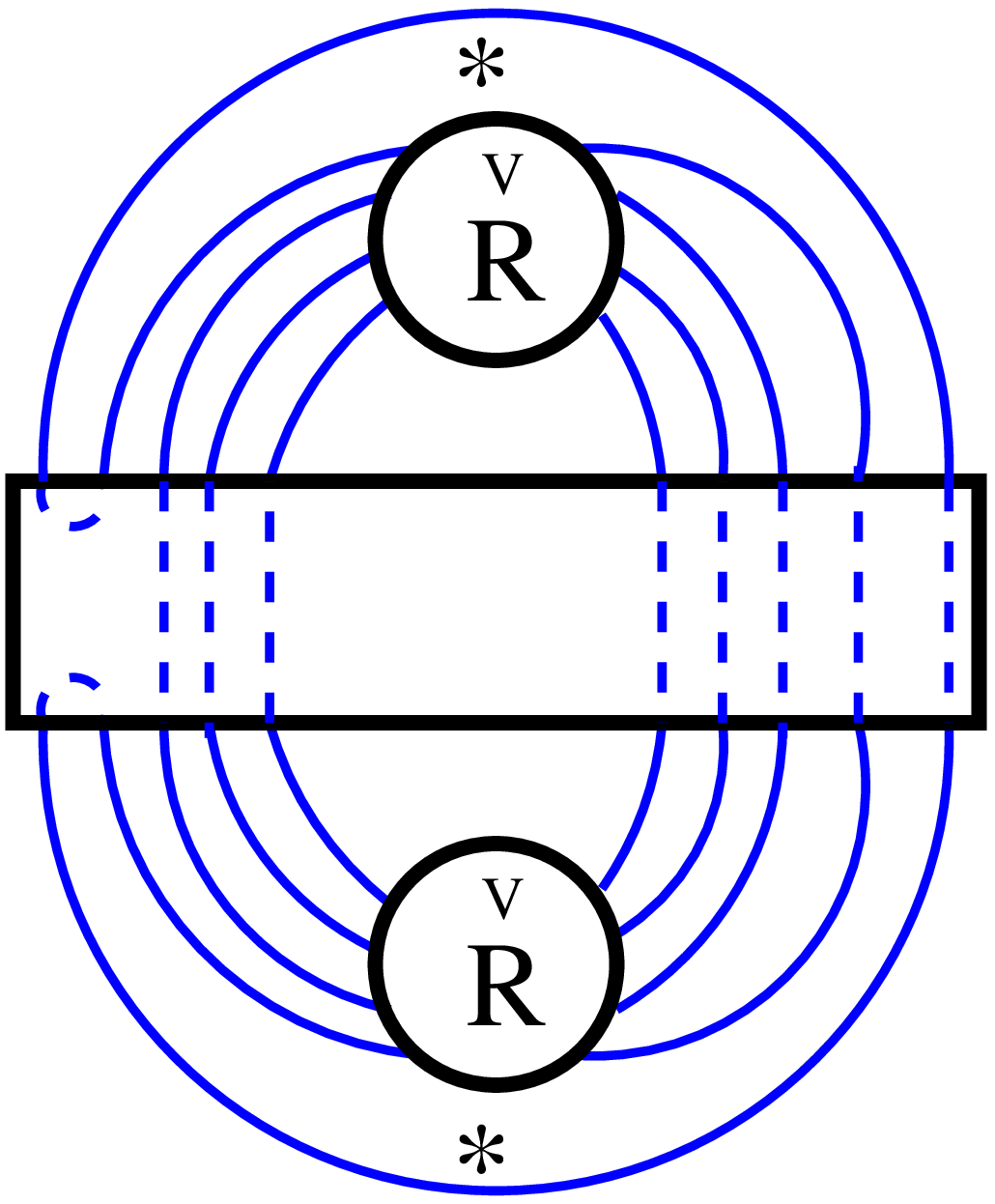} {1.3in} so contributes $\displaystyle-2\frac{[2n+1]}{[2n+2]}$
by \ref{prodeis}. Adding the contributions one gets\\ $\displaystyle \frac{[2n+2]\delta - 2[2n+1]-\omega
-\omega^{-1}}{[2n+2]}$ which is correct.
\end{proof}

We now expand the JW idempotent to obtain an expression for the dual basis in terms of the $\cup_i(R)$.

\begin{proposition} \label{dualbasis}
With $W=W_{2n+2,\omega_R}$ and $\sigma=\sigma_R$ we have\\
$\displaystyle \hat{\cup_0 R}=\frac{1}{W}\Big\{ [2n+2]\cup_0R + 
\sum_{i=-n}^{+n}
(-\sigma)^{n+i-1}([n+i+1] + \omega [n-i+1])\cup_{n-i+1}R{\Big \}}$
\end{proposition}

\begin{proof} Consider the diagram for $\pi_+(\tilde w )$:\\
\vpic{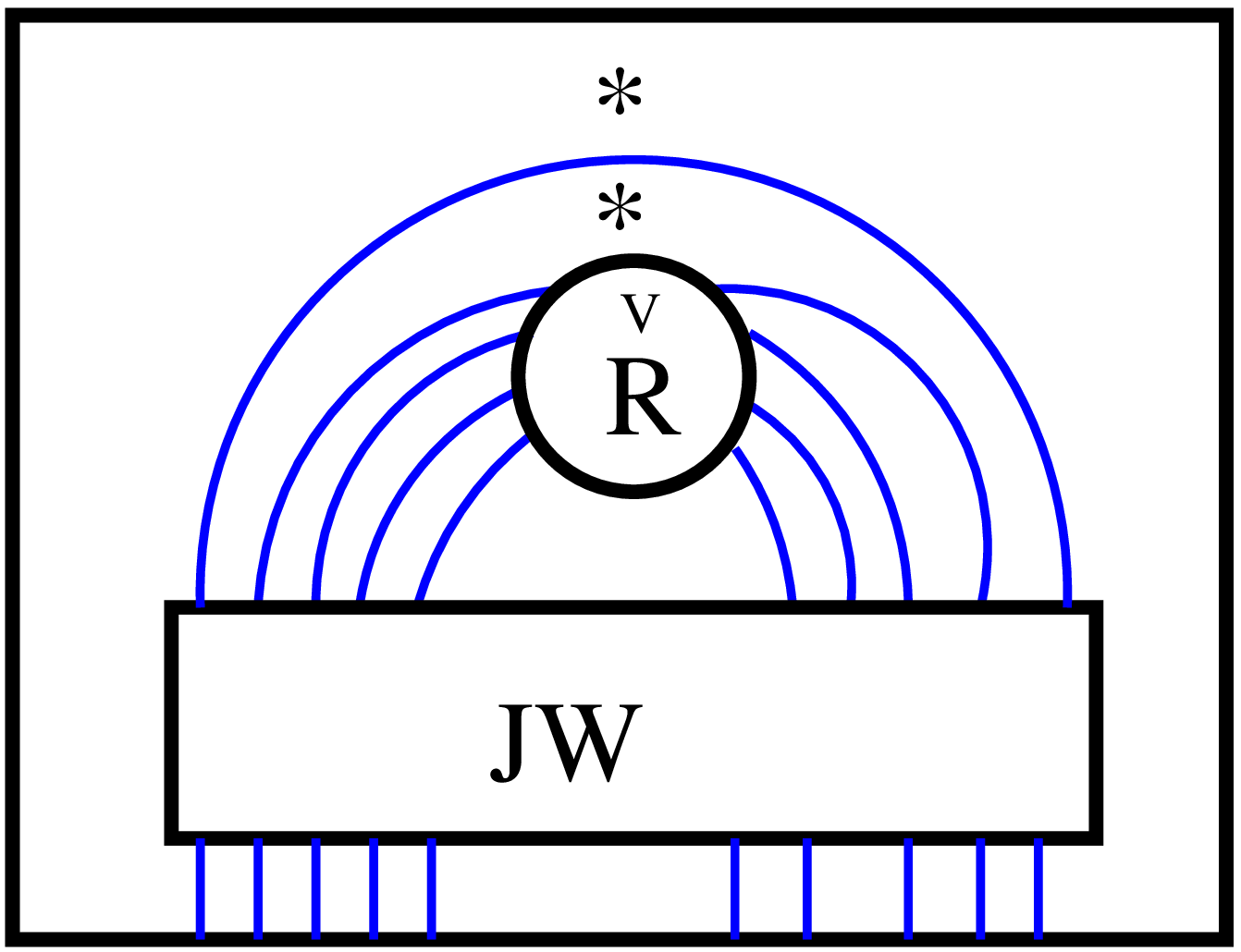} {1.3in} .
Caps at the top of the JW  can only occur at the extreme left and right but not at both.
The only way to get a multiple of $\cup_0R$ is to take the identity in JW. The rest of the
terms come in pairs, with caps at the top at the left and right, each with a cap at the bottom
which we will index by $i$, the distance from the middle interval at the bottom. We illustrate 
below the two terms contributing to $\cup_i R$ with $i>0$, $i=2$ in the example (if $i$ were
odd the only difference is that the internal box would contain $R$ rather than $\check R$) :\\

\vpic{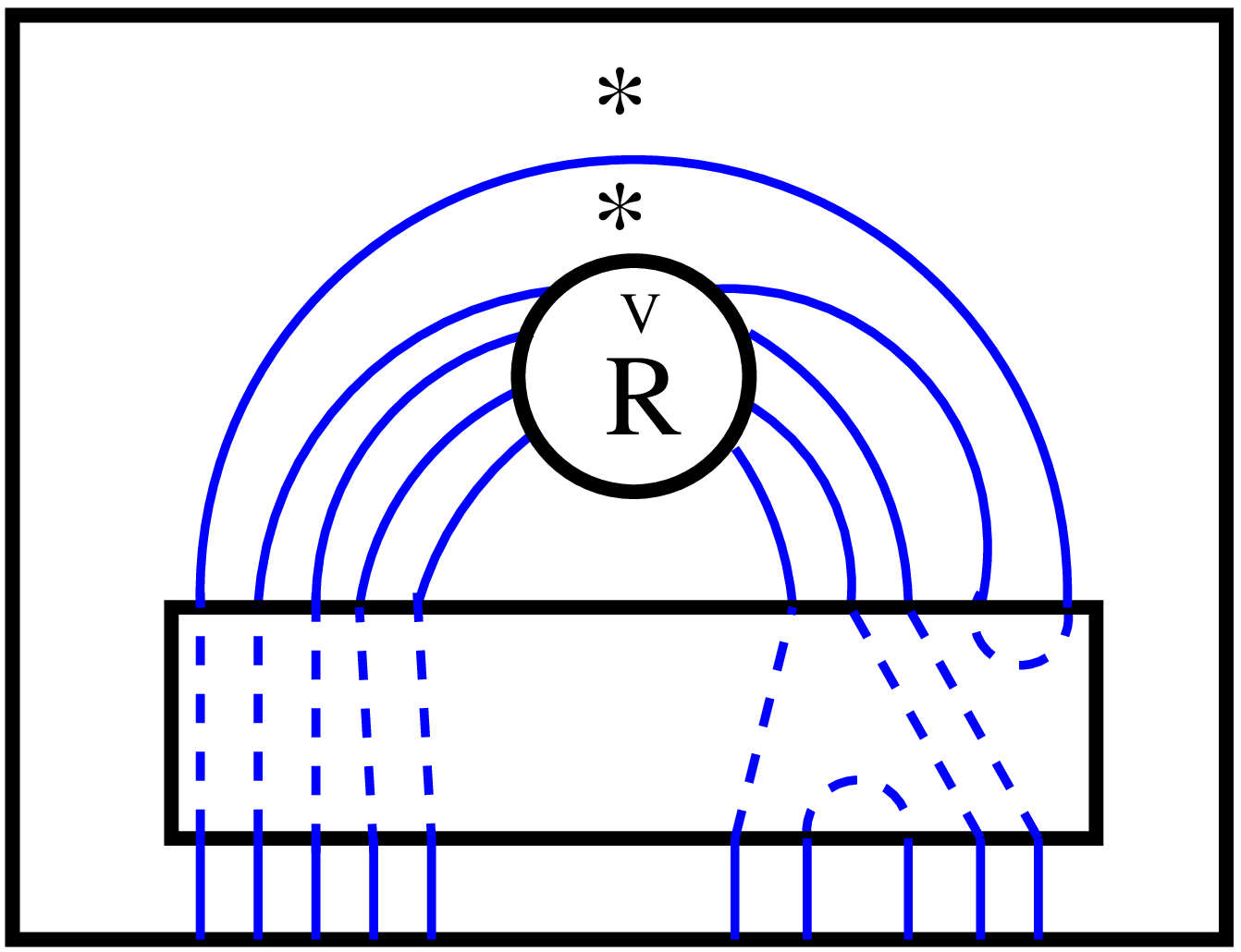} {2in} \qquad \qquad \vpic{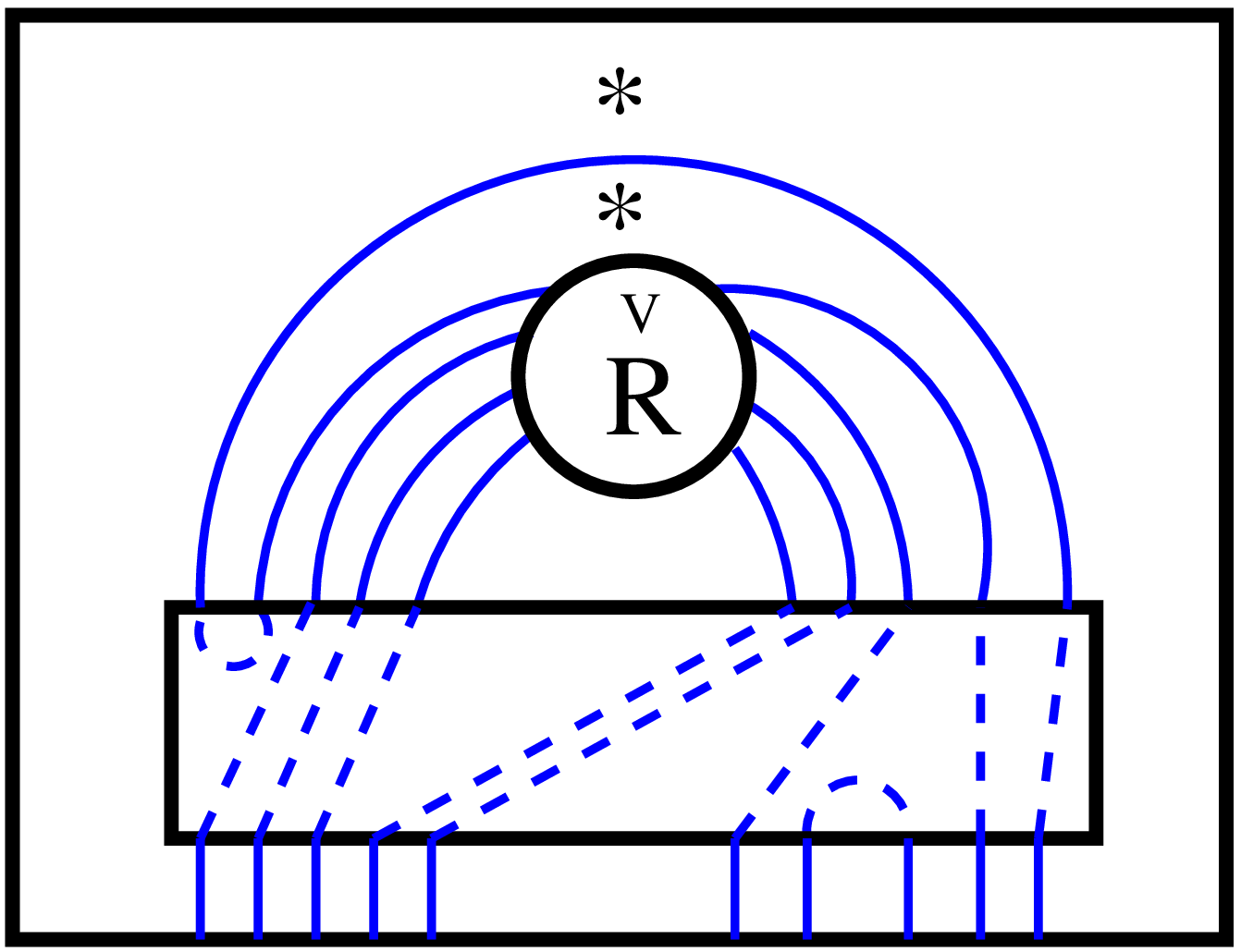} {2in} \\

In both cases the starred boundary interval lies $n-i+1$ intervals counterclockwise from
the cap on the outer boundary so both terms are multiples of $\cup_{n-i+1}R$. 
In the diagram on the left, the starred interval on the boundary of the disc containing $\check R$ needs
to be rotated $(n+i-1)$ intervals counterclockwise to line up with $\cup_{n-i+1}R$ and the term
inside the JW has $``p"=n-i$ in the notation of \ref{prodeis} whose coefficient is thus $\frac{[n+i+1]}{[2n+2]}$.
In the diagram on the left, the * of $\check R$ must be rotated $n+i+1$ counterclockwise 
and the $``p"$ factor is $n+i$. 

This establishes the terms in the sum for $i\geq 0$. Negative values of $i$ can be obtained by
taking the adjoint and using $(\cup_kR)^*=\cup_{-k} R$.
\end{proof}

All the other $\hat \cup_i R$ can be obtained from the above by applying a suitable power of $\rho^{1/2}$
but care needs to be taken with the indices as $\sigma_R$ is a $(2n)$th.  root of unity and not a $(2n+2)$th.

The formula is so important that we record  a few other versions of it.
\pagebreak
\begin{proposition}\label{versions}
(\romannumeral 1) \\ $\displaystyle \hat{\cup_0 R}=\frac{1}{W}
\Big\{ (-1)^{n+1}(\sigma +\sigma^{-1})[n+1] \vpic{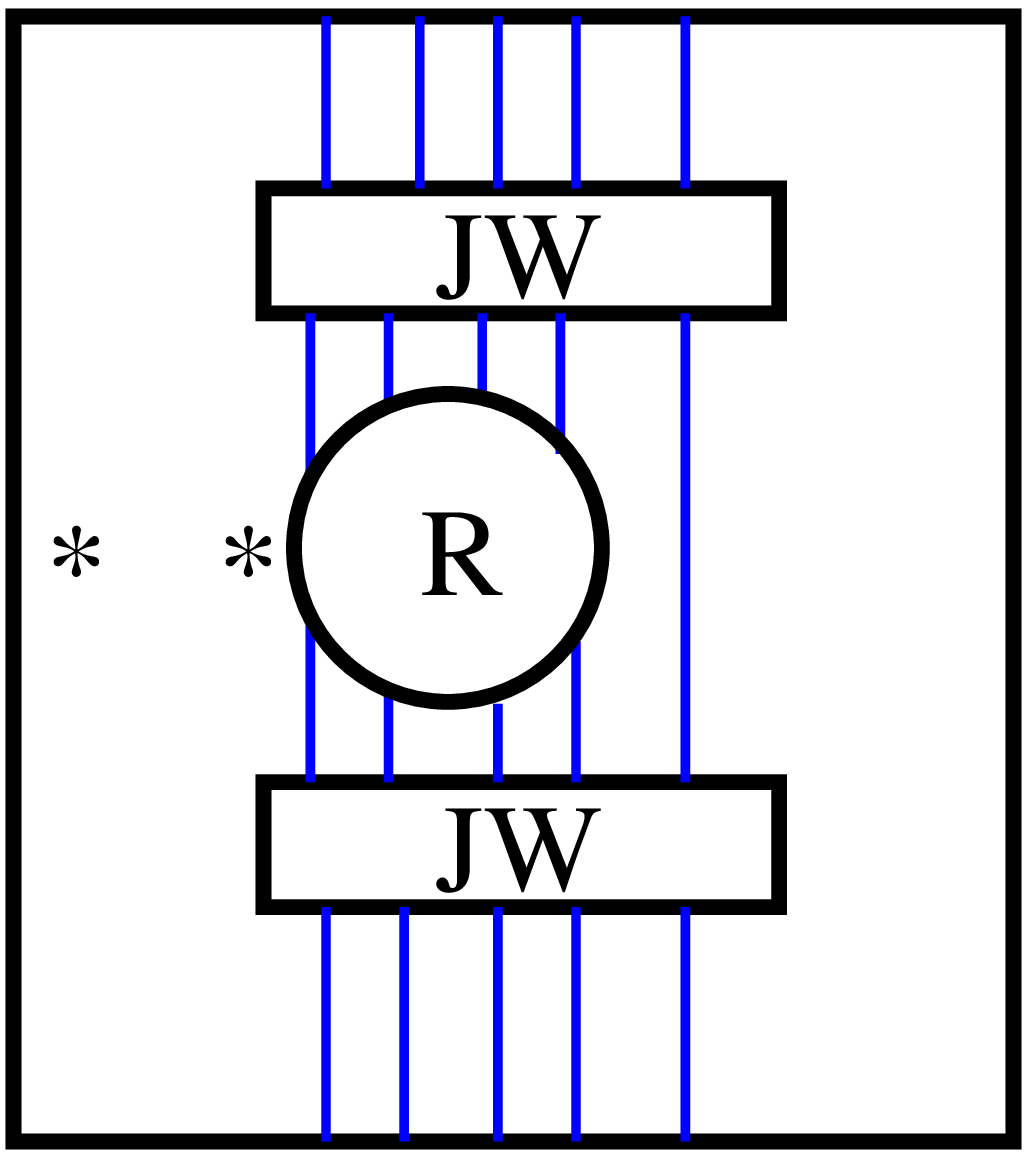} {1in} +[2n+2] \vpic{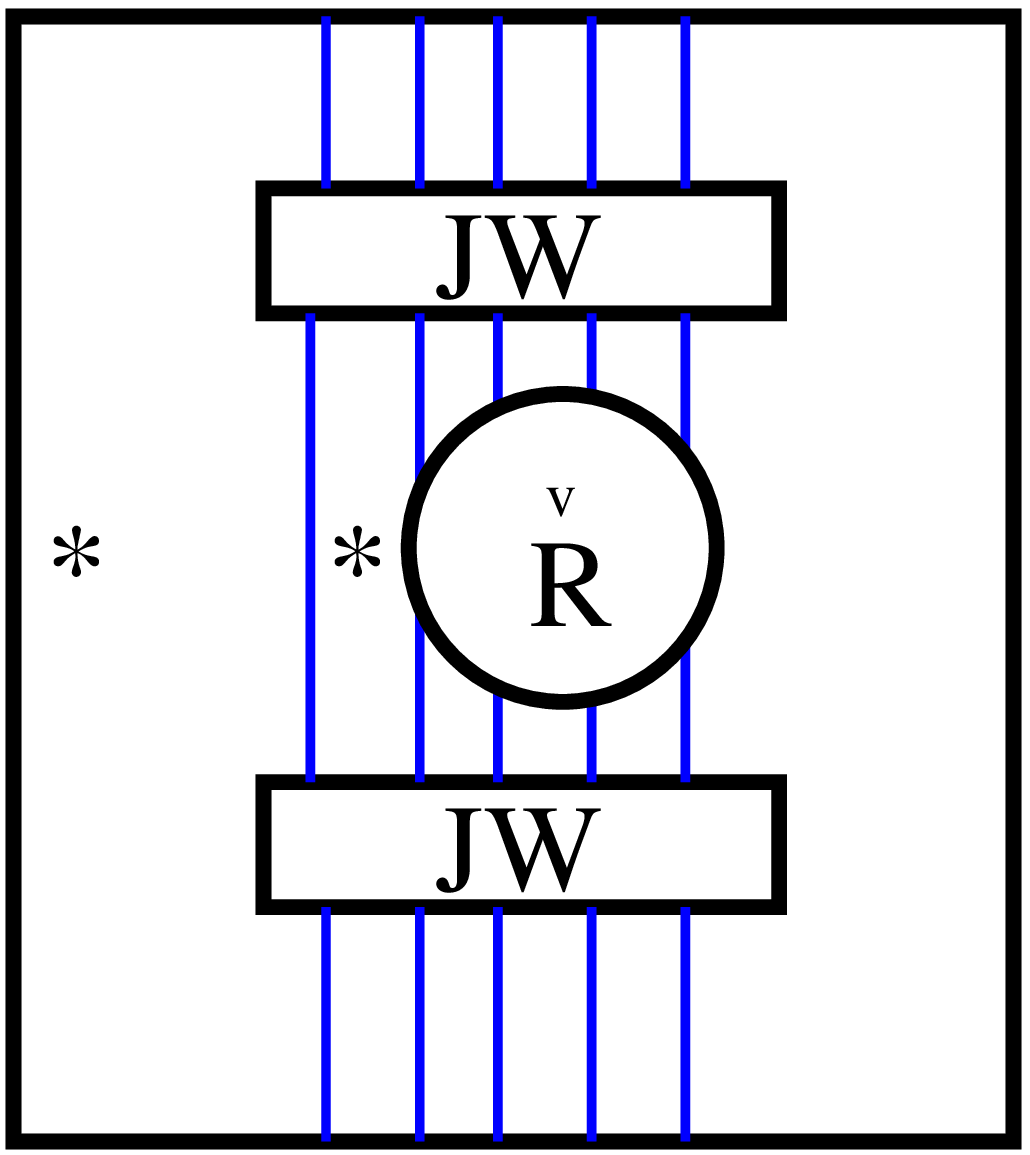} {1in} \Big\} $\\
\vspace {3pt}
\vskip 3pt
\noindent(\romannumeral 2) If $X= \sum_{j=1}^n(-\sigma)^{-j}(\omega[j]+[2n+2-j])\cup_{j}R$ then\\
\vspace {3pt}
$$\hat{\cup_0R}=\frac{1}{W}\Big\{[2n+2]\cup_0R +((-\sigma)^{n+1}+(-\sigma)^{-n-1})[n+1]\cup_{n+1}R +X+X^*\Big\}$$

\vspace {3pt} \vspace {3pt}
\noindent (\romannumeral 3) If $Y= \sum_{j=1}^n(-\sigma)^{n-j-1}(\omega[n+j+1]+[n-j+1])\cup_{j}R$ then\\
\vspace {3pt}
$$\hat{\cup_{n+1}R}=\frac{1}{W}\Big\{[2n+2]\cup_{n+1}R +((-\sigma)^{n+1}+(-\sigma)^{-n-1})[n+1]\cup_{0}R +Y+Y^*\Big\}$$
\end{proposition}
For instance:
\begin{corollary}
$\displaystyle \hat{\cup_{n+1} R}=\frac{[2n+2]}{W_{2n+2,\omega}}\cup_{n+1}R$ $$ + \frac{[2n+2]}{W_{2n+2,\omega_R}}\sum_{i=-n}^{+n}
{\Big \{}(-\sigma)^{n+i-1}[n+i+1] + (-\sigma)^{n+i+1}[n-i+1]{\Big \}}\cup_{-i}R$$
\end{corollary}

\subsection{Inner products between linear and quadratic tangles.}

We now define the main ingredients of this paper, certain quadratic tangles giving elements 
 in $P_{n+1}$.
 
 \begin{definition}
 For $S,T \in \mathfrak B$ let \\
 (\romannumeral 1) $S\circ T=$ \vpic{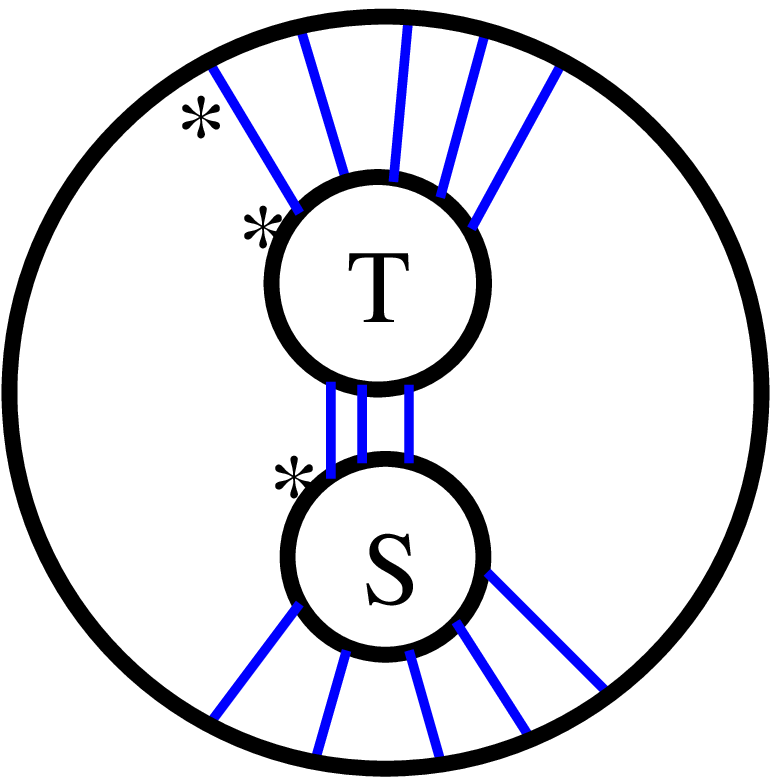} {0.8in} \\
\noindent (\romannumeral 2) $S\star T=\F(\F(S)\circ \F(T))= \sigma_S \sigma_T {\cal F} (\check S\circ \check T)=$ \vpic{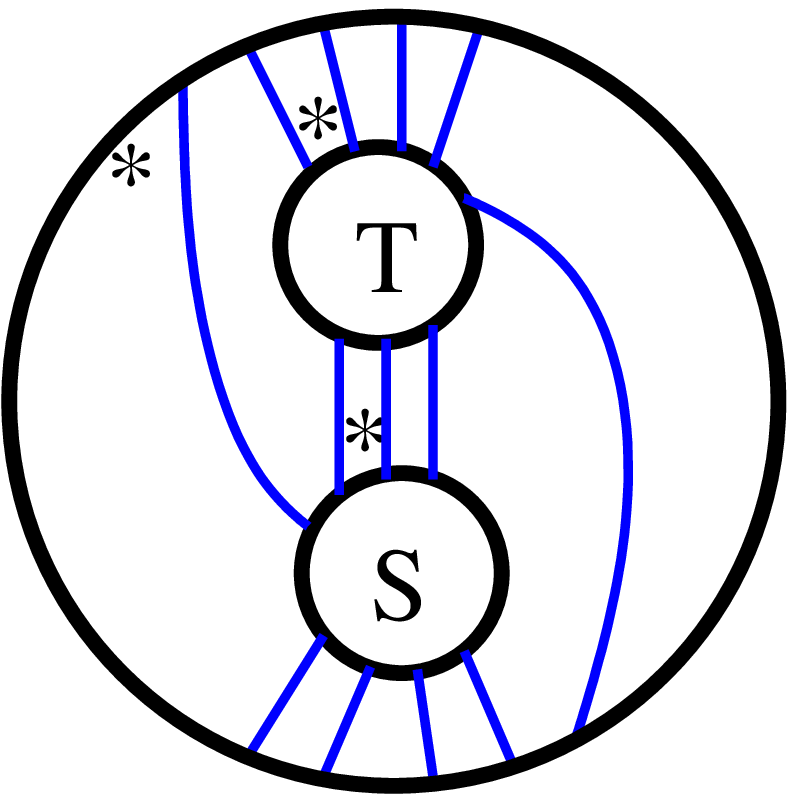} {0.8in}
 \end{definition} 

\subsection{Projection onto $\A$.}

We can  calculate $P_{\A}$, the orthogonal projection
onto $\A$, for many quadratic tangles.

\begin{proposition}\label{projannular}
$$\mbox{(\romannumeral 1)   }P_{\A}(S\circ T)=\sum_{R\in \mathfrak B} \sigma_R^nTr(RST)\hat{\cup_{n+1}R} +\sigma_T^{-1}\sigma_S Tr(\check R \check S \check T)\hat{ \cup_0R}$$
$$\mbox{(\romannumeral 2)   }P_{\A}(S\star T)=\sigma_S\sigma_T\sum_{R\in \mathfrak B} \sigma_R^nTr(\check R\check S\check T)\hat{\cup_{n+2}R} +\sigma_T^{-1}\sigma_S Tr(RST)\hat{ \cup_1R}$$
\end{proposition}
\begin{proof} (\romannumeral 1) Write $P_{\A}(S\circ T) $ as a linear combination of the $\hat \cup_i R$. The coefficient of  $\hat \cup_i R$ is just 
$\overline{\langle S\circ T, \cup_iR\rangle}$ which can only be non-zero if $i=0$ or $i=n+1$.
These two  cases are given by the following diagrams (for $n=4$):\\

\vpic{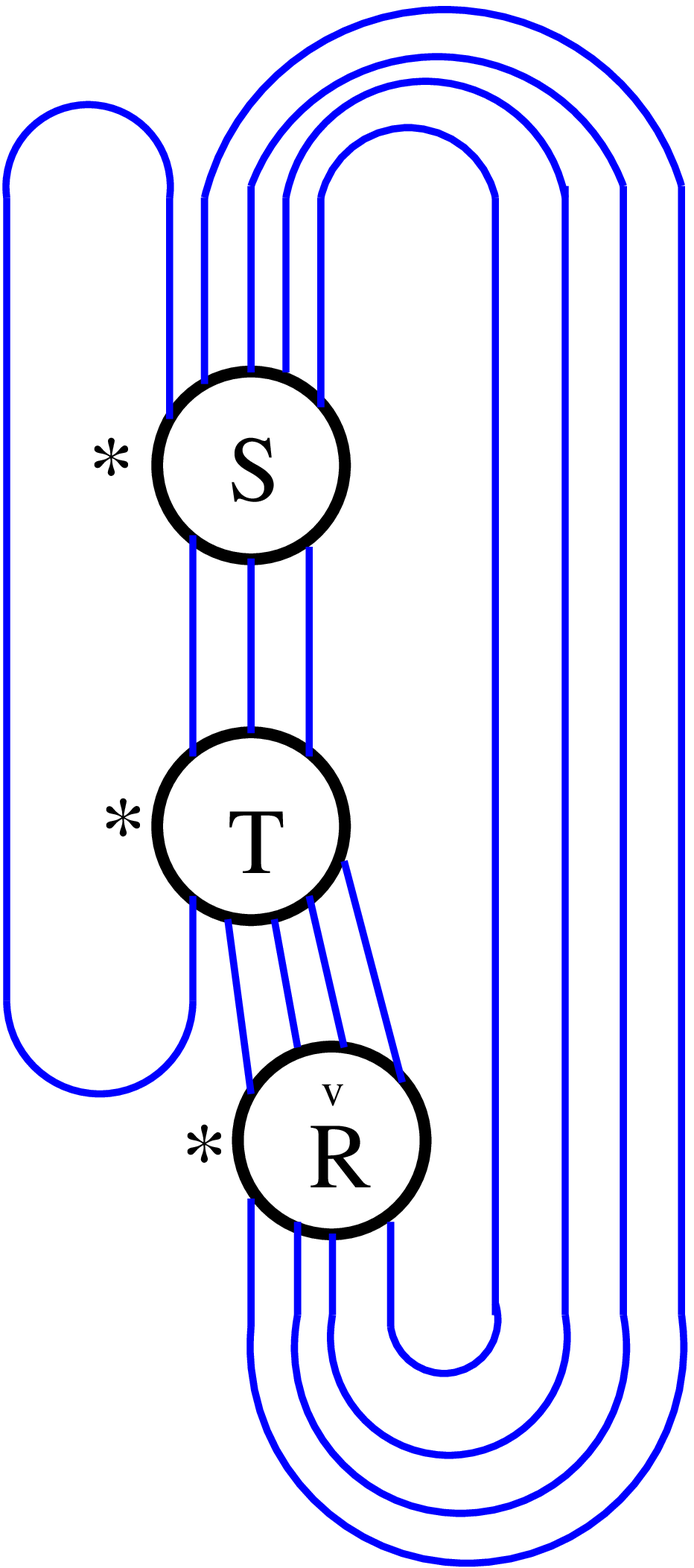} {0.8in} $ =\sigma_S^{-1}\sigma_T Tr(\check R \check T \check S) \hspace {12pt}
\vpic{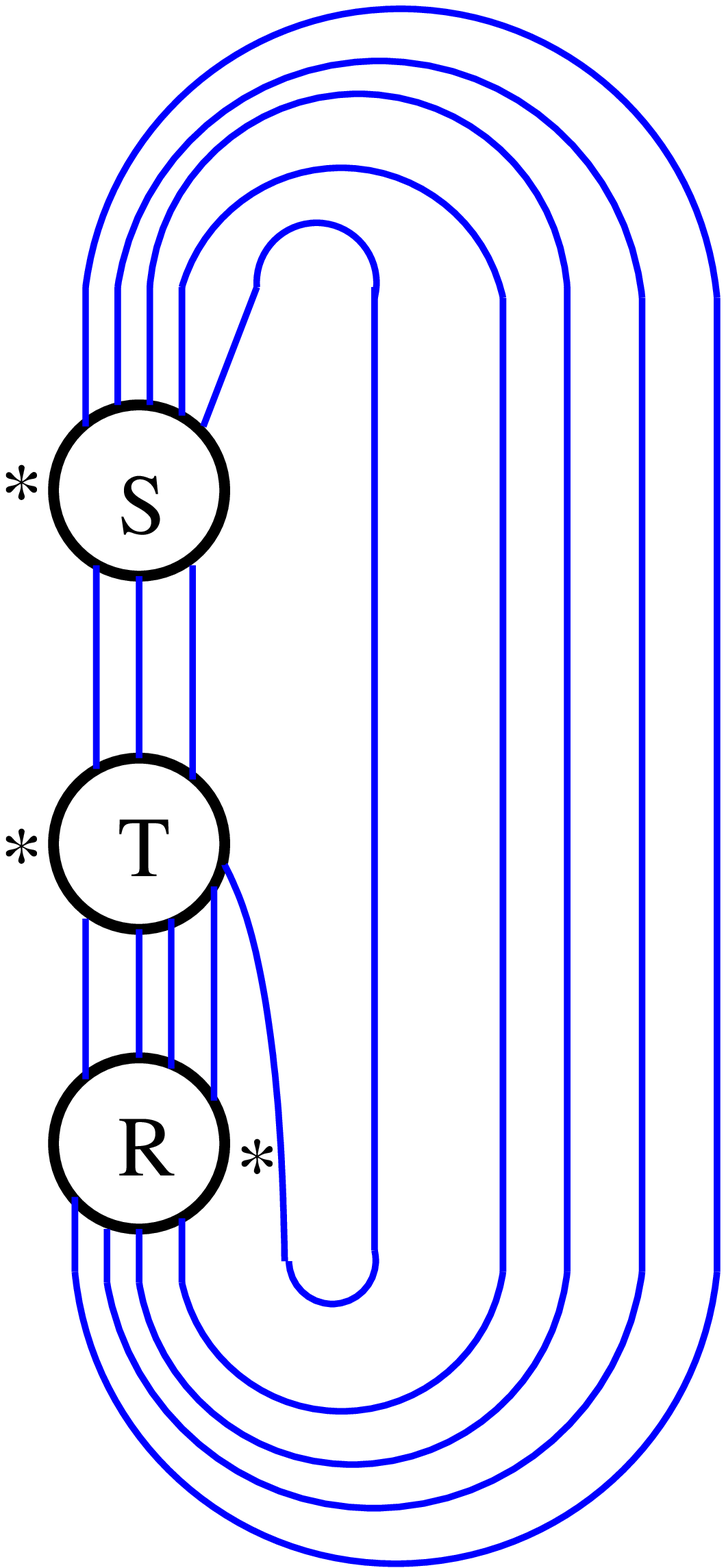} {0.8in}  =\sigma_R^nTr(RTS)$

(\romannumeral 2) follows from (\romannumeral 1) applied to the dual planar algebra, $S\star T= \sigma_S\sigma_T\F(\check S\circ \check T)$,  the fact that $\cal F$
interwines projection onto $\mathfrak A$ and the fact that $\sigma_{\hat R}=\sigma_R$ if we choose the basis $\check R$ in
the dual.

\end{proof}

\comment{We record the formula for $P_{\A}(S\circ T)$ in terms of the basis $\{\cup_iR\}$:
\begin{corollary}
$$P_{\A}(S\circ T)=$$
\end{corollary}
}
\comment{

\begin {proposition} Using the $a_R,b_R$ notation from the previous subsection,\\
(a) For $n=2k+1$, 
\begin{align*}
P_{\A}(S\circ T)=&\sum_{R\in \mathfrak B} (\omega_R\omega_S \overline{b_R})\hat{ \cup_{-1}R}+(\omega_R^{-k} \overline{a_R})  \hat{\cup_kR} \cr
P_{\A}(S\star T)=&\sum_{R\in \mathfrak B} (\omega_S \overline{a_R}) \hat{\uplus_{0}R}+(\omega_S\omega_T\omega_R^{-k} \overline{b_R}) \hat{ \uplus_{k+1}R}
\end{align*}
(b) For $n=2k$,
\begin{align*}
P_{\A}(S\circ T)=&\sum_{R\in \mathfrak B}  (\omega_R\omega_S \overline{b_R}) \hat{\cup_{-1}R}+(\omega_R^{-k} \overline{a_R}) \hat{\uplus_kR}\cr
P_{\A}(S\star T)=&\sum_{R\in \mathfrak B} (\omega_S \overline{a_R}) \hat{\uplus_{0}R}+(\omega_S\omega_T\omega_R^{k+1} \overline{b_R})  \hat{\cup_{k}R}
\end{align*}
\end{proposition}
 \begin{proof} If $\{v\}$ is a basis of a subspace of a Hilbert space and $\{\hat v\}$ is the dual basis then
 $$w\mapsto \sum_{v} \langle v,w\rangle \hat v$$ is the orthogonal projection onto that subspace. And in our
 situation all but two of the basis vectors have zero inner product with the relevant quadratic tangles.
 \end{proof} 
 
 endcomment}
 We can now give a kind of "master formula" for inner products after projection onto $\A$.
 
 \begin{proposition} \label{master} Fix $P,Q, S,T\in \mathfrak B$ and for each $R\in \mathfrak B$ let \\
$$a_R^{ST}=Tr(RST),   b_R^{ST}= Tr({\check R}{\check S}{\check T})$$  

Then putting $W_{2n+2,\omega_R}=W_R$,  for $0\leq 2j\leq n$,
\vskip 20pt
 $  \quad \langle P_{\A}(S\circ T),\rho^j(P\circ Q)\rangle=$\\

$\displaystyle  \sum_R \frac{\omega_R^j}{W_{R}} \bigg\{ \big(\overline{a_R^{ST}}{a_R^{PQ}}+
\sigma_T{\overline \sigma_S}{\overline \sigma_Q}\sigma_P\overline{b_R^{ST}}{b_R^{PQ}}\big)\hspace{3pt}
\big(\omega_R^{-1}[2j]+[2(n-j)+2]\big)$
$$\hbox{             } \qquad +(-1)^{n+1}\sigma_R\big({\overline \sigma_Q}\sigma_P \overline {a_R^{ST}}{b_R^{PQ}}
+\sigma_T{\overline \sigma_S}\overline{b_R^{ST}}{a_R^{PQ}}\big)\hspace{3pt}
\big(\omega_R^{-1} [n+2j+1]+[n-2j+1]\big)
\bigg \}$$

and  for $0\leq 2j<n$
$  \quad \langle P_{\A}(S\circ T),\rho^j(P\star Q)\rangle=$\\
$\displaystyle  \sigma_P\sigma_Q\sum_R \frac{-\omega_R^j}{W_{R}} \bigg\{   \overline{\sigma_R}\big(\sigma_T{\overline \sigma_S}{\overline \sigma_Q}\sigma_P
\overline{b_R^{ST}}{a_R^{PQ}}+ \overline{a_R^{ST}}{b_R^{PQ}}\big)\hspace{3pt}
\big(\omega_R[2(n-j)+1]+[2j+1]\big)$
$$\hbox{             }\qquad +(-1)^{n+1}\big(\overline{\sigma_P }\sigma_Q\overline{a_R^{ST}}{a_R^{PQ}}+\sigma_S{\overline {\sigma_T }}\overline{b_R^{ST}}{b_R^{PQ}}\big)\hspace{3pt}
\big(\omega_R[n-2j]+[n+2j+2]\big)
\bigg \}$$
 \end{proposition}
 \begin{proof} 
 Note first that if $\{w\}$ is a basis for a finite dimensional Hilbert space with dual basis $\{\hat w\}$ then
 $\langle \hat u,\hat v\rangle$ is just the coefficient of $u$ in the expression of $\hat v$ in the basis $\{w\}$.
 
 Let us show how the formulae are applied to obtain one of the coefficients. We want to calculate
 $$\langle \sum_{R\in \mathfrak B} \sigma_R^na_R^{ST}\hat{\cup_{n+1}R} +\sigma_T^{-1}\sigma_S b_R^{ST}\hat{ \cup_0R},
 \sum_{R\in \mathfrak B} \sigma_R^na_R^{PQ}\hat{\cup_{n+1+2j}R} +\sigma_Q^{-1}\sigma_P b_R^{PQ}\hat{ \cup_{2j}R}\rangle$$

 Obviously the inner products of terms with different $R$'s are zero so the contribution of the first ``cross-term" in the
 inner products is:
 
 $$\sum_{R\in \mathfrak B} \sigma_R^{-n}\sigma_Q^{-1}\sigma_P \overline{a_R^{ST}}b_R^{PQ}
 \langle \hat{\cup_{n+1}R},\hat{\cup_{2j}R}\rangle$$
 but $\displaystyle  \langle \hat{\cup_{n+1}R},\hat{\cup_{2j}R}\rangle= \langle \hat{\cup_{n+1-2j}R},\hat{\cup_{0}R}\rangle$
 and $1\leq n+1-2j\leq n+1$ so by (\romannumeral 2) of \ref{versions} it is equal to 
 $$\frac{1}{W}\big \{(-\sigma_R)^{2j-n-1}(\omega_R[n+1-2j]+[n+1+2j])\big\}=$$

 $$\sigma_R^n\frac{(-1)^{n+1}}{W_R}\big \{ \omega_R^{j}\sigma_R\big(\omega_R^{-1} [n+2j+1]+[n-2j+1]\big)\}$$
 
 and the other ``cross-term" is
 
  $$\sum_{R\in \mathfrak B} \sigma_R^{-n}\sigma_T\sigma_S^{-1} \overline{b_R^{ST}}a_R^{PQ}
 \langle \hat{\cup_0R},\hat{\cup_{n+1+2j}R}\rangle$$
 again $ \langle \hat{\cup_0R},\hat{\cup_{n+1+2j}R}\rangle=\langle \hat{\cup_{n+1-2j}R},\hat{\cup_{0}R}\rangle$
so that the cross terms contribute the right amount to the formula.
 
 The other terms are calculated in the same way.
 \end{proof}

Note that there is a nice check on these formulae. \\ 
 When $n$  is even, say $n=2k$, a picture shows that $\rho^k(P\circ Q)^*=   \omega_P^{k-1}\omega_Q^k P\star Q$
 so that $$\langle S\circ T, \rho^k(P\circ Q)\rangle =  \omega_P^{k+1}\omega_Q^k\langle S\circ T,(P\star Q)^*\rangle$$
 In general $\langle X,Y^*\rangle= \overline{\langle X^*, Y\rangle}$ and since $(S\circ T)^*=T\circ S$ we have
 $$\langle S\circ T, \rho^k(P\circ Q)\rangle =  \omega_P^{k+1}\omega_Q^k\overline{\langle T\circ S,P\star Q\rangle}$$
 
 If we calculate the left and right hand sides of this equation using the first and second parts of \ref{master} 
 we see they agree using $Tr(RPQ)=(\sigma_R\sigma_P\sigma_Q)^n Tr(RQP)=\overline{Tr(RQP)}$.
 
 \subsection{Projection onto TL}
 \begin{definition} Let $\T$ be the linear span of all TL diagrams in $P_{n+1}$. 
 \end{definition}
 It is obvious that $\A$ and $\T$ are orthogonal. 
 To understand the following formulae note that $TL$ has a meaning as an unshaded planar algebra.
This means we can interpret ${\cal F}{x}$ as an element in $TL_k$ for $x\in TL_k$ simply by reversing 
the shadings. We will use this convention frequently below.

 \begin{proposition} \label{tlproj}
 \begin{align*} 
 P_{\T} (S\circ T)&=\bigg\{\begin{array} {cc}{\displaystyle \frac{p_{n+1}}{[n+2]}} & \mbox {    if    } S=T\cr
                                  \vspace{3pt}                                                    0      &    \mbox{    otherwise}  \end{array} \cr
 P_{\T} (S\star T)&=\bigg\{\begin{array} {cc}{\displaystyle \frac{\omega_S \F p_{n+1} }{[n+2]}} & \mbox {    if    } S=T\cr
                                  \vspace{3pt}                                                    0      &    \mbox{    otherwise}  \end{array} \cr                                 
 P_{\T} (\F^j(S\circ T)) &= \F^j(P_{\T} (S\circ T))
\end{align*}
 \end{proposition}
 \begin{proof} $P_{\T} (S\circ T)$ is a multiple of $p_{n+1}$  by \ref{jwproperty}. Taking the trace gives
 the multiple. And $P_{\T}$ commutes with $\rho$.
 \end{proof}
 
 We see that we will need the inner products of the JW idempotents with their rotated versions.
 
 \begin{lemma}\label{tlrotip} For $m\geq i$,
$$\langle p_m ,\F^i(p_m)\rangle =(-1)^{mi}[m+1]\frac{[m-i][m-i-1]....[1]}{[m][m-1]...[i+1]}=(-1)^{i(m-i)}{\vspace{-20pt}\frac{[m+1]}{
\bigg[ \begin{array}{c} m\cr i\end{array} \bigg] }}$$
 \end{lemma}
 \begin{proof} Note the special cases $i=m, i=0$ which certainly work. So suppose $1\leq i<m$. We will use Wenzl's inductive formula for 
 $p_m$: $p_{m+1}=p_m-\frac{[m]}{[m+1]}p_mE_mp_m$. Drawing a picture for the inner product we see that
 the first term does not contribute by \ref{jwproperty}. The trace of the following picture represents $\langle p_{m+1}, p_mE_m p_m\rangle$\\
\vpic{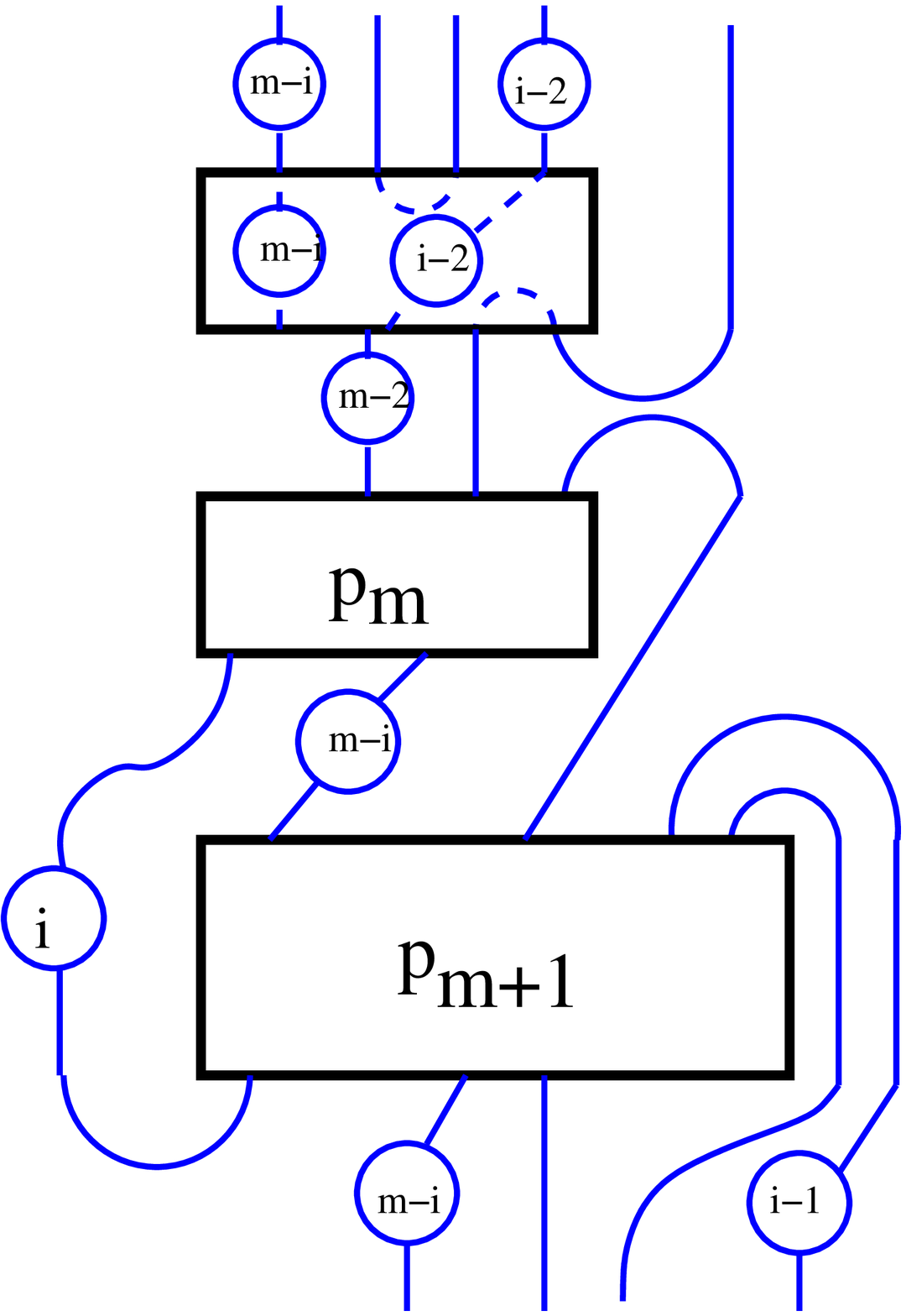} {1.8in} \\
where we have used numbers in circles to indicate multiple strings. Inside the top box is the only TL element which 
makes a non-zero contribution. Its coefficient is $(-1)^i\frac{[m-i+1]}{[m]}$ by \ref{prodeis}. Removing that box and doing some isotopy
one gets the trace of the following picture:\\
\vpic{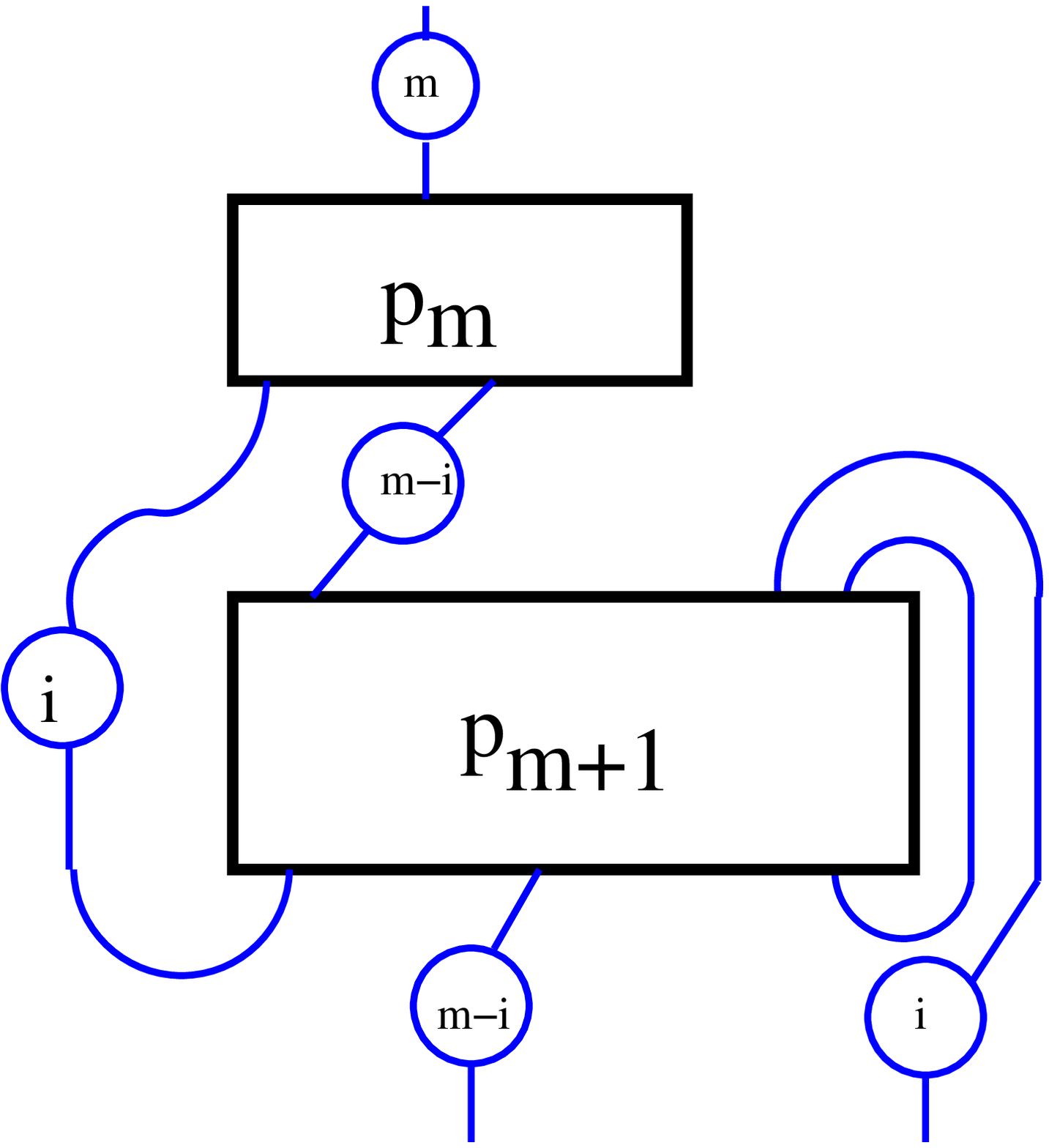} {2in} \\
By formula 3 of \ref{morejw} this is just $\frac{[m+2]}{[m+1]}\langle p_m ,\F^i(p_m)\rangle$ and we end up with
$$\langle p_{m+1} ,\F^i(p_{m+1})\rangle= (-1)^i\frac{[m-i+1][m+2]}{[m+1]^2}\langle p_m ,\F^i(p_m)\rangle$$
we may continue to apply the procedure. The last time we can do it, "$m+1$" will be $i+1$ and the equation
will be:
$$\langle p_{i+1} ,\F^i(p_{i+1})\rangle= (-1)^i\frac{[1][i+2]}{[i+1]^2}\langle p_i ,\F^i(p_i)\rangle$$
After making all the cancellations in the product and observing that $\langle p_i ,\F^i(p_i)\rangle=[i+1]$ we get
the answer.

 \end{proof}
 \vspace {30pt}

\section{Applications to subfactors.}\label{subfactors}
\subsection{Multiplicity one.}

 Let \P
be an (n-1)-supertransitive planar algebra with $n$-multiplicity one  and chirality $\omega$. 

We shall now make a careful choice of the elements $R$ (and $\check R$) that 
form the basis $\mathfrak B$ in section \ref{ipf}
 So as in \ref{combpa} choose minimal
central projections of $P_n$ with
$$e+f=p_n $$
$p_n$ being the JW projection (\ref{jwproperty}).  To be as 
precise as possible, suppose $Tr(e) \leq Tr(f)$ so that 
\begin{formula} \hspace{1.5in}$\displaystyle {r={{Tr(f)} \over {Tr(e)}} \geq 1.}$
\end{formula}

Define  $\tilde R$ orthogonal to $TL$ is
\begin{formula}\label{choicer} \hspace{1.5in}$\displaystyle {\tilde R= re-f.}$
\end{formula}
 
Since $ef=0$, algebra is easy and we obtain
\begin{formula} \hspace{1.5in}$\displaystyle {\tilde R^k=r^k e+(-1)^k f}$
\end{formula}

so that 
\begin{formula} \label{rsquared} \hspace{1.1in}
$\displaystyle \tilde R^2=(r-1)\tilde R + rp_n$
\end{formula}
\begin{formula}\label{tracer2}
 \hspace{1.1in}
$\displaystyle {Tr(\tilde R^2)=r[n+1],}$
\end{formula}

\begin{definition}
Set $\displaystyle R= {\tilde R \over {\sqrt {r[n+1]}}}$
\end{definition}
We have
$\langle R, R\rangle =1$ and\\
\begin{formula}\label{tracer3} \hspace{1.1in}
$\displaystyle {Tr( R^3)=\frac{1}{(r[n+1])^{3/2}}Tr(r^3e-f)=\frac{\sqrt r-{1\over {\sqrt r}}}{\sqrt{[n+1]}},}$
\end{formula}

Note that if $r\neq 1$ we have been able to choose $R$ without any ambiguity in
sign. See \ref{signs}

We will also need formulae involving ${\cal E}$ 
of \ref{condexp}. 

By multiplying ${\cal E}(R^2)$ and ${\cal E}(p_n)$ by
the $E_i$'s of \ref{ei} we see that they are both multiples
of $p_{n-1}$ and on taking the trace we get
\begin{formula}\qquad $\displaystyle{{\cal E}(R^2)={p_{n-1}\over [n]} }$\\

\end{formula}
so that

\begin{formula}\label{simpletet} \hspace{1.1in}
$\displaystyle {Tr({\cal E}(R^2)^2)={1\over[n]}.}$
\end{formula}

To be even-handed we need to do the same for $P_{n,-}$ so we will continue 
the convention of using a\quad $\check{}$\quad symbol to indicate the corresponding
objects for $P_{n,-}$. Thus we have $\check e$ and $\check f$ and
\begin{formula} 
\hspace{1.5in}$\displaystyle {\check r={{{Tr(\check f)} \over {Tr(\check e)}}} \geq 1,}$
\end{formula}
\noindent and all the above formulae have \quad $\check{}$ \quad versions.
At this stage we have two potential meanings for $\check R$.
To make them consistent just forces the choice of square root $\sigma$ of $\omega$
so that ${\cal F}(R)=\sigma \check R$.

Note that if $r$ and $\check r$ are both different from $1$ (i.e. neither $Tr(R^3)$ nor $Tr({\check R}^3)$
is zero),  there is no choice of signs anywhere - they are imposed by our conventions. 
We will use $\omega^{1/2}$ instead of $\sigma$.
\5

\begin{theorem} Let \P be an $n-1$-supertransitive subfactor  planar algebra with 
multiplicity sequence $0^{n-1}10$ and chirality $\omega$.
Let $r,\check r$ be as above.
We may suppose $r\leq \check r$ (by
passing to the dual if necessary).
Then $n$ is even,
$$ \check r =\frac{[n+2]}{[n]}$$ 
and 
$$ r+\frac{1}{r}=2+\frac{2+\omega +\omega^{-1}}{[n][n+2]}.$$
If $\omega=-1$, $n$ is divisible by $4$.
\end{theorem}

\begin{proof} Since the $(n+1)$-multiplicity is zero, $S\circ T$, $S\star T$ and their rotations
are in the linear span of $\A$ and $TL$ so we can calculate  inner products between them using the
formulae of the previous section. We record the ones we
will use (with $W=W_{2n+2,\omega}$):
\5
\5
Let $\displaystyle \alpha = Tr(R^3) \mbox{ and }Tr(\check R ^3)=\beta$.

$$\langle P_{\A}(R\circ R), R\circ R\rangle={1\over W}((\alpha^2+\beta^2)[2n+2]+(-1)^{n+1}2\alpha\beta(\omega^{1/2}+\omega^{-1/2})[n+1])$$
$\displaystyle \langle P_{\A}(R\circ R), R\star R\rangle=$  $${1\over W}\big\{(-1)^n\omega(\alpha^2+\beta^2)(\omega[n]+[n+2]) 
-2\alpha\beta\omega^{1/2}(\omega[2n+1]+1)\big\}$$

These are just  \ref{master} with $j=0$ and $P=Q=S=T=R$. 

When $n$ is odd the second formula is just \ref{master} with $j=0$. If $n$ is even one may obtain it from the first formula of
\ref{master} with $j=k$ by noting that $R\star R=\omega \rho^k(R\circ R)^*$.

We have, from \ref{tlrotip} and \ref{tlproj} that $$\langle P_{\T}(R\circ R),R\circ R\rangle ={1\over {[n+2]}}$$  and
$$\langle P_{\T}(R\circ R),R\star R\rangle ={(-1)^n\omega \over {[n+2][n+1]}}$$
But we can calculate $\langle R\circ R,R\circ R\rangle$ and $\langle R\circ R,R\star R\rangle$ directly from their
pictures and the algebraic relations satisfied by $R$.
The first is trivial-$$\langle R\circ R,R\circ R\rangle=Tr({\cal E}(R^2)^2)={1\over [n]}$$  \ref{simpletet}.
This implies the same value for $\langle R\star R,R\star R\rangle$.
For $\langle R\circ R,R\star R\rangle$, a little isotopy produces the trace of the following picture:\\
\vspace {5pt}
\vpic{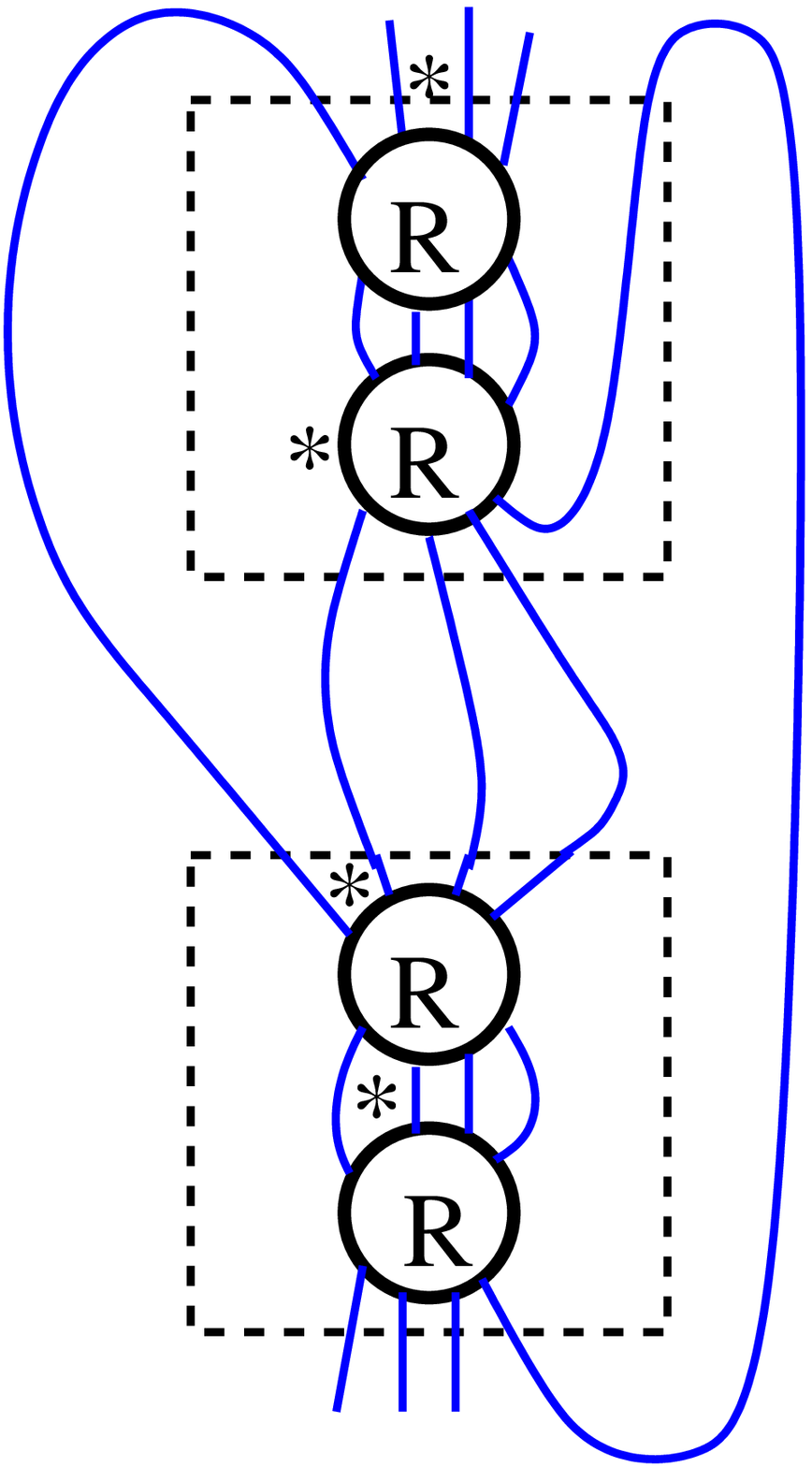} {2in} \\
which we recognise as $\omega \langle \F(R^2),\F(R)^2\rangle =\omega^2\langle \F(R^2),(\check R)^2\rangle$.
And $\displaystyle R^2=\alpha R+{p_n\over [n+1]}$, $\displaystyle (\check R)^2=\beta \check R+{p_n\over [n+1]}$ so
we get $\displaystyle \omega^{-1/2}\alpha\beta +\frac{\langle\F (p_n), p_n\rangle}{[n+1]^2}$ so finally
$$\langle R\circ R,R\star R\rangle=\omega^{3/2}\alpha\beta +\frac{(-1)^{n-1}\omega^2}{[n+1][n]}$$
Now we may use zero $n+1$-multiplicity $=0$ , which implies\\
 $\displaystyle \langle R\circ R,R\circ R\rangle =\langle P_{\A}(R\circ R),R\circ R\rangle+\langle P_{\T}(R\circ R),R\circ R\rangle$ and\\
 $\displaystyle \langle R\circ R,R\star R\rangle =\langle P_{\A}(R\circ R),R\star R\rangle+\langle P_{\T}(R\circ R),R\star R\rangle$ to obtain

$$ {1\over [n]}={1\over W}((\alpha^2+\beta^2)[2n+2]+(-1)^{n+1}2\alpha\beta(\omega^{1/2}+\omega^{-1/2})[n+1])+{1\over {[n+2]}} $$

and $\displaystyle \omega^{3/2}\alpha\beta +\frac{(-1)^{n-1}\omega^2}{[n+1][n]}=$
 $${1\over W}\big\{(-1)^n\omega(\alpha^2+\beta^2)(\omega[n]+[n+2]) -2\alpha\beta\omega^{1/2}(\omega[2n+1]+1)\big\}+
 {(-1)^n\omega \over {[n+2][n+1]}}$$
 
 Using $\displaystyle {1\over [n]}-{1\over [n+2]}=\frac{[2n+2]}{[n][n+1][n+2]}$   the first equation  becomes
 \begin{formula}\label{firstwog}
 $$\frac{W}{[n][n+1][n+2]}=(-1)^{n+1}\frac{2\alpha\beta(\omega^{1/2}+\omega^{-1/2})}{q^{n+1}+q^{-n-1}}+\alpha^2+\beta^2$$
 \end{formula} 
 and after a little work the second equation becomes:
 \begin{formula}\label{second}
$$ \frac{W([n]+[n+2]\omega)}{[n][n+1][n+2]}=(-1)^n\alpha\beta \omega^{1/2}([2n+2]\delta+\omega^{-1}-\omega) -(\alpha^2+\beta^2)
([n]\omega + [n+2])$$
\end{formula}
 
 Thus we have two linear equations for $\alpha^2 +\beta^2$ and $\alpha \beta$. It is easy to check that they
 are satisfied by
 $$\alpha\beta=\frac{(-1)^n(\omega^{1/2}+\omega^{-1/2})(q^{n+1}+q^{-n-1})}{[n][n+1][n+2]}$$
 $$\alpha^2+\beta^2=\frac{(q^{n+1}+q^{-n-1})^2+(\omega^{1/2}+\omega^{-1/2})^2}{[n][n+1][n+2]}$$
 
 Now if $n$ is odd, there is a trace-preserving isomorphism between $P_{n,+}$ and $P_{n,-}$
given by a suitable power of the Fourier transform. Thus $\alpha=\beta$ which
with these equations gives 
$\displaystyle q^{n+1}+q^{-n-1}+(-1)^{n+1}(\omega^\frac{1}{2} +\omega^{-\frac{1}{2}})=0$
which is impossible if $\delta >2$. So $n$ is even.
\5
It then follows that $$(\alpha \pm \beta)^2=\frac{(q^{n+1}+q^{-n-1} \pm
(\omega^\frac{1}{2} +\omega^{-\frac{1}{2}}))^2}{[n][n+1][n+2]}$$
so that if $\alpha \leq \beta$ (which is the same as $r \leq \check r$),
we get $$\beta=\frac{q^{n+1}+q^{-n-1}}{\sqrt{[n][n+1][n+2]}}\quad\mbox{  and} 
\quad\alpha=\frac{\omega^\frac{1}{2} +\omega^{-\frac{1}{2}}}{\sqrt{[n][n+1][n+2]}}.$$
These immediately imply the desired formulae.

To see that $n$ is divisible by $4$ when $\omega=-1$, observe that \ref{firstwog} implies that 
at least one of $\alpha$ and $\beta$ is non-zero, wolog suppose it's $\alpha$. Then $r\neq 1$ so
the traces of $e$ and $f$ are different. But $\rho^{n/2}$ acts as a trace-preserving automorphism
on the linear span of $e$ and $f$ and so must be the identity. This forces $n/2$ to be even.

\end{proof}

It is more than a little sensible to check these formulae in examples where all the
parameters are known. In the Fuss Catalan example with $a=\sqrt 2$ we have, from \ref{fusscatalan}
that $$n=2, \delta=q+q^{-1}=b\sqrt2, \check r= 2(b^2-1), r=\frac{b^2}{b^2-1} \mbox{ and } \omega=1$$
and the formulae of the theorem are rapidly verified.
For the partition planar algebra $$n=4, \check r=\frac{\delta^4-4\delta^2+3}{\delta^2-2},
r=\frac{\delta^4-3\delta^2+2}{\delta^4-3\delta^2} \mbox{ and }\omega=1.$$
\begin{corollary} The principal graph (corresponding to $\check r$ ) is 
as below for vertices of distance $\leq n+2$ from *: 
\5
\vpic{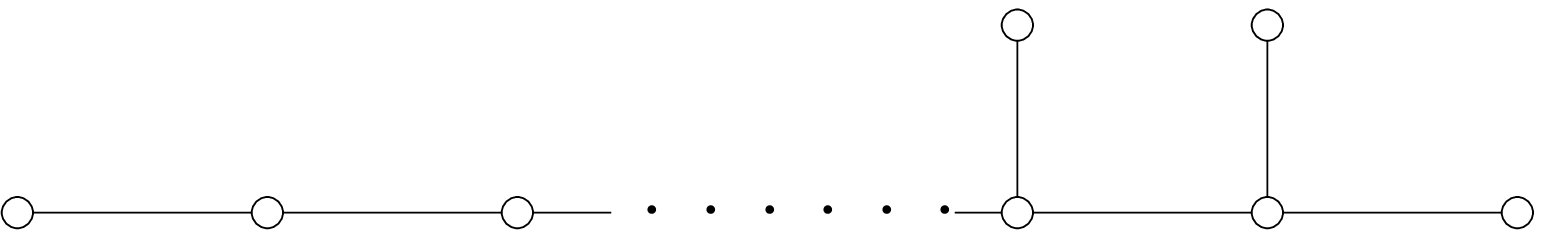} {3in}

\5
The other principal
graph is as below:
\5
\vpic{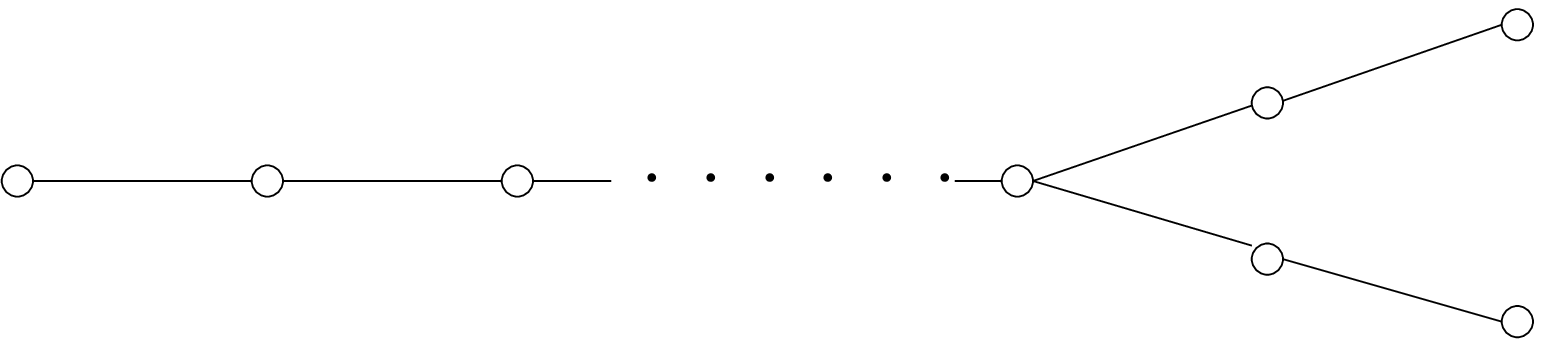} {3in}
 
\5
\end{corollary}

\begin{proof} By the dimension constraints the principal graph for $\check r$ must begin
as one of the two pictures.  But the equation $$ \check r =\frac{[n+2]}{[n]}$$  is just the
Perron Frobenius eigenvector equation for the first graph (at the top univalent vertex). So
the eigenvalue equation of the second graph cannot satisfy that equation. 
\end{proof}
\begin{corollary}
The Haagerup planar algebra has chirality $-1$, but the Haagerup-Asaeda 
planar algebra has chirality $1$. 
\end{corollary}
\begin{proof}
As soon as one of the principal graphs has a symmetry exchanging the
two vertices at distance $n$ from *, one has that $ r =1$. This is
the case for the Haagerup subfactor. The Asaeda-Haagerup value can be
deduced from the Perron Frobenius data.
\end{proof}

\comment
{
\begin{theorem} If a subfactor has even supertransitivity, multiplicity one and zero FAD then the rotational eigenvalue $\omega_R$ is real.
\end{theorem}
\begin{proof}
We apply the master formula to calculate $\langle P_{\A}( R\circ R), \rho(R\circ R)\rangle$. We obtain 
$$\frac{1}{W}\bigg \{(\alpha^2+\beta^2)([2]+\omega[2n])-2\alpha\beta(\omega^{1/2}[n+3]+\omega^{3/2}[n-1])\bigg \}.$$
The contribution of the TL part to $\langle R\circ R, \rho(R\circ R)\rangle$ is $ \displaystyle \frac{[2]}{[n][n+1][n+2]}$
by \ref{tlproj} and \ref{tlrotip}. So altogether we obtain that $\langle P_{\A}( R\circ R), \rho(R\circ R)\rangle$ is a real
multiple of 
$$X=(\alpha^2+\beta^2)([2]+\omega[2n])-2\alpha\beta(\omega^{1/2}[n+3]+\omega^{3/2}[n-1])+\frac{[2]W}{[n][n+1][n+2]}.$$
But we see from the following picture of $\langle R\circ R, \rho(R\circ R)\rangle$ that its complex conjugate is itself
times $\omega^{-2}$:\\
\vspace{20pt} \hspace{2in}\hpic{tetra} {2in}
So we must have $X+\omega^2 \overline X=0$.

 Letting $x=(\alpha^2+\beta^2)[n][n+1][n+2]$ and $y=\alpha\beta [n][n+1][n+2]$,
this means
\end{proof}
}
\subsection{Multiplicity two.} 
Begin by choosing an orthonormal basis $\mathfrak B = \{S,T\}$ of self-adjoint rotational eigenvectors
for the orthogonal complement of $TL_n$ in $P_n$, with eigenvalues $\omega_S$ and
$\omega_T$. And the corresponding $\sigma_S,\sigma_T,\check S$ and $\check T$.

Recall the notation $a_R^{PQ}=Tr(RPQ)$. The ideal $\mathfrak I$ in
$P_n$ which is the orthogonal complement of the TL elements with less than $n$ through
strings is orthogonally spanned by $S,T$ and $p_n$. It is a commutative $C^*$-algebra 
whose identity is $p_n$. By taking inner products the multiplication law for $\mathfrak I$ is:
\begin{align*}
S^2=&a_S^{SS}S+a_S^{ST}T+\frac{p_n}{[n+1]}\cr
T^2=&a_S^{TT}S+a_T^{TT}T+\frac{p_n}{[n+1]}\cr
ST=&a_S^{ST}S+a_S^{TT}T
\end{align*}

\begin{proposition} {Associativity constraint.}
$$(a_S^{ST})^2+(a_S^{TT})^2=a_S^{SS}a_S^{TT}+a_T^{TT}a_S^{ST}+\frac{1}{[n+1]}$$

\end{proposition}
\begin{proof} Write $x=a_S^{SS}, y=a_T^TT$ and $u=a_S^{ST}, v=a_S^{TT}$. Evaluate $S^2T$ in two ways:\\
(a) $(S^2)T= (xS+uT+\frac{p_n}{[n+1]})T=(xu+uv)S+(xv+uy+\frac{1}{[n+1]})T+u\frac{p_n}{[n+1]}$\\
(b)$S(ST)=S(uS+vT)=(ux+vu)S+(u^2+v^2)T+u\frac{p_n}{[n+1]}$.
The conclusion follows immediately.
\end{proof}
\begin{theorem} \label{evenst}For $k\in \mathbb Z, k>0$, there is no 2k-supertransitive subfactor  
with multiplicity sequence beginning  *20. 

\end{theorem} 
\begin{proof}
We will actually show that any such subfactor has to have index $\geq 4.5$ which is enough by
Haagerup's classification (\cite{H5}), or as Snyder has pointed out, the smallest graph which could
have multiplicity sequence beginning $*20$ has norm-squared equal to $\frac{5+\sqrt{17}}{2}$ which is bigger than $4.5$.

Let $n=2k+1$. The rotation ${\cal F}^n$ gives a trace preserving antiisomorphism between 
$P_{n,+}$ and $P_{n,-}$. And a picture shows immediately that if $R,P,Q\in \{S,T\},$
$$Tr(\check R \check P \check Q)=(\sigma_R\sigma_P\sigma_Q)^{n}Tr(RQP).$$Moreover $S$ and $T$ all 
commute since the multiplicity is only $2$ and $S$ and $T$ are commuting and self-adjoint which
 shows that $Tr(RPQ),Tr(\check R \check P \check Q) \in \mathbb R$ so 
$Tr(RPQ)=\pm Tr(\check R \check P \check Q)$. 
We specialise the master formula \ref{master} to the case $P=T, Q=S, j=0$ which
gives us $\langle P_{\A}(S\circ T),T\circ S\rangle$. By \ref{tlproj} and the multiplicity 
sequence beginning $*20$, this is the same
as $\langle S\circ T,T\circ S\rangle$ which is clearly zero.
 We are summing over the two values $S$ and $T$ of $R$. Note
that, using the conventions of \ref{master} we have

\begin{align*}
a_S^{ST}=a_S^{PQ}=Tr(S^2T), \hbox{  call it }a_S, & \quad b_S^{ST}=b_S^{PQ}=\pm a_S\cr
a_T^{ST}=a_T^{PQ}=Tr(ST^2),  \hbox{  call it }a_T, & \quad  b_T^{ST}=b_T^{PQ}=\pm a_T
\end{align*}
Since $S$ and $T$ are self-adjoint and commute, $a_S$ and $a_T$ are real
 and we obtain 
$$0= \frac{a_S^2}{W_S}\bigg( [2n+2](1+\omega_S^{-1}\omega_T)+\alpha[n+1]\bigg) $$
$$\mbox{      \qquad         }+\frac{a_T^2}{W_T}\bigg( [2n+2](1+\omega_S^{-1}\omega_T)+\beta
[n+1]\bigg)$$
Where $\alpha$ and $\beta$ are sums of four roots of unity.
Multiplying through by a square root of $\omega_S\omega_T^{-1}$ which is also an $n$th. root of unity we get
$$0=(2\cos\frac{2r\pi}{n}[2n+2]+\alpha'[n+1]) \frac{a_S^2}{W_S} +(2\cos\frac{2r\pi}{n}[2n+2]+\beta'[n+1])
\frac{a_T^2}{W_T}$$ for some integer 
$r$, where $|\alpha'|$ and $|\beta'|$ are $\leq 4$.
We want to show that the $[2n+2]$ term dominates so that the
real parts of the coefficients of $a_S^2$ and $a_T^2$ get their sign from the
$2\cos\frac{2r\pi}{n}[2n+2]$. For this it clearly suffices to show that $$|2\cos\frac{2r\pi}{n}[2n+2]|>4[n+1].$$
We again use that $n$ is odd-the smallest $|2\cos\frac{2r\pi}{n}|$ can be is $2\sin\pi/{2n}$
and the inequality becomes:$$[2n+2]>\frac{2[n+1]}{\sin\pi/2n}.$$
Writing $s$ for $\frac{1}{\sin\pi/2n}$ and $Q$ for $q^{n+1}$ this becomes $Q^2-Q^{-2} >2s(Q-Q^{-1})$. Completing the square
and using $Q>1$ we see that this is implied by $Q>2s$. But $(\pi/2n) s$ decreases to $1$ as $n$ increases 
so $2s<\frac{2\pi/6}{\sin \pi/6}\frac{2n}{\pi}  $.
Hence it suffices to prove $q^{n+1}\geq \frac{4n}{3}$. Consider the graphs of $q^{x+1}$ and $(4/3)x$. Clearly the $x$ coordinate of
the largest point of intersection is decreasing as a function of $q$. So if we can show that 
this value of $x$ is $3$ for $q=\sqrt 2$ then 
for all greater $q$ and all $n\geq 3$ we will have $q^{n+1}> \frac{4n}{3}$.
 But $(\sqrt 2)^{3+1}= 4$ and by calculus $q^{x+1} - \frac{4x}{3}$
is (just) increasing for $x=3$. 

The value $q=\sqrt 2$ corresponds to index $4.5$.
Since $q>1$ the $W$ factors are positive so we can conclude that, for index $\geq 4.5$,  $a_S=a_T=0$. This contradicts the associativity constraint.
\end{proof}

Notes. \\
(\romannumeral 1) There is nothing special about   $4.5$, it was chosen simply because it suffices and the
estimates are convenient.

\noindent (\romannumeral 2) The even supertransitivity assumption is necessary since the
GHJ subfactor of index $3+\sqrt 3$ (\cite{GHJ},\cite{EK7}) has multiplicity sequence beginning $*20$.

We end with an intriguing observation about chirality in the $*20$ case.

\begin{theorem} Let $N\subset M$ be a subfactor with multiplicity sequence beginning $*20$,
and let $S,T, \omega_S$ and $\omega_T$ be as above. Then $\omega_S\neq \omega_T$.
\end{theorem}
\begin{proof} Consider (\romannumeral 1) of \ref{projannular}. As we have observed, $S$ and
$T$ commute and have zero projection onto TL. Thus if $\sigma_S^{-1}\sigma_T=\sigma_T^{-1}\sigma_S$
we have $S\circ T=T\circ S$ which is impossible since $S\circ T$ and $T\circ S$ are orthogonal
and non-zero.
\end{proof}

\end{document}